%% file: modifiedLandweber.tex
\pdfoutput=1
\documentclass{shinyart}

\usepackage[utf8]{inputenc}
\usepackage{enumitem} 

\usepackage{shinybib}

\usepackage{booktabs}
\usepackage[leftcaption]{sidecap}
\usepackage[exponent-product=\cdot]{siunitx}

\usepackage{biblatex}
\newlist{assumption}{enumerate}{1}
\setlist[assumption]{label=(\textsc{a}\arabic*)}

\crefname{assumptioni}{Assumption}{Assumptions}

\newcommand{\N}{\mathbb{N}}
\newcommand{\R}{\mathbb{R}}
\newcommand{\1}{\mathbb{1}}
\newcommand{\Linop}{\mathbb{L}}

\DeclareMathOperator{\diag}{\mathrm{diag}}
\newcommand{\norm}[1]{\|#1\|}

\def\wkto{\rightharpoonup}
\renewcommand{\epsilon}{\varepsilon}

\usepackage{graphicx}
\usepackage{pgfplots}
\pgfplotsset{compat=newest}
\pgfplotsset{plot coordinates/math parser=false}

\usepackage[textsize=footnotesize]{todonotes}

\title{Bouligand--Landweber iteration for a non-smooth ill-posed problem}
\author{Christian Clason\thanks{Faculty of Mathematics, University of Duisburg-Essen, Thea-Leymann-Strasse 9, 45127 Essen, Germany\newline%
    (\email{christian.clason@uni-due.de}, \email{huu.vu@uni-due.de})}
    \and Vu Huu Nhu\footnotemark[1]
}

\begin{document}
\maketitle

\begin{abstract}
    This work is concerned with the iterative regularization of a non-smooth nonlinear ill-posed problem where the forward mapping is merely directionally but not Gâteaux differentiable. Using a Bouligand subderivative of the forward mapping, a modified Landweber method can be applied; however, the standard analysis is not applicable since the Bouligand subderivative mapping is not continuous unless the forward mapping is Gâteaux differentiable. We therefore provide a novel convergence analysis of the modified Landweber method that is based on the concept of \emph{asymptotic stability} and merely requires a generalized tangential cone condition. These conditions are verified for an inverse source problem for an elliptic PDE with a non-smooth Lipschitz continuous nonlinearity, showing that the corresponding \emph{Bouligand--Landweber iteration} converges strongly for exact data as well as in the limit of vanishing data if the iteration is stopped according to the discrepancy principle. This is illustrated with a numerical example.

    %
    %
\end{abstract}

\section{Introduction}

We consider the (iterative) regularization of inverse problems $F(u) = y$ for a nonlinear parameter-to-state mapping $F:U\to Y$ between two Hilbert spaces $U$ and $Y$ that is compact and directionally but not Gâteaux differentiable. Specifically, we are interested in mappings arising as the solution operator to nonlinear partial differential equations with piecewise continuously differentiable nonlinearities. To fix ideas, let $\Omega$ be an open bounded subset of $\R^d, d \in \{2,3\}$, with a Lipschitz boundary $\partial\Omega$, and consider the non-smooth semilinear equation
\begin{equation} \label{eq:maxpde-intro}
    - \Delta y + y^+ = u \quad \text{in} \ \Omega, \quad y =0 \ \text{on} \ \partial\Omega
\end{equation} 
with $u \in L^2(\Omega)$ and $y^+(x) := \max(y(x),0)$ for almost every $x \in \Omega$; see \cite{Christof2017}. This equation models the deflection of a stretched thin membrane partially covered by water (see \cite{Kikuchi1984}); a similar equation arises in free boundary problems for a confined plasma; see, e.g., \cite{Temam:1975,Rappaz:1984,Kikuchi1984}. More complicated but related models (where the nonlinearity enters into higher-order terms) can be used to describe problems with sharp phase transitions such as the weak formulation of the two-phase Stefan problem \cite{Oleinik,Visintin96}.

Our goal is to estimate the source term $u$ in such models from noisy measurements $y^\delta$ of the state.
For the sake of presentation, in this work we will focus on \eqref{eq:maxpde-intro}, although our results also apply to similar equations with piecewise continuously differentiable nonlinearities in the potential term (cf.~\cref{app:piecewise-diff}). Since solution operators to elliptic equations are usually completely continuous, this problem is ill-posed and has to be regularized.
Here we consider iterative regularization methods of Landweber-type, which for a \emph{differentiable} forward mapping $F:U\to Y$ is given by
\begin{equation}\label{eq:landweber_smooth}
    u_{n+1}^\delta = u_{n}^\delta + w_n F'(u_n^\delta)^* \left( y^\delta - F(u_n^\delta) \right), \quad n \geq 0,
\end{equation} 
for a step size $w_n >0$ and the adjoint $F'(u)^*$ of the Fréchet derivative of $F$ at $u\in U$. For noisy data, the iteration has to be stopped at a stopping index $N=N(\delta,y^\delta)<\infty$ in order to be stable, e.g., according to the \emph{Morozov discrepancy principle} at the first index for which the residual norm $\norm{F(u_N^\delta)-y^\delta}_Y$ reaches the noise level $\delta$, where $\norm{y^\delta -y^\dag}_Y\leq \delta$ with $y^\dag = F(u^\dag)$ for some $u^\dag\in U$. 
Since the residual is calculated as part of the iteration, this principle can be evaluated cheaply in every iteration, avoiding unnecessary computational work (in contrast to, e.g., Tikhonov regularization, where in general the full solution has to be computed for a given regularization parameter before the principle can be checked).   
It is then possible to show that $u_{N}^\delta \to u^\dag$ as $\delta\to 0$, provided that a \emph{tangential cone condition} (which bounds the linearization error by the nonlinear residual) is satisfied at $u^\dag$; see \cite{Hanke1995}, \cite[Chaps.~2, 3]{Kaltenbacher2008}, \cite[Chap.~10]{Scherzer2011}. Needless to say, if $F$ is not Gâteaux differentiable, this procedure is not applicable.

However, Scherzer showed in \cite{Scherzer1995} that it is possible to replace the Fréchet derivative $F'(u)$ in \eqref{eq:landweber_smooth} by another linear operator $G_u$ that is sufficiently close to $F'(u)$ in an appropriate sense, leading to the so-called \emph{modified Landweber method}; in \cite{Kuegler:2003,Kuegler:2005}, such an operator was constructed for a class of parameter identification problems for linear elliptic equations. The purpose of this work is to show that the linear operator $G_u$ in the modified Landweber method can be taken from the Bouligand subdifferential of $F$, which is defined as the set of limits of Fréchet derivatives in differentiable points (see, e.g., \cite[Def.~2.12]{okz98} or \cite[Sec.~1.3]{kk02}) and in our case can be explicitly characterized via the solution of a suitable linearized PDE (cf.~\eqref{eq:bouligand_pde} below). We refer to this special case of the modified Landweber method as \emph{Bouligand--Landweber iteration}.
The main difficulty here is that the mapping $u\mapsto G_u$ is not continuous (cf.~\cref{ex:discon}), which is a critical tool in the classical convergence analysis used in \cite{Scherzer1995}. As one of the main contributions of our work, we therefore provide a new convergence analysis of the modified Landweber method based on the concept of \emph{asymptotic stability} of the iterates $u_{n}^\delta$ (cf.~\cref{def:AS}) which we show to hold under a \emph{generalized tangential cone condition} (cf. \cref{ass:gtcc}). We verify that the necessary conditions are satisfied for the Bouligand--Landweber iteration applied to \eqref{eq:maxpde-intro} provided the set of points where the non-smooth nonlinearity is non-differentiable at the exact data $y^\dag$ has sufficiently small Lebesgue measure (cf.~\cref{cor:TCC}). Although this analysis is specific to our model problem, we expect that it can serve as a framework for the iterative regularization of other Bouligand differentiable non-smooth mapping such as those involving variational inequalities \cite{Meyer:2018,Rauls:2018a,Rauls:2018b} and Stefan-type problems.

Let us briefly comment on related literature. Non-smooth inverse problems have attracted immense interest in recent years, although the focus has been mainly in the context of non-differentiable regularization methods in Banach spaces; see, e.g., the monographs \cite{SKHK12,Scherzeretal2009} as well as the references therein. One particular aspect relevant in our context are variational source conditions used to derive convergence rates, which require no explicit assumptions on the regularity of the forward operator and are thus applicable to non-smooth operators as well; see \cite{HKPS07}. However, none of the works so far focus on inverse problems for non-differentiable operators. In particular, the construction of $G_u$ in \cite{Kuegler:2003,Kuegler:2005} crucially depends on the linearity of the PDE (for a given parameter) and leads to the continuity of the mapping $u\mapsto G_u$, which is in fact required for their analysis. (Hence, their Landweber method is ``derivative-free'' in the same sense that Krylov methods can be implemented in a ``matrix-free'' way.) 
An alternative to iterative regularization is Tikhonov regularization, which for problems of the form \eqref{eq:maxpde-intro} leads to optimization problems that are known as \emph{mathematical programs with complementarity constraints}, which are challenging both analytically and numerically. 
Well-posedness and the numerical solution, but not its regularization properties, for the specific example of \eqref{eq:maxpde-intro} were treated in \cite{Christof2017}, on which our analysis is based. Similar results for a parabolic version of \eqref{eq:maxpde-intro} were obtained in \cite{ms16}.

\bigskip

This paper is organized as follows. After briefly summarizing basic notation,
we give our new convergence analysis of the modified Landweber method in \cref{sec:Alg}: in \cref{sec:free}, we show its well-posedness as well as the convergence in the noise-free setting, while \cref{sec:Reg} is devoted to its asymptotic stability and its regularization property. 
\Cref{sec:aux} then verifies the necessary assumptions for the specific model problem \eqref{eq:maxpde-intro}, in particular the generalized tangential cone condition, showing convergence and regularization properties of the corresponding Bouligand--Landweber iteration. Numerical examples illustrating its properties are presented in \cref{sec:Num}. 
Finally, the more technical \cref{app:piecewise-diff} extends the results of \cref{sec:aux} to a more general class of non-smooth PDEs involving piecewise differentiable nonlinearities.

\paragraph*{Notation.}  For a Hilbert space $X$, we denote by $(\cdot,\cdot)_{X}$ and $\|\cdot\|_{X}$ the inner product and the norm on $X$, respectively. For a given $z\in Z$ and $\rho>0$, we denote by $B_Z(z,\rho)$ and $\overline{B}_Z(z, \rho)$, the open and closed balls in $Z$ of radius $\rho$ centered at $z$.  For each measurable function $u$, we write $\{u <0\}$, $\{u=0\}$, and $\{u>0\}$ for the sets of almost every $x \in \Omega$ at which $u(x)$ is negative, zero and positive. For a measurable set $S$, we denote by $|S|$ its $d$-dimensional Lebesgue measure of $S$ and by $\1_S$ its characteristic function, i.e., $\1_S(x) = 1$ if $x \in S$ and $\1_S(x) = 0$ if $x \notin S$. Finally, the set of all bounded linear operators between the Hilbert spaces $X$ and $Y$ is denoted by $\Linop(X,Y)$.

\section{A new convergence analysis of the modified Landweber method}\label{sec:Alg}

The goal of this section is to show that the modified Landweber method of \cite{Scherzer1995} converges under more general conditions that are applicable to the non-smooth model problem \eqref{eq:maxpde-intro}.

We thus consider for some mapping $F:U\to Y$ between the real Hilbert spaces $U$ and $Y$ the inverse problem
\begin{equation} \label{eq:general-inverse}
    F(u) = y^\dag 
\end{equation}
for given $y^\dag\in \mathcal{R}(F)$, i.e., there exists a $u^\dag\in U$ with $F(u^\dag)=y^\dag$. 
For some $\rho> 0$, let 
\begin{equation*}
    S(u^\dag , \rho) := \left\{ u \in \overline  B_U(u^\dag , \rho ): F(u) = y^\dag \right\}
\end{equation*}
stand for the set of all solutions in $\overline B_U(u^\dag,\rho)$ of \eqref{eq:general-inverse}. Obviously, $u^\dag\in S(u^\dag ,\rho)$ for all $\rho> 0$.

We assume that $F$ together with a mapping $u\mapsto G_u\in \Linop(U,Y)$   satisfies the following conditions.
\begin{assumption}
\item\label{ass:cont} $F:U\to Y$ is completely continuous.

\item\label{ass:bouligand} There exist constants $L>0$ and $\rho_0>0$ such that
    $\| G_u \|_{\Linop(U,Y)} \leq L$ for every $u \in \overline B_U(u^\dag , \rho_0)$.

\item\label{ass:gtcc} There exist constants $\rho \in (0, \rho_0]$ and $\mu\in[0,1)$ such that the \emph{generalized tangential cone condition} 
    \begin{equation}
        \label{eq:GTCC}
        \|F(\hat u) - F(u) - G_u(\hat u - u) \|_{Y} \leq \mu \|F(\hat u) - F(u) \|_{Y}
        \tag{GTCC}
    \end{equation} 
    for all $u,\hat u \in \overline B_U(u^\dag,\rho)$ holds.

\item\label{ass:adj_bounded} There exists a Banach space $Z$ 
    such that $\bigcup_{u\in U}\mathcal{R}(G_u^*) \subset Z$ with $Z\subset U$ compactly. 
    Moreover, there exists a constant $\hat L>0$ such that 
    $\| G_u^* \|_{\Linop(Y,Z)} \leq \hat L$
    for all $u \in \overline  B_U(u^\dag,\rho)$.
\end{assumption}
Note that in contrast to \cite{Scherzer1995}, we do not require the continuity of the mapping $u\mapsto G_u$.  

Let now $y^\delta \in Y$ with $\norm{y^\delta-y^\dag}_Y \leq \delta$.
The \emph{modified Landweber iteration} for $F$ and $u\mapsto G_u$ is then given by
\begin{equation}
    \label{eq:modified-landweber} 
    u_{n+1}^\delta = u_n^\delta + w_n G_{u_n^\delta}^* \left( y^\delta - F(u_n^\delta) \right), \quad n \geq 0,
\end{equation} 
for the starting point $u_0^\delta := u_0$ and the step sizes $w_n >0$.
The iteration is stopped after $N_\delta := N(\delta,y^\delta)$ steps according to the \emph{discrepancy principle}, i.e., such that
\begin{equation}
    \label{eq:discrepancy}
    \| y^\delta - F(u^\delta_{N_\delta}) \|_{Y} \leq \tau \delta < \| y^\delta - F(u^\delta_{n}) \|_{Y}, \quad 0 \leq n < N_\delta,
\end{equation} 
for some constant $\tau >1$.

\subsection{Well-posedness and convergence}\label{sec:free}

We first show the well-posedness of \eqref{eq:modified-landweber} under our new assumptions. The proof of the following lemma is similar to the one in \cite[Prop.~2.2]{Hanke1995} with some modifications.
\begin{lemma} \label{lem:aux} 
    Assume that \cref{ass:bouligand,ass:gtcc} are fulfilled and let $\tau >1$, $\Lambda \geq \lambda > 0$ be such that
    \begin{equation} \label{choice}
        \frac{2(\mu + 1)}{\tau} - (2 - 2\mu - \Lambda L^2) < 0.
    \end{equation}
    Then, for any $\delta>0$, any starting point $u_0 \in \overline  B_{U}(u^\dag, \rho)$, and the step sizes $w_n \in [\lambda, \Lambda]$, the sequence $\{u_n^\delta\}_{0 \leq n \leq N_\delta}$ generated by \eqref{eq:modified-landweber} with the stopping index $N_\delta$ defined by the discrepancy principle \eqref{eq:discrepancy} satisfies the following assertions:
    \begin{enumerate}[label=(\roman*)]
        \item\label{it:aux1} the stopping index is finite, i.e., $N_\delta < \infty$;
        \item\label{it:aux2} $\| u_{n+1}^\delta - \tilde{u}\|_{U} < \| u_{n}^\delta - \tilde{u}\|_{U}$ for all $0 \leq n \leq N_\delta-1$ and for any $\tilde{u} \in S(u^\dag ,\rho)$.  Consequently, $u_n^\delta \in \overline  B_{U}(u^\dag, \rho)$ for all $0 \leq n \leq N_\delta$. 
    \end{enumerate}
\end{lemma}
\begin{proof}
    We first justify the inequality in assertion \ref{it:aux2} and therefore prove   by induction that $u_n^\delta \in \overline  B_{U}(u^\dag, \rho)$ for all $0 \leq n \leq N_\delta$. By assumption, $u_0^\delta = u_0 \in \overline  B_{U}(u^\dag, \rho)$. Let us now assume that $u_n^\delta \in \overline  B_{U}(u^\dag, \rho)$ for some $n \leq N_\delta -1$ and let $\tilde{u}$ be an arbitrary element of $S(u^\dag ,\rho)$.     
    We have
    \begin{multline}
        \label{eq:esti0}
        \| u_{n+1}^\delta - \tilde{u} \|_{U}^2 - \| u_{n}^\delta - \tilde{u} \|_{U}^2 \\
        \begin{aligned}[t]
            & = 2 \left( u_n^\delta - \tilde{u}, u_{n+1}^\delta - u_n^\delta \right)_{U} + \| u_{n+1}^\delta - u_n^\delta\|_{U}^2 \\
            & = 2 w_n \left(G_{u_n^\delta} (u_n^\delta - \tilde{u}), y^\delta - F(u_n^\delta) \right)_{Y} + \| u_{n+1}^\delta - u_n^\delta\|_{U}^2 \\
            & = 2 w_n \left(F(\tilde{u}) - F(u_n^\delta) - G_{u_n^\delta} (\tilde{u} -u_n^\delta), y^\delta - F(u_n^\delta) \right)_{Y} \\
            & \quad - 2 w_n \left(F(\tilde{u}) - F(u_n^\delta), y^\delta - F(u_n^\delta) \right)_{Y} + \| u_{n+1}^\delta - u_n^\delta\|_{U}^2 \\
            & = 2 w_n \left(F(\tilde{u}) - F(u_n^\delta) - G_{u_n^\delta} (\tilde{u} -u_n^\delta), y^\delta - F(u_n^\delta)\right)_{Y} - 2w_n \| y^\delta - F(u_n^\delta)\|_{Y}^2 \\
            & \quad - 2w_n \left(y^\dag - y^\delta, y^\delta - F(u_n^\delta) \right)_{Y} + w_n^2 \left\|G_{u_n^\delta}^*\left(y^\delta - F(u_n^\delta) \right) \right\|_{U}^2, 
        \end{aligned}
    \end{multline}
    which together with \cref{ass:gtcc} implies that
    \begin{multline}\label{esti1}
        \| u_{n+1}^\delta - \tilde{u} \|_{U}^2 - \| u_{n}^\delta - \tilde{u} \|_{U}^2 \\
        \begin{aligned}[t]
            &\leq 2w_n\mu \|y^\dag- F(u_n^\delta) \|_{Y} \|y^\delta - F(u_n^\delta) \|_{Y} - 2w_n \| y^\delta - F(u_n^\delta)\|_{Y}^2 \\ 
            &\quad + 2w_n \delta \|y^\delta - F(u_n^\delta) \|_{Y}  + L^2w_n^2 \| y^\delta - F(u_n^\delta)\|_{Y}^2\\
            &= w_n \|y^\delta - F(u_n^\delta) \|_{Y} \left[ 2\mu \|y^\dag- F(u_n^\delta) \|_{Y} - (2 - L^2w_n) \| y^\delta - F(u_n^\delta)\|_{Y} + 2 \delta \right ].
        \end{aligned}
    \end{multline}
    Here we have used the fact that $\| G^*_{u} \|_{\Linop(Y, U)} = \| G_{u} \|_{\Linop(U, Y)}$ and the uniform bound from \cref{ass:bouligand}.
    From the discrepancy principle \eqref{eq:discrepancy}, one has
    \begin{equation} \label{disc1}
        \delta < \frac{1}{\tau}\| y^\delta - F(u_n^\delta)\|_{Y} \quad \text{for all } 0 \leq n < N_\delta 
    \end{equation}
    and so 
    \begin{equation*}
        \begin{aligned}
            \|y^\dag- F(u_n^\delta) \|_{Y} & \leq \delta + \|y^\delta- F(u_n^\delta) \|_{Y} \\
                                           & < \left( \frac{1}{\tau} + 1 \right)\|y^\delta- F(u_n^\delta) \|_{Y} 
            \quad \text{for all } 0\leq n<N_\delta.
        \end{aligned}
    \end{equation*} 
    This together with \eqref{esti1} and \eqref{disc1} implies for all $0 \leq n < N_\delta$ that 
    \begin{equation}
        \label{eq:esti_dec} 
        \begin{aligned}[t]
            \| u_{n+1}^\delta - \tilde{u} \|_{U}^2 - \| u_{n}^\delta - \tilde{u} \|_{U}^2 
            & < w_n \|y^\delta - F(u_n^\delta) \|_{Y}^2 \left[ 2\mu \left( \frac{1}{\tau} + 1 \right) - (2 - L^2w_n) + \frac{ 2}{\tau} \right ] \\
            & \leq w_n \left( \frac{2(\mu + 1)}{\tau} - (2 - 2\mu - \Lambda L^2) \right) \|y^\delta - F(u_n^\delta) \|_{Y}^2 \\
            & \leq \lambda\left( \frac{2(\mu + 1)}{\tau} - (2 - 2\mu - \Lambda L^2) \right) \|y^\delta - F(u_n^\delta) \|_{Y}^2 \\
            & = - \alpha \|y^\delta - F(u_n^\delta) \|_{Y}^2 
        \end{aligned}
    \end{equation} 
    with 
    \begin{equation*}
        \alpha :=- \lambda\left( \frac{2(\mu + 1)}{\tau} - (2 - 2\mu - \Lambda L^2) \right) > 0.
    \end{equation*}
    Here we have used the choice of parameters $w_n \in [\lambda, \Lambda]$ and condition \eqref{choice} in the last inequality.
    This implies that
    \begin{equation}
        \label{esti2}
        \| u_{n+1}^\delta - \tilde{u} \|_{U}^2 < \| u_{n}^\delta - \tilde{u} \|_{U}^2.   
    \end{equation}
    Applying \eqref{esti2} to the case $\tilde{u} = u^\dag$, we obtain $u_{n+1}^\delta \in \overline B_U(u^\dag ,\rho)$. Proceeding as above, we can show that \eqref{esti2} holds for all $0 \leq n \leq N_\delta -1$. This yields assertion \ref{it:aux2}.

    To obtain assertion \ref{it:aux1}, we first define the set
    \begin{equation*}
        I := \{ n \in \N: \|y^\delta - F(u_n^\delta) \|_{Y} > \tau \delta \}.
    \end{equation*}
    For any $n \in I$, we see from \eqref{eq:esti_dec} that            
    \begin{equation*}
        \|y^\delta - F(u_n^\delta) \|_{Y}^2      < \frac{1}{\alpha} \left( \| u_{n}^\delta - \tilde{u} \|_{U}^2 - \| u_{n+1}^\delta - \tilde{u} \|_{U}^2 \right)
    \end{equation*}
    and thus 
    \begin{equation} \label{esti3}
        \sum_{n \in I} \|y^\delta - F(u_n^\delta) \|_{Y}^2 < \frac{1}{\alpha} \| u_{0}- \tilde{u} \|_{U}^2 < \infty.
    \end{equation}
    From the definition of the set $I$, we obtain $\|y^\delta - F(u_n^\delta) \|_{Y} > \tau \delta$ for all $n \in I$ and therefore
    \begin{equation*}
        \sum_{n \in I} \|y^\delta - F(u_n^\delta) \|_{Y}^2 > \sum_{n \in I} (\tau \delta)^2 = (\tau \delta)^2|I|.
    \end{equation*}
    This together with \eqref{esti3} ensures that the set $I$ and hence $N_\delta = |I| +1$ is finite as claimed. 
\end{proof}



From now on, we need to differentiate between the cases of noise-free $(\delta =0)$ and noisy $(\delta >0)$ data. 
Let thus $u_n^\delta$, $y_n^\delta := F(u_n^\delta)$ and $u_n$, $y_n:= F(u_n)$ be generated by the modified Landweber iteration \eqref{eq:modified-landweber} corresponding to $\delta >0$ and $\delta=0$, respectively. We first consider the noise-free setting.
\begin{lemma} \label{lem:RW}
    Let \cref{ass:bouligand,ass:gtcc} be fulfilled. Let further $\lambda$ and $\Lambda$  satisfy  $\Lambda \geq \lambda >0$ and
    \begin{equation} \label{choice2}
        (2 - 2\mu - \Lambda L^2) > 0.
    \end{equation}
    Then, for any starting point $u_0 \in \overline  B_{U}(u^\dag, \rho)$ and the step sizes $\{w_n\}_{n\in\N} \subset [\lambda, \Lambda]$, we have that
    \begin{equation} \label{eq:decreasing_free}
        \| u_{n+1} - u^\dag \|_{U}^2 \leq \| u_{n} - u^\dag \|_{U}^2 \quad \text{for all } n \geq 0 
    \end{equation}
    and
    \begin{equation} \label{reasonable-wanderer}
        \sum_{n=0}^\infty \| y^\dag - F(u_n) \|_{Y}^2 < \infty.
    \end{equation}
\end{lemma}
\begin{proof}
    Similarly to \eqref{eq:esti0} with $\tilde{u} := u^\dag$, we obtain that
    \begin{multline*}
        \| u_{n+1}- u^\dag \|_{U}^2 - \| u_{n} - u^\dag \|_{U}^2  = 2 w_n \left(F(u^\dag) - F(u_n) - G_{u_n} (u^\dag -u_n), y^\dag - F(u_n)\right)_{Y}  \\
        \begin{aligned}
            &- 2w_n \| y^\dag - F(u_n)\|_{Y}^2 + w_n^2 \left\|G_{u_n}^*\left(y^\dag - F(u_n) \right) \right\|_{U}^2, 
        \end{aligned}
    \end{multline*}
    which together with \cref{ass:gtcc,ass:bouligand} yields that
    \begin{equation*} 
        \begin{aligned}[t]
            \| u_{n+1} - u^\dag \|_{U}^2 - \| u_{n} - u^\dag \|_{U}^2 
            & \leq \|y^\dag - F(u_n) \|_{Y}^2 \left[ 2 w_n \mu - 2w_n + w_n^2 L^2 \right]\\
            & \leq - \lambda \left(2- 2\mu - L^2\Lambda \right)\|y^\dag - F(u_n) \|_{Y}^2 
        \end{aligned}
    \end{equation*} 
    for all $n\geq 0$, where we have used the fact that $w_n \in [\lambda, \Lambda]$ for all $n\geq 0$.
    Consequently, we obtain \eqref{eq:decreasing_free} and
    \begin{equation*} 
        \sum_{n=0}^\infty \| y^\dag - F(u_n) \|_{Y}^2 \leq \frac{1}{\lambda \left(2- 2\mu - L^2\Lambda \right)} \| u_{0} - u^\dag \|_{U}^2 <\infty,
    \end{equation*}
    which yields \eqref{reasonable-wanderer}.  
\end{proof}

We can now obtain a convergence result for the noise-free setting, whose proof follows along the lines of the one of \cite[Thm.~2.3]{Hanke1995}.
\begin{theorem} \label{thm:free}
    Under the assumptions of \cref{lem:RW}, the modified Landweber iteration \eqref{eq:modified-landweber} corresponding to $\delta=0$ either stops after finitely many iterations with an iterate coinciding with an element of $S(u^\dag, \rho)$ or generates a sequence of iterates that converges strongly to an element of $S(u^\dag, \rho)$ in $U$.
\end{theorem}
\begin{proof}
    If the algorithm stops after finitely many iterations, then the last iterate $u_N$ satisfies $F(u_N) = y^\dag$ due to the discrepancy principle \eqref{eq:discrepancy}. From \eqref{eq:decreasing_free} and the fact that $u_0 \in \overline B_U(u^\dag ,\rho)$, we have $u_N \in \overline B_U(u^\dag ,\rho)$ and hence $u_N \in S(u^\dag, \rho)$.

    It remains to prove the claim for the case where the algorithm generates an infinite sequence $\{ u_n \}_{n \in \N}$. To this end, we first observe from \eqref{eq:decreasing_free} and the fact $u_0 \in \overline B_U(u^\dag ,\rho)$ that $u_n \in \overline B_U(u^\dag ,\rho)$ for all $n \geq 0$. We now set $e_n := u^\dag - u_n$ for all $n \geq 0$. Then, \eqref{eq:decreasing_free} implies that $\{\|e_n \|_{U}\}_{n \in \N}$ is monotonically decreasing and hence
    \begin{equation} \label{eq:lim1-free}
        \lim_{n \to \infty} \|e_n \|_{U} = {\gamma}
    \end{equation}
    for some ${\gamma} \geq 0$. 
    For any $m,l \in \N$ with $m \leq l$, choose
    \begin{equation} \label{choice-r-free}
        k \in \arg\min_{m\leq t\leq l}\| y^\dag - y_t \|_{Y}.
    \end{equation}
    The Cauchy--Schwarz inequality then yields that
    \begin{equation} \label{Cau-1-free}
        \| u_m - u_l \|_{U}^2 \leq 2 \left( \| u_m -  u_k \|_{U}^2 + \|  u_k - u_l \|_{U}^2 \right),
    \end{equation}
    and the three-point identity
    \begin{equation*}
        \|a-b \|_{U}^2 = \|a-c \|_{U}^2 - \|b-c \|_{U}^2 + 2(a-b,c-b)_{U} 
    \end{equation*}
    further implies that
    \begin{align*}
        \| u_m -  u_k \|_{U}^2 &= \|  u_m - u^\dag \|_{U}^2 - \| u_k - u^\dag \|_{U}^2 + 2 \left( u_{m}-  u_k, u^\dag - u_k \right)_{U}, \\
        \| u_l - u_k \|_{U}^2 &= \| u_l - u^\dag \|_{U}^2 - \| u_k - u^\dag \|_{U}^2 + 2 \left( u_{l}- u_k, u^\dag - u_k \right)_{U}. 
    \end{align*}
    Combining this with \eqref{Cau-1-free} yields that
    \begin{equation}\label{eq:ine-43-free}
        \begin{aligned}[t]
            \| u_m - u_l \|_{U}^2 & \leq 2 \left[\| e_m \|_{U}^2 + \| e_l \|_{U}^2 - 2\| e_k \|_{U}^2 \right] + 4\left(e_k - e_m, e_k \right)_{U} + 4 \left(e_k-e_l, e_k \right)_{U} \\
                                  & = a_{m,l,k} + b_{m,l,k} 
        \end{aligned}
    \end{equation}
    with
    \begin{align*}
        a_{m,l,k} &:= 2 \left[\| e_m \|_{U}^2 + \| e_l \|_{U}^2 - 2\| e_k \|_{U}^2 \right]
        \shortintertext{and} 
        b_{m,l,k} &:= 4\left(e_k-e_m, e_k \right)_{U} + 4 \left(e_k -e_l, e_k \right)_{U}.      
    \end{align*}
    Since $l \geq k \geq m$, it follows that $k \to \infty$ and $l \to \infty$ whenever $m \to \infty$. From this and \eqref{eq:lim1-free}, we obtain that
    \begin{equation} \label{eq:lim-431-free}
        a_{m,l,k} \to 0 \quad \text{as } m \to \infty.
    \end{equation}
    Moreover, we have that
    \begin{equation}\label{eq:ine-44-free}
        \left(e_{k}-e_m, e_k \right)_{U}  = \sum_{n=m}^{k-1} \left(e_{n+1}-e_n, e_k \right)_{U}  \leq \sum_{n=m}^{k-1} |\left(e_{n+1}-e_n, e_k \right)_{U}|. 
    \end{equation}
    From \eqref{eq:modified-landweber}, we then obtain that $e_{n+1} - e_{n} = -w_n G_{u_n}^*(y^\dag -  y_n)$, and hence
    \begin{equation*}
        \begin{aligned}
            \left(e_{n+1}-e_n, e_k \right)_{U} & = -w_n\left(y^\dag -  y_n,G_{u_n} e_k \right)_{Y} \\
                                               & = w_n\left(y^\dag -  y_n,G_{ u_n} (u_k - u^\dag) \right)_{Y}. 
        \end{aligned}
    \end{equation*}
    It follows that 
    \begin{equation} \label{eq:ine-45-free}
        | \left(e_{n+1}-e_n, e_k \right)_{U} | \leq w_n \| y^\dag -  y_n \|_{Y} \| G_{ u_n} ( u_k - u^\dag)\|_{Y}.
    \end{equation}
    We now estimate the term $\| G_{ u_n}( u_k - u^\dag ) \|_{Y}$. From \cref{ass:gtcc} and the triangle inequality, it follows that 
    \begin{equation}\label{eq:ine-461-free}
        \begin{aligned}[t]
            \| G_{u_n}(u_k - u^\dag ) \|_{Y} & \leq \| G_{u_n}( u^\dag - u_n ) \|_{Y} + \| G_{u_n}(u_k -u_n ) \|_{Y} \\
                                             & \leq     \| y^\dag - y_n \|_{Y} +        \| F(u^\dag) - F( u_n) - G_{ u_n}( u^\dag - u_n ) \|_{Y} \\
                                             \MoveEqLeft[-1] + \| G_{u_n}(u_k - u_n ) \|_{Y} \\
                                             & \leq     (1+ \mu)\| y^\dag - y_n \|_{Y} + \| G_{u_n}(u_k - u_n ) \|_{Y}. 
        \end{aligned}
    \end{equation}
    In addition, we see from \eqref{eq:GTCC} that 
    \begin{equation*}
        \| F(u_k) - F(u_n) - G_{u_n}( u_k - u_n ) \|_{Y} \leq \mu \| F( u_k) - F(u_n) \|_{Y}
    \end{equation*}
    and hence
    \begin{equation*}
        \begin{aligned}
            \| G_{u_n}(u_k - u_n ) \|_{Y}
            & \leq \left(1+ \mu\right)\| F( u_k) - F( u_n) \|_{Y} \\
            & \leq \left(1+ \mu\right) \left(\| y^\dag - F(u_k) \|_{Y}+ \| y^\dag - F(u_n) \|_{Y} \right) \\
            & \leq 2 \left(1+ \mu\right) \|y^\dag - y_n \|_{Y}.
        \end{aligned}
    \end{equation*}
    This and \eqref{eq:ine-461-free} give 
    \begin{equation}
        \label{eq:ine-46-free} 
        \| G_{u_n}(u_k - u^\dag ) \|_{Y} \leq 3(1+ \mu) \| y^\dag - y_n \|_{Y}.
    \end{equation} 
    The combination of this with \eqref{eq:ine-45-free} yields that
    \begin{equation*}
        | \left(e_{n+1}-e_n, e_k \right)_{U} | \leq 3(1+\mu) w_n \| y^\dag -  y_n \|_{Y}^2, 
    \end{equation*}
    which, together with \eqref{eq:ine-44-free}, ensures that 
    \begin{equation*}
        | \left(e_k - e_m, e_k \right)_{U} |  \leq 3(1+\mu) \sum_{n=m}^{k-1} w_n{\| y^\dag -  y_n \|_{Y}^2} 
        \leq 3(1+\mu) \Lambda \sum_{n=m}^{k-1}{\| y^\dag - y_n \|_{Y}^2}.
    \end{equation*}
    Similarly, we have that
    \begin{equation*}
        | \left(e_k - e_l, e_k \right)_{U} | \leq 3(1+\mu) \Lambda \sum_{n=k}^{l-1}{\| y^\dag - y_n \|_{Y}^2},
    \end{equation*}
    leading to
    \begin{equation*}
        b_{m,l,k} = 4\left(e_k-e_m, e_k \right)_{U} + 4 \left(e_k -e_l, e_k \right)_{U} 
        \leq 12(1+\mu) \Lambda \sum_{n=m}^{l-1}{\| y^\dag - y_n \|_{Y}^2}.
    \end{equation*}
    Combining this with \eqref{reasonable-wanderer} yields that
    \begin{equation}
        \label{eq:lim-432-free} 
        b_{m,l,k} \to 0 \quad \text{as } l\geq k \geq m \to \infty.
    \end{equation}
    The limits \eqref{eq:lim-431-free} and \eqref{eq:lim-432-free} together with \eqref{eq:ine-43-free} imply that $\{u_n\}_{n \in \N}$ is a Cauchy sequence in $U$.
    Thus, there exists an element $\bar u \in U$ such that $u_n \to \bar u$ and hence
    $F(u_n) \to F(\bar u)$ by \cref{ass:cont} as $n \to \infty$. 
    In addition, we see from \eqref{reasonable-wanderer} that
    $y^\dag - F(u_n) \to 0$ as $n \to \infty$, and hence $y^\dag = F(\bar u)$. Since $u_n \in \overline B_U(u^\dag ,\rho)$ for all $n \geq 0$, it holds that $\bar u \in \overline{B}_U(u^\dag ,\rho)$ and hence that $\bar u \in S(u^\dag , \rho)$, which completes the proof.
\end{proof}

\subsection{Regularization property} \label{sec:Reg}

We now consider the convergence of the modified Landweber method for $\delta\to 0$. 
To simplify the notation in this subsection, for any $\delta_k>0$ and corresponding noisy data $y^{\delta_k}\in \overline B_Y(y^\dag,\delta_k)$ we introduce $N_k:=N(\delta_k, y^{\delta_k})$ and $u_k:=u_{N_k}^{\delta_k}$.

We first note that assertion \ref{it:aux2} in \cref{lem:aux} ensures the boundedness of the family $\left\{u_k\right\}_{k \in \N}$, which together with the reflexivity of $U$ already ensures weak convergence as $\delta_k \to 0$.
\begin{proposition} 
    Assume that all hypotheses of \cref{lem:aux} hold and that in addition \cref{ass:cont} is fulfilled. Let $\{\delta_k\}_{k \in \N}$ be a positive zero sequence. Then, any subsequence of $\left\{u_k\right \}_{k \in \N}$ contains a further subsequence that converges weakly to some $\bar u\in S(u^\dag, \rho)$ in $U$. 
    In addition, if $u^\dag$ is the unique solution of \eqref{eq:general-inverse} in $\overline B_U(u^\dag , \rho)$, then $\left\{u_k\right \}_{k \in \N}$ converges weakly to $u^\dag $ in $U$.
\end{proposition}
\begin{proof}
    Without loss of generality, let $\{\delta_{k} \}_{k \in \N}$ itself be an arbitrary subsequence.
    Since $\left\{u_k\right \}_{k \in \N}$ is bounded in $U$, there exist a subsequence, also denoted by $\left\{u_k\right \}_{k \in \N}$, and an element $\bar u \in U$ such that
    \begin{equation*}
        u_k \rightharpoonup \bar u \quad \text{as } k \to \infty.
    \end{equation*} 
    By virtue of \cref{ass:cont}, 
    \begin{equation*}
        F\left(u_k\right) \to F(\bar u) \quad \text{as } k \to \infty
    \end{equation*}
    and hence $y^{\delta_k}- F(u_k) \to y^\dag - F(\bar u)$ in $Y$. From the discrepancy principle, we have that
    \begin{equation*}
        \lim_{k \to \infty} \|y^{\delta_k} - F(u_k) \|_{Y} =0, 
    \end{equation*}
    which implies that $F(\bar u) = y^\dag$ and thus $\bar u \in S(u^\dag , \rho)$. 

    If $u^\dag$ is the unique solution of \eqref{eq:general-inverse} in $\overline B_U(u^\dag , \rho)$, a subsequence--subsequence argument ensures that the original, full, sequence $\left\{u_k\right \}_{k \in \N}$ converges weakly to $u^\dag $ in $U$.
\end{proof}

In the remainder of this section, we will show that the modified Landweber iteration together with the discrepancy principle is a \emph{strongly} convergent regularization method, i.e., for any positive zero sequence $\left\{\delta_k\right \}_{k \in \N}$, the sequence $\left\{u_k\right \}_{k \in \N}$ generated by the  \eqref{eq:modified-landweber} stopped according to \eqref{eq:discrepancy} admits a subsequence that converges strongly to an element of $S(u^\dag ,\rho)$.
Note that we have \emph{not} assumed the continuity of the mapping $U \ni u \mapsto G_{u} \in \Linop(U,Y)$, which implies that $u_n^\delta$ is, in general, not continuous with respect to $y^\delta$. We therefore cannot apply the standard technique from \cite{Hanke1995,Scherzer1995,Scherzer2011}.
To overcome this difficulty,  we need the following notion. 
\begin{definition} \label{def:AS}
    Let $\{u_n^\delta\}_{n\leq N_\delta}$ be a (finite or infinite) sequence generated by an iterative method for some $\delta>0$.
    Then the method is \emph{asymptotically stable} if any positive zero sequence $\{\delta_k\}_{ k \in \N}$ has a subsequence $\{ \delta_{k_i} \}_{i \in \N}$ such that $\overline N := \lim_{i \to \infty}N_{{\delta_{k_i}}} \in \N \cup \{\infty\}$ and the following conditions hold:    
\begin{enumerate}[label=(\roman*)]
    \item\label{it:AS1} 
        For all $0 \leq n \leq \overline N$  (where the last inequality is strict if $\overline N = \infty$),
        \begin{equation}\label{asym-stab}
            u_{n}^{\delta_{k_i}} \to \tilde{u}_n \quad \text {in $U$ as } i \to \infty  
        \end{equation}
        for some $\tilde{u}_n \in \overline{B}_U(u^\dag ,\rho)$.

    \item\label{it:AS2} 
        If $\overline N = \infty$, there exists a $\tilde{u} \in S(u^\dag ,\rho)$ such that
        \begin{equation*}
            \tilde{u}_n \to \tilde{u} \quad \text {in } U \quad \text{as }  n \to \infty.
        \end{equation*} 
\end{enumerate}
\end{definition}
We now show that the modified Landweber iteration \eqref{eq:modified-landweber} is asymptotically stable under the \crefrange{ass:cont}{ass:adj_bounded}. The proof consists of a sequence of technical lemmas. The first lemma verifies condition \ref{it:AS1} in \cref{def:AS}.
\begin{lemma} \label{lem:asym-stab1}
    Assume that \crefrange{ass:cont}{ass:adj_bounded} as well as \eqref{choice} hold. 
    Let the starting point $u_0 \in \overline  B_{U}(u^\dag, \rho)$ and the step sizes $w_n \in [\lambda, \Lambda]$ be arbitrary. Assume further that $\{\delta_{k}\}_{k \in \N}$ is a positive zero sequence. Then there exist a subsequence $\{\delta_{k_i} \}_{i \in \N}$  and a sequence $\{\tilde{u}_n \}_{n\in\N} \subset \overline  B_{U}(u^\dag, \rho)$ such that condition \ref{it:AS1} in \cref{def:AS} is fulfilled.

    Moreover, the sequence $\{\tilde{u}_n \}_{n\in\N}$ satisfies
    \begin{equation}\label{eq:Land-aux}
        \tilde{u}_0 = u_0, \quad \tilde{u}_{n+1} = \tilde{u}_n + w_n G_{\tilde{u}_n}^*(y^\dag - F(\tilde{u}_n)) + w_n r_n
    \end{equation} 
    for some $r_n\in Z$ and for all $0 \leq n < \overline N$, where $\overline N := \lim_{i \to \infty}N_{k_i}$.
\end{lemma}
\begin{proof}
    We first note that since $\{N_{k}\}_{k\in\N}$ is a sequence of natural numbers, there exists a subsequence $\{\delta_{k_i} \}_{i \in \N}$ such that $N_{k_i}$ either is constant for all $i$ large enough or tends increasingly to infinity as $i \to \infty$. 

    We now show by induction that there exist a sequence $\left\{\tilde{u}_n \right\}_{n\in\N} \subset \overline{B}_U(u^\dag ,\rho)$ and a subsequence  of $\{ \delta_{k_i} \}_{i \in \N}$, which fulfill the assertion of the lemma. To this end, we start with the case where $N_{k_i}$ tends increasingly to infinity as $i \to \infty$. In order to simplify the notation, we set $u_n^i := {u}_{n}^{\delta_{k_i}}$, $y_n^i := F({u}_{n}^{\delta_{k_i}})$, and $y^i :=  y^{\delta_{k_i}}$. 

    First, \eqref{asym-stab} holds for $n=0$ with $\tilde{u}_0 = u_0 \in \overline B_U(u^\dag ,\rho)$.
    By a slight abuse of notation, we assume $\{\delta_{k_i}\}_{i \in \N}$ itself is a subsequence satisfying  ${u}_{n}^{i} \to \tilde{u}_{n}$ as $i \to \infty$ for some $\tilde{u}_n \in \overline{B}_U(u^\dag , \rho)$. 
    Setting
    \begin{equation*}
        a_{n}^{i} := G_{{u}_{n}^i}^*(y^i - y_n^{i}), \qquad a_n := G_{\tilde{u}_{n}}^*(y^\dag - \tilde{y}_n), \qquad \zeta_n^i := a_n^i - a_n
    \end{equation*} 
    with $\tilde{y}_n := F(\tilde{u}_n)$,
    we have that
    \begin{equation*}
        \begin{aligned}
            \zeta_n^i & = G_{{u}_{n}^i}^*(y^i - y_n^{i}) - G_{\tilde{u}_{n}}^*(y^\dag - \tilde{y}_n) \\
                      & = \left[G_{{u}_{n}^i}^*(y^\dag - \tilde{y}_n) - G_{\tilde{u}_{n}}^*(y^\dag - \tilde{y}_n) \right] + G_{{u}_{n}^i}^*(y^i - y_n^{i} - y^\dag + \tilde{y}_n) \\
                      & = \eta_n^i - \eta_n + b_n^i 
        \end{aligned}
    \end{equation*}
    with
    \begin{align*}
        & \eta_n^i := G_{{u}_{n}^i}^*(y^\dag - \tilde{y}_n), \qquad \eta_n := a_n = G_{\tilde{u}_{n}}^*(y^\dag - \tilde{y}_n), \\
        & b_n^i := G_{{u}_{n}^i}^*(y^i - y_n^{i} - y^\dag + \tilde{y}_n). 
    \end{align*}
    \Cref{ass:cont} together with the fact $u_{n}^i \to \tilde{u}_n$ now implies that $y_n^i \to \tilde{y}_n$ as $i \to \infty$. From this and the boundedness of $\{ \| G_{{u}_{n}^i}^*\|_{\Linop(Y, U)} \}_{i \in \N}$ by \cref{ass:bouligand}, we obtain that
    \begin{equation}
        b_n^i \to 0 \quad \text{in $U$ as } i \to \infty. \label{eq:lim-bn}
    \end{equation}
    From \cref{ass:adj_bounded}, we further see that $\{\eta_n^i\}_{i \in \N}$ and hence $\{\eta_n^i - \eta_n\}_{i \in \N}$ is bounded in $Z$. Since $Z \hookrightarrow U$ compactly,  
    there exist an $r_n \in Z$ and a subsequence of $\{\delta_{k_i}\}_{k \in \N}$, denoted in the same way, such that
    \begin{equation}
        \label{eq:lim-eta}                      
        \eta_n^{i} - \eta_n \to r_n \quad \text {in $U$ as } i \to \infty.       
    \end{equation}
    Since
    \begin{equation*}
        \begin{aligned}
            u_{n+1}^i  &= u_n^i + w_n G_{u_n^i}^* \left( y^i - y_n^i  \right) \\
                       & = u_n^i + w_n a_n^{i} \\
                       & = u_n^i + w_n a_n + w_n (\eta_n^{i} - \eta_n) + w_n b_n^{i},
        \end{aligned}
    \end{equation*}
    letting $i \to \infty$ and using the limits \eqref{eq:lim-bn}, \eqref{eq:lim-eta}, and $u_n^i \to \tilde{u}_n$ implies that 
    \begin{equation*}
        u_{n+1}^i  \to \tilde{u}_n + w_na_n + w_n r_n = \tilde{u}_n + w_n G_{\tilde{u}_{n}}^*(y^\dag - \tilde{y}_n) + w_nr_n.
    \end{equation*}
    By setting $\tilde{u}_{n+1} := \tilde{u}_n + w_n G_{\tilde{u}_{n}}^*(y^\dag - \tilde{y}_n) + w_nr_n$, we obtain \eqref{asym-stab} for $n+1$ as well as \eqref{eq:Land-aux}. Since $u_{n+1}^i\in \overline B_{U}(u^\dag,\rho)$ for all $i\in\N$, also $\tilde{u}_{n+1} \in \overline  B_{U}(u^\dag, \rho)$. 

    The argument for the case where $\bar N < \infty$ proceeds similarly.
\end{proof}


In order to verify condition \ref{it:AS2} in \cref{def:AS}, we need the following properties of sequences $\left\{\tilde{u}_n \right\}_{n\in\N}$ and $\left\{ r_n \right\}_{n\in\N}$.
\begin{lemma}  \label{lem:asym-stab2}
    Assume the conditions of \cref{lem:asym-stab1} hold. If the sequence $\left\{\delta_{k_i} \right\}_{i \in \N}$ in \cref{lem:asym-stab1} satisfies $N_{k_i} \to \infty$ as $i \to \infty$, then the sequences $\left\{ \tilde{u}_n \right\}_{n \in \N}$ and $\left\{ r_n \right\}_{n \in \N}$ given in \eqref{eq:Land-aux} satisfy for all $n\in\N$ the following estimates:
    \begin{enumerate}[label=(\roman*)]
        \item\label{it:asym-stab1} $\| r_n \|_{U} \leq 2L \| y^\dag - \tilde{y}_n \|_{Y}$,
        \item\label{it:asym-stab2} $\left( r_n, \tilde{u}_n - \tilde{u} \right)_{U} \leq (-1 + \mu) \| y^\dag - \tilde{y}_n \|_{Y}^2 - \left( y^\dag - \tilde{y}_n, G_{\tilde{u}_n} (\tilde{u}_n - \tilde{u}) \right)_{Y}$,
        \item\label{it:asym-stab3} $\left| \left( r_n, \tilde{u}_m - \tilde{u} \right)_{U}\right |\leq 2(1+\mu)\|y^\dag - \tilde{y}_n \|_{Y} \left[ \|y^\dag - \tilde{y}_n \|_{Y} +  \|\tilde{y}_m - \tilde{y}_n \|_{Y} \right]$
            for all $m \geq 0$,
    \end{enumerate} 
    for $\tilde{y}_n := F(\tilde{u}_n)$, any $\tilde{u} \in S(u^\dag ,\rho)$, and  $L>0$ from \cref{ass:bouligand}.      
\end{lemma}
\begin{proof}
    We employ the same notation as in the proof of \cref{lem:asym-stab1}. For \ref{it:asym-stab1}, we obtain from \cref{ass:bouligand} that
    \begin{equation*}
        \|\eta_n^{i} - \eta_n \|_{U} = \|  G_{{u}_{n}^i}^*(y^\dag - \tilde{y}_n) - G_{\tilde{u}_{n}}^*(y^\dag - \tilde{y}_n) \|_{U} \leq 2L \| y^\dag - \tilde{y}_n \|_{Y}. 
    \end{equation*}
    Combining this with \eqref{eq:lim-eta} yields that 
    \begin{equation*}
        \| r_n \|_{U} = \lim_{i \to \infty} \|\eta_n^{i} - \eta_n\|_{U}   \leq 2L \| y^\dag - \tilde{y}_n \|_{Y}, 
    \end{equation*}
    which gives assertion \ref{it:asym-stab1}.

    For \ref{it:asym-stab2}, let $\tilde{u}\in S(u^\dag,\rho)$ be arbitrary. 
    We then see from \eqref{eq:lim-eta} that 
    \begin{equation}\label{rn1}
        \begin{aligned}[t]
            \left( r_n, \tilde{u}_n - \tilde{u} \right)_{U} & = \lim_{i \to \infty}      \left( \eta_n^{i} - \eta_n, \tilde{u}_n - \tilde{u} \right)_{U} \\
                                                            & = \lim_{i \to \infty}      \left( y^\dag - \tilde{y}_n, G_{u_n^i}( \tilde{u}_n - \tilde{u}) \right)_{Y} -  \left( y^\dag - \tilde{y}_n, G_{\tilde{u}_n}( \tilde{u}_n - \tilde{u})\right)_{Y}\\
                                                            & = \lim_{i \to \infty} A_n^{i} - B_n 
        \end{aligned}
    \end{equation}
    with 
    \begin{align*}
        A_n^{i} &:= \left( y^\dag - \tilde{y}_n, G_{u_n^i }( \tilde{u}_n -  \tilde{u}) \right)_{Y}, \\
        B_n&:= \left( y^\dag -\tilde{y}_n, G_{\tilde{u}_n}( \tilde{u}_n - \tilde{u})\right)_{Y}.                         
    \end{align*}
    Moreover, 
    \begin{equation*}
        \begin{aligned}
            A_n^{i} & = \left( y^\dag - {y}_n^{i}, G_{u_n^i}( \tilde{u}_n - \tilde{u}) \right)_{Y} + \left( y_n^{i} - \tilde{y}_n, G_{u_n^i}( \tilde{u}_n - \tilde{u}) \right)_{Y} \\
                    & = \left( y^\dag - {y}_n^{i},y^\dag - {y}_n^{i} - G_{u_n^i}( \tilde{u} - \tilde{u}_n)\right)_{Y}\\
                    &\quad\quad - \|y^\dag - {y}_n^{i} \|_{Y}^2 + \left( y_n^{i} - \tilde{y}_n, G_{u_n^i}( \tilde{u}_n - \tilde{u}) \right)_{Y} \\
                    & = \left( y^\dag - {y}_n^{i},y^\dag - {y}_n^{i} - G_{u_n^i}( \tilde{u} - u_n^i)\right)_{Y} - \left( y^\dag - {y}_n^{i}, G_{u_n^i}( u_n^i - \tilde{u}_n)\right)_{Y} \\
                    & \quad \quad - \|y^\dag - {y}_n^{i} \|_{Y}^2 + \left( y_n^{i} - \tilde{y}_n, G_{u_n^i}( \tilde{u}_n - \tilde{u}) \right)_{Y} \\
                    & \leq (-1+\mu) \|y^\dag - {y}_n^{i} \|_{Y}^2 +L \|y^\dag - {y}_n^{i} \|_{Y} \|u_n^i - \tilde{u}_n \|_{U} \\
                    & \quad \quad + L \| y_n^{i} - \tilde{y}_n \|_{Y} \|\tilde{u}_n - \tilde{u}\|_{U}, 
        \end{aligned}
    \end{equation*}             
    where the last inequality follows from \cref{ass:gtcc,ass:bouligand} together with the Cauchy--Schwarz inequality.
    Letting $i \to \infty$, we have that $u_n^i \to \tilde{u}_n$ and $y_n^{i} \to \tilde{y}_n$, and hence
    \begin{equation*}
        \lim_{i \to \infty}      A_n^{i} \leq (-1 + \mu ) \|y^\dag - \tilde{y}_n \|_{Y}^2. 
    \end{equation*} 
    From this and \eqref{rn1}, we obtain assertion \ref{it:asym-stab2}. 

    For assertion \ref{it:asym-stab3}, we first estimate 
    \begin{equation*}
        \begin{aligned}
            \left| \left( r_n, \tilde{u}_m - \tilde{u} \right)_{U} \right | & = \lim_{i \to \infty} \left| \left( \eta_n^{i} - \eta_n, \tilde{u}_m - \tilde{u} \right)_{U} \right| \\
                                                                            & = \lim_{i \to \infty} \left| \left( y^\dag - \tilde{y}_n, G_{u_n^i}( \tilde{u}_m - \tilde{u}) \right)_{Y} - \left( y^\dag - \tilde{y}_n, G_{\tilde{u}_n}( \tilde{u}_m - \tilde{u}) \right)_{Y} \right| \\
                                                                            & \leq \|y^\dag - \tilde{y}_n \|_{Y} \left[ \limsup_{i \to \infty} \|G_{u_n^i}( \tilde{u}_m - \tilde{u}) \|_{Y} + \|G_{\tilde{u}_n}( \tilde{u}_m - \tilde{u}) \|_{Y}\right]. 
        \end{aligned}
    \end{equation*}
    Due to \cref{ass:gtcc}, we can apply the \eqref{eq:GTCC} to obtain
    \begin{equation*}
        \begin{aligned}
            \|G_{u_n^i}( \tilde{u}_m - \tilde{u}) \|_{Y} & \leq \|G_{u_n^i}( u_n^i - \tilde{u}) \|_{Y} + \|G_{u_n^i}( \tilde{u}_m - u_n^i) \|_{Y} \\
                                                         & \leq (1+\mu) \|y^\dag - y_n^{i} \|_{Y} + (1+ \mu) \|\tilde{y}_m - y_n^{i} \|_{Y} \\
                                                         & = (1+\mu) \left[ \|y^\dag - y_n^{i} \|_{Y} +  \|\tilde{y}_m - y_n^{i} \|_{Y} \right],
        \end{aligned}
    \end{equation*} 
    which implies that 
    \begin{equation*}
        \limsup_{i \to \infty} \|G_{u_n^i}( \tilde{u}_m - \tilde{u}) \|_{Y} \leq (1+\mu) \left[\|y^\dag - \tilde{y}_n \|_{Y} +  \|\tilde{y}_m - \tilde{y}_n \|_{Y}\right].
    \end{equation*} 
    Also, \eqref{eq:GTCC} yields that
    \begin{equation*}
        \|G_{\tilde{u}_n}( \tilde{u}_m - \tilde{u}) \|_{Y} \leq (1+\mu) \left[\|y^\dag - \tilde{y}_n \|_{Y} +  \|\tilde{y}_m - \tilde{y}_n \|_{Y}\right].
    \end{equation*}                             
    From the above inequalities, we obtain that
    \begin{equation*}
        \begin{aligned}
            &\left| \left( r_n, \tilde{u}_m - \tilde{u} \right)_{U} \right |  \leq 2(1+\mu)\|y^\dag - \tilde{y}_n \|_{Y} \left[ \|y^\dag - \tilde{y}_n \|_{Y} +  \|\tilde{y}_m - \tilde{y}_n \|_{Y} \right], 
        \end{aligned}
    \end{equation*} 
    which yields \ref{it:asym-stab3}.
\end{proof}


The following lemma completes the verification of condition \ref{it:AS2} in \cref{def:AS}. Its proof is a modification of the proof of \cref{thm:free}.      
\begin{lemma} \label{lem:conv-aux}
    Assume that \crefrange{ass:cont}{ass:adj_bounded} hold. Let further $\lambda$ and $\Lambda$  satisfy  $\Lambda \geq \lambda >0$ and
    \begin{equation}\label{choice-aux}
        -1 + \mu + 5\Lambda L^2  < 0. 
    \end{equation}
    Let the starting point $u_0 \in \overline  B_{U}(u^\dag, \rho)$ and the step sizes $w_n \in [\lambda, \Lambda]$ be arbitrary. Assume furthermore that $\{\tilde{u}_{n}\}_{n \in \N}$ is defined by \eqref{eq:Land-aux} and satisfies conditions \ref{it:asym-stab1}--\ref{it:asym-stab3} of \cref{lem:asym-stab2}. Then $\{\tilde{u}_{n}\}_{n \in \N}$ converges strongly to some $\bar u\in S(u^\dag ,\rho)$ as $n \to \infty$.
\end{lemma}
\begin{proof} 
    From \eqref{eq:Land-aux}, assertions \ref{it:asym-stab1} and \ref{it:asym-stab2} of \cref{lem:asym-stab2} for the case where $\tilde{u} := u^\dag$, and the Cauchy--Schwarz inequality, we have that
    \begin{equation}\label{bound-aux1}
        \begin{multlined}[t][0.9\linewidth]
            \| \tilde u_{n+1} - u^\dag \|_{U}^2 - \|\tilde u_{n} - u^\dag \|_{U}^2 \\
            \begin{aligned}[t]
                & = 2 \left( \tilde u_n - u^\dag, \tilde u_{n+1} - \tilde u_n \right)_{U} + \|\tilde u_{n+1} -\tilde u_n\|_{U}^2 \\
                & = 2 w_n \left(G_{\tilde u_n} (\tilde u_n - u^\dag), y^\dag - \tilde{y}_n \right)_{Y} +2 w_n \left(r_n, \tilde u_n - u^\dag \right)_{U} + \|\tilde u_{n+1} -\tilde u_n\|_{U}^2 \\
                & \leq 2w_n(-1+ \mu)\|y^\dag - \tilde{y}_n \|_{Y}^2 + w_n^2\| G_{\tilde{u}_n}^*(y^\dag - \tilde{y}_n) + r_n \|_{U}^2 \\
                & \leq 2w_n(-1+ \mu)\|y^\dag - \tilde{y}_n \|_{Y}^2 + 2w_n^2\| G_{\tilde{u}_n}^*(y^\dag - \tilde{y}_n) \|_{U}^2+ 2w_n^2\| r_n \|_{U}^2 \\
                & \leq 2w_n(-1+ \mu)\|y^\dag - \tilde{y}_n \|_{Y} ^2 + 2w_n^2L^2\| y^\dag - \tilde{y}_n \|_{Y}^2 + 8w_n^2 L^2  \| y^\dag - \tilde{y}_n \|_{Y}^2  \\
                &\leq 2w_n \|y^\dag - \tilde{y}_n \|_{Y} ^2 \left[ -1 + \mu + 5\Lambda L^2 \right] 
            \end{aligned}
        \end{multlined}
    \end{equation} 
    for all $n \geq 0$. Consequently, 
    \begin{equation} \label{sum-yn}
        \sum_{n \geq 0} \| y^\dag - \tilde{y}_n \|_{Y}^2 \leq \frac{1}{2\lambda \left(1- \mu - 5 \Lambda L^2\right)} \| u_{0} - u^\dag \|_{U}^2 <\infty.
    \end{equation}
    The inequality \eqref{bound-aux1} also yields that $\{\|\tilde e_n \|_{U}\}_{n \in \N}$ with $\tilde e_n := u^\dag -\tilde u_n$ is monotonically decreasing, and hence $\lim_{n \to \infty} \|\tilde e_n \|_{U} = {\tilde \gamma}$
    for some $\tilde{\gamma} \geq 0$. 

    For any $m,l \in \N$ with $m \leq l$, we now choose
    \begin{equation} \label{choice-r}
        k \in \arg\min_{m\leq t\leq l}\| y^\dag - \tilde y_t \|_{Y}.
    \end{equation}
    As in \eqref{eq:ine-43-free}, it holds that
    \begin{equation}\label{eq:ine-43}
        \begin{aligned}[t]
            \|\tilde u_m -\tilde u_l \|_{U}^2 & \leq  \tilde a_{m,l,k} + \tilde b_{m,l,k} 
        \end{aligned}
    \end{equation}
    with
    \begin{align} \label{eq:lim-431}
        \tilde a_{m,l,k} &:= 2 \left[\|\tilde e_m \|_{U}^2 + \|\tilde e_l \|_{U}^2 - 2\|\tilde e_k \|_{U}^2 \right] \to 0 \quad \text{as } l\geq k\geq m \to \infty 
        \shortintertext{and}
        \tilde b_{m,l,k} &:= 4\left(\tilde e_k-\tilde e_m, \tilde e_k \right)_{U} + 4 \left(\tilde e_k -\tilde e_l, \tilde e_k \right)_{U}.     
        \nonumber
    \end{align}
    Furthermore,
    \begin{equation}\label{eq:ine-44}
        \left(\tilde e_{k}-\tilde e_m, \tilde e_k \right)_{U}  = \sum_{n=m}^{k-1} \left(\tilde e_{n+1}-\tilde e_n, \tilde e_k \right)_{U}        \leq \sum_{n=m}^{k-1} |\left(\tilde e_{n+1}-\tilde e_n, \tilde e_k \right)_{U}|. 
    \end{equation}
    From \eqref{eq:Land-aux}, we obtain $\tilde e_{n+1} - \tilde e_{n} = -w_n G_{\tilde u_n}^*(y^\dag - \tilde y_n) - w_n r_n$ and hence
    \begin{equation*}
        \begin{aligned}
            \left(\tilde e_{n+1}-\tilde e_n, \tilde e_k \right)_{U} & = -w_n\left(y^\dag - \tilde y_n,G_{\tilde u_n} \tilde e_k \right)_{Y} - w_n\left(r_n, \tilde e_k \right)_{U} \\
                                                                    & = w_n\left(y^\dag - \tilde y_n,G_{\tilde u_n} (\tilde u_k - u^\dag) \right)_{Y} + w_n\left(r_n, \tilde u_k - u^\dag \right)_{U}. 
        \end{aligned}
    \end{equation*}
    It follows that 
    \begin{equation} \label{eq:ine-45}
        | \left(\tilde e_{n+1}-\tilde e_n, \tilde e_k \right)_{U} | \leq w_n \| y^\dag - \tilde y_n \|_{Y} \| G_{\tilde u_n} (\tilde u_k - u^\dag)\|_{Y}+ w_n \left| \left(r_n, \tilde u_k - u^\dag \right)_{U}\right|,
    \end{equation}
    and proceeding as in the proof of estimate \eqref{eq:ine-46-free} shows that
    \begin{equation}
        \label{eq:ine-46} 
        \| G_{\tilde u_n}(\tilde u_k - u^\dag ) \|_{Y} \leq 3(1+ \mu) \| y^\dag - \tilde y_n \|_{Y}.
    \end{equation} 
    On the other hand, assertion \ref{it:asym-stab3} of \cref{lem:asym-stab2} implies that
    \begin{equation*}
        \begin{aligned}[t]
            \left| \left( r_n, \tilde{u}_k - u^\dag \right)_{U} \right | 
            & \leq 2(1+\mu)\|y^\dag - \tilde{y}_n \|_{Y} \left[\|y^\dag - \tilde{y}_n \|_{Y} +  \|\tilde{y}_k - \tilde{y}_n \|_{Y} \right] \\
            & \leq 2(1+\mu)\|y^\dag - \tilde{y}_n \|_{Y} \left[\|y^\dag - \tilde{y}_n \|_{Y} \right.\\
            \MoveEqLeft[-9] \left. + \|\tilde{y}_k - y^\dag\|_{Y} +\|\tilde{y}_n - y^\dag\| _{Y} \right] \\
            & \leq 6(1+\mu) \|y^\dag - \tilde{y}_n \|_{Y}^2 .
        \end{aligned}
    \end{equation*}
    In combination with \eqref{eq:ine-45} and \eqref{eq:ine-46}, we obtain that
    \begin{equation*}
        | \left(\tilde e_{n+1}-\tilde e_n, \tilde e_k \right)_{U} | \leq 9(1+\mu) w_n \| y^\dag - \tilde y_n \|_{Y}^2, 
    \end{equation*}
    which together with \eqref{eq:ine-44} ensures that 
    \begin{equation*}
        | \left(\tilde e_k - \tilde e_m, \tilde e_k \right)_{U} |  \leq 9(1+\mu) \sum_{n=m}^{k-1} w_n{\| y^\dag - \tilde y_n \|_{Y}^2} 
        \leq 9(1+\mu) \Lambda \sum_{n=m}^{k-1}{\| y^\dag - \tilde y_n \|_{Y}^2}.
    \end{equation*}
    Similarly, 
    \begin{equation*}
        | \left(\tilde e_k - \tilde e_l, \tilde e_k \right)_{U} | \leq 9(1+\mu) \Lambda \sum_{n=k}^{l-1}{\| y^\dag - \tilde y_n \|_{Y}^2}. 
    \end{equation*}
    We therefore obtain that 
    \begin{equation*}
        \tilde b_{m,l,k} = 4\left(\tilde e_k-\tilde e_m, \tilde e_k \right)_{U} + 4 \left(\tilde e_k -\tilde e_l, \tilde e_k \right)_{U} 
        \leq 36(1+\mu) \Lambda \sum_{n=m}^{l-1}{\| y^\dag - \tilde y_n \|_{Y}^2},
    \end{equation*}
    which together with \eqref{sum-yn} yields that
    \begin{equation}
        \label{eq:lim-432} 
        \tilde b_{m,l,k} \to 0 \quad \text{as } l\geq k\geq m \to \infty. 
    \end{equation}
    From \eqref{eq:lim-431}, \eqref{eq:lim-432}, and \eqref{eq:ine-43}, we now obtain that $\{\tilde u_n\}_{n \in \N}$ is a Cauchy sequence in $U$.
    Thus, there exists a $\bar u \in \overline  B_{U}(u^\dag, \rho)$ such that $\tilde u_n \to \bar u$ and thus $F(\tilde u_n) \to F(\bar u)$ by \cref{ass:cont} as $n \to \infty$.
    Now \eqref{sum-yn} implies that $y^\dag - F(\tilde u_n) \to 0$ as $n \to \infty$. Hence, $y^\dag = F(\bar u)$ and therefore $\bar u \in S(u^\dag , \rho)$, which completes the proof.
\end{proof}

We have thus shown the following result.
\begin{corollary}\label{cor:asy-stable}
    Under \crefrange{ass:cont}{ass:adj_bounded}, the modified Landweber iteration \eqref{eq:modified-landweber} stopped according to the discrepancy principle \eqref{eq:discrepancy} for $\tau>1$ 
    is asymptotically stable for any starting point $u^0\in \overline B_U(u^\dag,\rho)$ and any step sizes $\{w_n\}_{n\in\N}\subset [\lambda,\Lambda]$ for  $\Lambda\geq \lambda > 0$ satisfying \eqref{choice} as well as \eqref{choice-aux}.
\end{corollary}

We are now well prepared to prove our main result.
\begin{theorem}\label{thm:REG}
    Let \crefrange{ass:cont}{ass:adj_bounded} hold and $\tau>1$ and $\Lambda\geq \lambda > 0$ satisfy conditions \eqref{choice} as well as \eqref{choice-aux}. Assume further that $\{\delta_k\}_{k \in \N}$ is a positive zero sequence.
    Let the starting point $u_0 \in \overline  B_{U}(u^\dag, \rho)$ and the step sizes $w_n \in [\lambda, \Lambda]$ be arbitrary and let the stopping index $N_k$ be chosen according to the discrepancy principle \eqref{eq:discrepancy}.Then, any subsequence of $\{u_{N_k}^{\delta_k}\}_{k \in \N}$ contains a subsequence that converges strongly to an element of $S(u^\dag ,\rho)$. Furthermore, if $u^\dag$ is the unique solution of \eqref{eq:general-inverse}, then $u_{N_k}^{\delta_k} \to u^\dag$ in $U$ as $k\to\infty$.
\end{theorem}
\begin{proof}
    Let $\{\delta_{k_i}\}_{i \in \N}$ be an arbitrary subsequence of $\{\delta_k\}_{k \in \N}$. By virtue of \cref{cor:asy-stable}, there exist a sequence $\{ \tilde{u}_n \} \subset \overline{B}_U(u^\dag ,\rho)$ and a subsequence of $\{\delta_{k_i}\}_{i \in \N}$, denoted in the same way, satisfying conditions \ref{it:AS1}--\ref{it:AS2} in \cref{def:AS}.

    Assume first that  $\lim_{i \to \infty} N_{k_i} = N$ for some $N \in \N$. From condition \ref{it:AS1} of \cref{def:AS}, we then have 
    \begin{equation}\label{eq:REG-limit1}
        u_N^{\delta_{k_i}} \to \tilde u_N \quad \text{as } i \to \infty.
    \end{equation}
    Furthermore, we see from the discrepancy principle that
    \begin{equation*} 
        \| y^{\delta_{k_i}} - F(u_{N}^{\delta_{k_i}}) \|_{Y} \leq \tau \delta_{k_i} \quad \text {for all $i$  large enough}.
    \end{equation*}
    Letting $i \to \infty$ in the above estimate and using \eqref{eq:REG-limit1} together with the continuity of $F$ yields that $y^\dag = F(\tilde u_N)$ and hence $\tilde u_N \in S(u^\dag ,\rho)$. 

    It remains to consider the case where $N_{k_i} \to \infty$ as $i \to \infty$. Since $\{N_{k}\}_{k\in\N}\subset\N$, we can assume without loss of generality that $\{N_{k_i} \}_{i \in \N}$ is monotonically increasing. 
    Condition \ref{it:AS2} of \cref{def:AS} then provides some $\tilde{u} \in S(u^\dag ,\rho)$ that together with $\{\delta_{k_i}\}_{i \in \N}$ and $ \{ \tilde{u}_n \}_{n \in \N}$ satisfies
    \begin{align}
        & u_n^{\delta_{k_i}} \to \tilde u_n \quad \text {as } i \to \infty, \quad \text {for $0 \leq n \leq N_{k_i}$ with all $i$ large enough} \label{eq:conver1} \\
        & \tilde{u}_n \to \tilde{u} \quad \text {as } n \to \infty.  \label{eq:conver2}
    \end{align}
    From \eqref{eq:conver2}, for each $\epsilon >0$, there exists an integer $n^*$ such that 
    \begin{equation*}
        \|\tilde u_{n^*} - \tilde{u} \|_{U} < \frac{\epsilon}{2}.
    \end{equation*}
    It also follows from \eqref{eq:conver1} and the fact $N_{k_i}$ tends increasingly to infinity as $i \to \infty$ that an $\bar i \in \N$ exists such that
    \begin{equation*} 
        n^* \leq N_{k_i} \quad \text {and} \quad \|  u_{n^*}^{\delta_{k_i}} - \tilde u_{n^*} \|_{U} < \frac{\epsilon}{2}\quad \text{for all } i \geq \bar i.
    \end{equation*}     
    \cref{lem:aux} thus implies that
    \begin{equation*}
        \| u_{N_{k_i}}^{\delta_{k_i}} - \tilde{u} \|_{U}  \leq \| u_{n^*}^{{\delta_{k_i}}} - \tilde{u} \|_{U} \leq \| u_{n^*}^{{\delta_{k_i}}} -\tilde u_{n^*} \|_{U} + \| \tilde u_{n^*} - \tilde{u} \|_{U} < \epsilon \quad \text{for all }  i \geq \bar i.
    \end{equation*} 
    We thus obtain that $\lim_{i \to \infty}\| u_{N_{k_i}}^{\delta_{k_i} } - \tilde{u} \|_{U} =0$ as claimed.
\end{proof}

\section{Iterative regularization for a non-smooth forward operator} \label{sec:aux}

In this section, we study the solution operator for our model problem. In particular, we show that a Bouligand subderivative of the solution operator satisfies the assumptions -- in particular, the generalized tangential cone condition \eqref{eq:GTCC} -- for our convergence analysis of the modified Landweber method in \cref{sec:Alg}, thus justifying our \emph{Bouligand--Landweber method}.

\subsection{Well-posedness and directional differentiability}

Let $\Omega\subset\R^d$, $2 \leq d \leq 3$, be a bounded domain with Lipschitz boundary $\partial\Omega$. For $u\in L^2(\Omega)$, we consider the equation
\begin{equation}\label{eq:maxpde}
    \left\{
        \begin{aligned}
            -\Delta y + y^+ &= u \quad\text{in }\Omega,\\
            y  &=0 \quad\text{on }\partial\Omega,
        \end{aligned}
    \right.
\end{equation}
which, as all partial differential equations from here on, is to be understood in the weak sense.
From \cite[Thm.~4.7]{Troltzsch}, equation \eqref{eq:maxpde} admits, for each $u \in L^2(\Omega)$, a unique solution $y_u$ belonging to $H^1_0(\Omega) \cap C(\overline{\Omega})$ and satisfying the a priori estimate
\begin{equation*}
    \norm{y_u}_{H^1_0(\Omega)} + \norm{y_u}_{C(\overline{\Omega})} \leq c_\infty \norm{u}_{L^2(\Omega)}
\end{equation*} 
for some constant $c_\infty>0$ independent of $u$.

Let us denote by $F: L^2(\Omega) \to H^1_0(\Omega) \cap C(\overline{\Omega})$ the solution operator of \eqref{eq:maxpde}. 
The global Lipschitz continuity of $F$ is established by the following proposition.
\begin{proposition}[{\cite[Prop.~2.1]{Christof2017}}] \label{prop:well-posed}
    $F$ is globally Lipschitz continuous as a function from $L^2(\Omega)$ to $ H^1_0(\Omega) \cap C(\overline{\Omega})$, i.e., 
    there is a constant $C_F>0$ satisfying 
    \begin{align}
        \| F(u)- F(v)\|_{H^1_0(\Omega)} + \| F(u)- F(v)\|_{C(\overline\Omega)} &\leq C_F \| u -v \|_{L^2(\Omega)} \label{eq:F-Lip}
    \end{align} 
    for all $u,v \in L^2(\Omega)$. 
\end{proposition}
\begin{proof}
    Let us set $y_u := F(u)$ and $y_v := F(v)$. By subtracting equation \eqref{eq:maxpde} corresponding to $y_u$ and $y_v$, we have 
    \begin{equation}
        \label{eq:subtract_pde}
        \left\{
            \begin{aligned}
                -\Delta (y_u - y_v) + (y_u)^+ - (y_v)^+ &= u -v \quad && \text{in }\Omega,\\
                y_u - y_v  &=0 \quad && \text{on }\partial\Omega.
            \end{aligned}
        \right.
    \end{equation}
    Testing the above equation with $y_u -y_v$ and exploiting the monotonicity of the max-operator, we arrive at
    \begin{equation*}
        \norm{\nabla y_u - \nabla y_v}_{L^2(\Omega)}^2 \leq \norm{u -v}_{L^2(\Omega)}\norm{y_u -y_v}_{L^2(\Omega)},
    \end{equation*}
    which together with the Poincar\'{e} inequality yields that
    \begin{equation}
        \label{eq:esti}
        \norm{y_u - y_v}_{L^2(\Omega)} \leq C_1 \norm{u -v}_{L^2(\Omega)}
    \end{equation}
    for some constant $C_1>0$. We now apply \cite[Thm.~4.7]{Troltzsch} to equation \eqref{eq:subtract_pde} to obtain that
    \begin{equation*}
        \begin{aligned}
            \norm{y_u - y_v}_{H^1_0(\Omega)} + \norm{y_u - y_v}_{C(\overline{\Omega})} & \leq C_2 \norm{u - v - \left((y_u)^+ - (y_v)^+ \right)}_{L^2(\Omega)} \\
                                                                                       & \leq C_2 \left[ \norm{u - v }_{L^2(\Omega)} + \norm{(y_u)^+ - (y_v)^+}_{L^2(\Omega)} \right]  \\
                                                                                       & \leq C_2 \left[\norm{u - v }_{L^2(\Omega)} + \norm{y_u - y_v }_{L^2(\Omega)} \right].
        \end{aligned}
    \end{equation*}
    Here we have used the global Lipschitz continuity of the max-operator to derive the last inequality. From this and the estimate \eqref{eq:esti}, we deduce \eqref{eq:F-Lip}.
\end{proof}

This implies \emph{a fortiori} that $F$ is continuous from $L^2(\Omega)$ to $L^2(\Omega)$. In our analysis, we will also need the complete continuity of $F$ between these spaces.
\begin{lemma} \label{lem:weak-con} 
    The mapping $F: L^2(\Omega) \to L^2(\Omega)$ is completely continuous, i.e., $u_n\wkto u$ implies $F(u_n)\to F(u)$.
\end{lemma}
\begin{proof}
    From {\cite[Cor. 3.8]{Christof2017}}, we obtain that $F$ is weakly continuous from $L^2(\Omega)$ to $ H^1_0(\Omega)$. The compact embedding $H^1_0(\Omega) \hookrightarrow L^2(\Omega)$ then yields that $F: L^2(\Omega) \to L^2(\Omega)$ is completely continuous.
\end{proof}

We now turn to the differentiability of the solution mapping. We first recall that $F$ is directionally differentiable.
\begin{proposition}[{\cite[Thm.~2.2]{Christof2017}}]
    For any $u\in L^2(\Omega)$ and $h\in L^2(\Omega)$, the mapping $F:L^2(\Omega)\to H^1_0(\Omega)$ is directionally differentiable, with the directional derivative $F'(u; h)$ in direction $h\in L^2(\Omega)$ given by the solution $\eta\in H^1_0(\Omega)$ to
    \begin{equation} \label{eq:deri}
        \left\{
            \begin{aligned}
                - \Delta \eta + \1_{\{y_u = 0\}}\eta^+ + \1_{\{y_u >0\}}\eta &=h \quad \text{in }\Omega,\\
                \eta &=0\quad\text{on }\partial\Omega,
            \end{aligned}
        \right.
    \end{equation}
    where $y_u = F(u)$.
\end{proposition}
However, $F$ is in general not Gâteaux differentiable.
\begin{proposition}\label{prop:gateaux}
    Let $u\in L^2(\Omega)$. Then $F: L^2(\Omega) \to L^2(\Omega)$ is Gâteaux differentiable in $u$ if and only if $|\{y_u=0\}|=0$.
\end{proposition}
\begin{proof}
    Assume that $|\{y_u=0\}|=0$. Then by virtue of {\cite[Cor.~2.3]{Christof2017}}, $F: L^2(\Omega) \to  H^1_0(\Omega)$ is Gâteaux differentiable in $u$. Since $H^1_0(\Omega) \hookrightarrow L^2(\Omega)$ continuously, $F$ is Gâteaux differentiable in $u$ as a function from $L^2(\Omega)$ to $L^2(\Omega)$. It remains to prove that Gâteaux differentiability of $F: L^2(\Omega) \to L^2(\Omega)$ in  $u$ implies that $|\{y_u=0\}|=0$. First, there exists a bounded operator $S: L^2(\Omega) \to L^2(\Omega)$ such that
    \begin{equation} \label{eq:deri1}
        \frac{F(u +th) - F(u)}{t}\to Sh \quad \text{in } L^2(\Omega)\quad \text{as }t\to 0^+
    \end{equation}
    for any $h \in L^2(\Omega)$. Moreover, the right hand side of \eqref{eq:deri1} tends to $F'(u;h)$ in $ H^1_0(\Omega)$ and so in $L^2(\Omega)$ whenever $t \to 0^+$. It must hold that $S = F'(u; \cdot)$ and thus $F'(u;h) = -F'(u;-h)$ for any $h \in L^2(\Omega)$. 
    Fixing $h \in L^2(\Omega)$ and setting $\eta := F'(u;h)$, we see from \eqref{eq:deri} that
    \begin{equation} \label{eq:deri2}
        \left\{
            \begin{aligned}
                - \Delta \eta + \1_{\{y_u = 0\}}\eta^+ + \1_{\{y_u >0\}}\eta &=h \quad \text{in }\Omega,\\
                \eta &=0\quad\text{on }\partial\Omega.
            \end{aligned}
        \right.
    \end{equation}
    Since $-\eta = F'(u;-h)$, it also holds that
    \begin{equation} \label{eq:deri3}
        \left\{
            \begin{aligned}
                - \Delta (-\eta) + \1_{\{y_u = 0\}}(-\eta)^+ + \1_{\{y_u >0\}}(-\eta)  &= -h \quad \text{in }\Omega,\\
                -\eta &=0\quad\text{on }\partial\Omega.
            \end{aligned}
        \right.
    \end{equation}    
    Adding \eqref{eq:deri2} and \eqref{eq:deri3}, we obtain that $\1_{\{y_u = 0\}}\eta^+ + \1_{\{y_u = 0\}}(-\eta)^+ = 0$ a.e. in $\Omega$, which is equivalent to
    \begin{equation} \label{eq:test}
        \1_{\{y_u = 0\}}|F'(u;h)| = \1_{\{y_u = 0\}}|\eta| = 0.
    \end{equation} 
    Now by \cite[Lem.~A.1]{Christof2017}, there exists a function $\psi \in C^\infty(\R^d)$ satisfying $\psi >0$ in $\Omega$ and $\psi =0$ in $\R^d \backslash \Omega$. Setting
    \begin{equation*}
        h : = - \Delta \psi + \1_{\{y_u \geq 0\}} \psi \in L^2(\Omega),
    \end{equation*}
    we then have $F'(u;h) = \psi$. Plugging this into \eqref{eq:test} yields $\1_{\{y_u = 0\}}\psi = 0$. Consequently, we have $|\{y_u=0\}|=0$ as claimed.
\end{proof}

The directional derivative is difficult to exploit algorithmically. A more convenient object can be constructed using the Bouligand subdifferential, which also arises in the definition of the Clarke subdifferential \cite{Clarke:1990} (as the convex hull of the Bouligand subdifferential) and is used in the construction of semi-smooth Newton methods \cite{Chen:2000a,Ulbrich2011} (as a set of candidates for slant or Newton derivatives). We first define the set of \emph{Gâteaux points} of $F$ as
\begin{equation*}
    D := \{ v\in L^2(\Omega) : \text{$F: L^2(\Omega) \to  H^1_0(\Omega)$ is G\^ateaux differentiable in $v$}\}.
\end{equation*} 
The (strong-strong) \emph{Bouligand subdifferential} at $u\in L^2(\Omega)$ is then defined as
\begin{align*} 
    \partial_{B} F(u) 
    := \{ & G_u \in \Linop(L^2(\Omega), H^1_0(\Omega)) : 
        \text{there exists }  \{u_n\}_{n\in\N} \subset D \text{ such that}\\
        & u_n \to u \text{ in } L^2(\Omega)
    \text{ and }  F'(u_n;h) \to G_u\,h \text{ in }  H^1_0(\Omega)  \text{ for all } h \in L^2(\Omega)\}.
\end{align*}
From the definition and the Lipschitz continuity of $F$, it follows that any $G_u\in \partial_BF(u)$ is uniformly bounded for all $u\in L^2(\Omega)$ and that if $F$ is Gâteaux differentiable in $u$, then $F'(u)\in \partial_B F(u)$; cf.~\cite[Lem.~3.3]{Christof2017}. In particular, we deduce that there exist constants $L$ and $\hat L$ satisfying
\begin{equation}\label{eq:Gu-norm}
    \| G_u \|_{\Linop(L^2(\Omega), L^2(\Omega))} \leq L, \quad \| G_u\|_{\Linop(L^2(\Omega), H^1_0(\Omega))} \leq \hat L. 
\end{equation}

We can give a convenient characterization of a specific Bouligand subderivative of $F$.
\begin{proposition}[{\cite[Prop.~3.16]{Christof2017}}] \label{prop:Gu}
    Given $u \in L^2(\Omega)$, let $G_u: L^2(\Omega) \to H^1_0(\Omega) \hookrightarrow L^2(\Omega)$ be the solution operator mapping $h\in L^2(\Omega)$ to the unique solution $\eta\in  H^1_0(\Omega)$ to
    \begin{equation} \label{eq:bouligand_pde}
        \left\{
            \begin{aligned}
                - \Delta \eta + \1_{\{y_u>0\}}\eta &= h \quad \text {in } \Omega,\\
                \eta &= 0 \quad \text {on } \partial\Omega,
            \end{aligned}
        \right.
    \end{equation}
    where $y_u:=F(u)$. Then $G_u \in \partial_{B} F(u)$.
\end{proposition} 
\begin{remark}
    We refer to \cite[Thm.~3.18]{Christof2017} for a precise characterization of the full Bouligand subdifferential. By replacing one or both convergences with the corresponding weak convergence, we further arrive at different variants of the Bouligand subdifferential; see \cite[Sec.~3.1]{Christof2017} for the precise definitions and the relations between them. For our purposes, however, the strong notion suffices. Furthermore, although the results in this section also hold for arbitrary elements from these weaker notions of the Bouligand subdifferential as well as for slant derivatives, there is no obvious benefit of these choices in our context, and we thus restrict ourselves to \eqref{eq:bouligand_pde} to keep the presentation concise.
\end{remark}
Clearly, $G_u$ is a self-adjoint operator when considered acting from $L^2(\Omega)$ to $L^2(\Omega)$. Furthermore, for this specific choice of the subderivative, we can derive an $L^p$ version of the estimates \eqref{eq:Gu-norm} that will be needed in the following.
\begin{lemma} \label{lem:lin_apriori}
    Let $\frac{d}{2} < p \leq 2$. Then there exists a constant $L_p>0$ such that
    \begin{equation}\label{eq:lin2}
        \| G_u \|_{\Linop(L^p(\Omega), C(\overline\Omega))} \leq L_p\quad\text{for all }u\in U.
    \end{equation}
\end{lemma}
\begin{proof}
    Let $h \in L^p(\Omega)$ with $\frac{d}{2} < p \leq 2$ and $u\in U$ be arbitrary. From \cref{prop:Gu}, we have that $\eta = G_uh$ satisfies
    \begin{equation*}
        \left\{
            \begin{aligned}
                - \Delta \eta + a\eta &= h \quad \text {in } \Omega,\\
                \eta &= 0 \quad \text {on } \partial\Omega
            \end{aligned}
        \right.
    \end{equation*}
    for some $a \in L^\infty(\Omega)$ with $0 \leq a(x) \leq 1$ for a.e. $x \in \Omega$. 
    Stampacchia's theorem \cite[Thm.~12.4]{Chipot} and \cite[Thm.~4.7]{Troltzsch} thus ensure $\eta \in C(\overline{\Omega}) \cap H^{1}_0(\Omega)$ and satisfies
    \begin{equation}
        \| \eta \|_{C(\overline{\Omega})} \leq L_p \| h\|_{L^p(\Omega)} 
    \end{equation}
    for some constant $L_p$ independent of $a$ and $h$, i.e., \eqref{eq:lin2}.
\end{proof}

Finally, the following example shows that the mapping $u\mapsto G_{u}h$ is in general not continuous, which is the main difficulty in showing convergence of a modified Landweber method.
\begin{example} \label{ex:discon} 
    Let $\Omega=\{x\in\R^2: |x|_2 \leq 1\}$ be the unit ball in $\R^2$. For each $\epsilon >0$, we set 
    \begin{equation*}
        u_\epsilon(x) := \epsilon \left( 5 - x_1^2 - x_2^2 \right).
    \end{equation*} 
    Then $u_\epsilon$ tends to $\bar u := 0$ as $\epsilon \to 0^+$. Furthermore, we have $y_\epsilon(x) = F(u_\epsilon)(x) = \epsilon (1- x_1^2 -x_2^2) >0$ for all $x=(x_1,x_2) \in \Omega$. It follows that $\1_{\{y_\epsilon >0\}}(x) = 1$ almost everywhere in $\Omega$, and hence $G_{u_\epsilon}\equiv G$ 
    for the operator 
    $G: L^2(\Omega) \to L^2(\Omega)$ defined by $z:=Gh$ being a unique solution to 
    \begin{equation*}
        \left\{
            \begin{aligned}
                -\Delta z + z &= h \quad \text{in } \Omega,\\
                z &= 0\quad\text{on }\partial\Omega.
            \end{aligned}
        \right.
    \end{equation*}
    On the other hand, $\bar z := G_{\bar u}h$ satisfies 
    \begin{equation*}
        \left\{
            \begin{aligned}
                -\Delta \bar z &= h \quad \text{in } \Omega,\\
                \bar z &= 0\quad\text{on }\partial\Omega
            \end{aligned}
        \right.
    \end{equation*} 
    for any $h\in L^2(\Omega)$.
    We thus have $z \neq \bar z$ whenever $h \neq 0$. Therefore, if $h\neq 0$, 
    \begin{equation*}
        G_{u_\epsilon}h \nrightarrow G_{\bar u}h \quad \text{as } \epsilon \to 0^+ .
    \end{equation*} 
\end{example}

\subsection{Generalized tangential cone condition}

We now verify that the solution mapping for our example satisfies the generalized tangential cone condition \eqref{eq:GTCC}. We begin with a crucial lemma deriving a ``pointwise'' tangential cone condition.
\begin{lemma} \label{lem:gtcc-key}
    Let $u, \hat u \in L^2(\Omega)$ and $\frac{d}{2} < p <2$. Then, one has 
    \begin{equation*}
        \|F(\hat u) - F(u) - G_u(\hat u - u) \|_{L^2(\Omega)} \leq L_p |\Omega|^{1/2}M(u,\hat u)^{1/p'} \|F(\hat u) - F(u) \|_{L^2(\Omega)}
    \end{equation*} 
    with $p' = \frac{2p}{2-p}$, $L_p$ as in \cref{lem:lin_apriori}, and
    \begin{equation*}
        M(u,\hat u) := \left| \{ y_u \leq 0, y_{\hat u} >0 \} \cup \{ y_u > 0, y_{\hat u} \leq 0 \} \right|.
    \end{equation*}     
\end{lemma}
\begin{proof}
    Setting $y := y_u$, $\hat y:= y_{\hat u}$, $\zeta := G_u(\hat u - u)$, and $\omega := \hat y - y - \zeta$, we have from the definitions that
    \begin{align*}
        - \Delta \hat y + {\hat y}^+ &= \hat u,\\
        - \Delta y + y^+ &= u,\\
        - \Delta \zeta + \1_{\{y>0\}} \zeta &= \hat u - u.
    \end{align*}
    This implies that 
    \begin{equation*}
        -\Delta \omega + \1_{\{y>0\}} \omega = \left( \1_{\{y>0\} }- \1_{ \{\hat y>0\} } \right)\hat y.
    \end{equation*}
    By simple computation, it follows that
    \begin{equation*}
        a := \left(\1_{\{y>0\} }- \1_{ \{\hat y>0\} } \right)\hat y  = \left(\1_{\{y>0, \hat y \leq 0\} }- \1_{ \{y \leq 0, \hat y>0\} } \right)\hat y 
    \end{equation*}
    and hence
    \begin{equation*}
        0 \geq a \geq \left( \1_{\{y>0, \hat y \leq 0\} }- \1_{ \{y \leq 0, \hat y>0\} } \right)(\hat y - y).
    \end{equation*}
    Consequently, 
    \begin{equation*}
        |a(x)| \leq |e(x)||\hat y(x) - y(x)| \quad \text{for a.e.} \ x \in \Omega
    \end{equation*}
    with 
    \begin{equation*}
        e: = \left( \1_{\{y>0, \hat y \leq 0\} }- \1_{ \{y \leq 0, \hat y>0\} } \right).
    \end{equation*}
    From this, \cref{lem:lin_apriori}, and the H\"{o}lder inequality, we obtain
    \begin{equation*}
        \begin{aligned}
            \|\omega \|_{C(\overline{\Omega})} & \leq L_p \|a\|_{L^p(\Omega)} \\
                                               & \leq L_p \|\hat y - y \|_{L^2(\Omega)} \|e\|_{L^{p'}(\Omega)} \\
                                               & \leq L_p M(u,\hat u)^{1/p'} \|\hat y - y \|_{L^2(\Omega)}. 
        \end{aligned}
    \end{equation*}
    This together with the inequality $\|\omega\|_{L^2(\Omega)} \leq \|\omega \|_{C(\overline{\Omega})} |\Omega|^{1/2}$ implies the desired estimate.
\end{proof}

The following result verifies that the solution mapping satisfies a generalized tangential cone condition, which is close to the classical tangential cone condition \cite{Scherzer1995,Hanke1995,Margotti2015} and will be crucial in the convergence analysis of the following section. 
\begin{proposition} \label{cor:TCC} 
    Let $\bar u\in L^2(\Omega)$ and $\mu>0$ and assume that 
    \begin{equation} \label{eq:cond-meas}
        L_p |\Omega|^{1/2} \left(2|\{F(\bar u) = 0\}| \right)^{1/p'} < \mu
    \end{equation} 
    with $p' : =\frac{2p}{2-p}$ and $L_p$ as in \cref{lem:lin_apriori}.
    Then there exists $\rho >0$ such that \eqref{eq:GTCC} holds in $\bar u$ for $\rho$ and $\mu$.
\end{proposition}
\begin{proof}
    Set $\bar y = F(\bar u)$ and let $\rho>0$ be arbitrary. Due to \cref{prop:well-posed}, we then have for any $u \in \overline B_{L^2(\Omega)}(\bar u, \rho)$ that
    \begin{equation*}
        \| \bar y - y_u \|_{C(\overline{\Omega})} \leq  C_F \|\bar u - u \|_{L^2(\Omega)} \leq  C_F\rho =:\epsilon.
    \end{equation*}
    Hence, for any $u \in \overline B_{L^2(\Omega)}(\bar u, \rho)$, it follows that
    \begin{equation*}
        -\epsilon + y_u(x) \leq \bar y \leq \epsilon + y_u(x) 
    \end{equation*} for all $x \in \bar\Omega$. 
    This implies for any $u, \hat u \in \overline B_{L^2(\Omega)}(\bar u, \rho)$ that 
    \begin{align*}
        & \{y_{u}>0, y_{\hat u} \leq 0\} \subset \{-\epsilon \leq \bar y \leq \epsilon\},\\
        & \{y_{u} \leq 0, y_{\hat u} > 0\} \subset \{-\epsilon \leq \bar y \leq \epsilon\},
    \end{align*}
    and thus 
    \begin{equation*}
        M(u,  \hat u) \leq 2|\{0 \leq |\bar y| \leq \epsilon\}|.
    \end{equation*}
    From condition \eqref{eq:cond-meas}, we have
    \begin{equation*}
        \lim_{\epsilon \to 0^+} L_p |\Omega|^{1/2} \left(2|\{0 \leq |\bar y| \leq \epsilon\} | \right)^{1/p'} <\mu.
    \end{equation*}
    Note that $\epsilon \to 0^+$ as $\rho \to 0^+$. Hence, choosing $\rho>0$ small enough such that
    \begin{equation*}
        L_p |\Omega |^{1/2} \left(2|\{0 \leq |\bar y| \leq  \epsilon\}|\right)^{1/p'} \leq \mu
    \end{equation*}
    yields that 
    \begin{equation*}
        \|F(\hat u) - F(u) - G_u(\hat u - u) \|_{L^2(\Omega)} \leq \mu \|F(\hat u) - F(u) \|_{L^2(\Omega)}
    \end{equation*}
    for all $u, \hat u \in \overline B_{L^2(\Omega)}(\bar u, \rho)$.
\end{proof}
The condition \eqref{eq:cond-meas} is related to -- but weaker than -- the active set condition introduced in \cite{Wachsmuth:2011a,Wachsmuth:2011b} in order to derive strong convergence rates for the Tikhonov regularization of singular and non-smooth optimal control problems. We stress that the condition \eqref{eq:cond-meas} does not require that $F$ is differentiable at the exact solution $u^\dag$.

\subsection{Bouligand--Landweber iteration}

The results obtained so far show that the solution mapping $F$ to \eqref{eq:maxpde} together with $u\mapsto G_u$ with $G_u$ the Bouligand subderivative given in \cref{prop:Gu} satisfies the assumptions of \cref{sec:Alg}, provided that condition \eqref{eq:cond-meas} is valid. We also note that in this case $F$ is injective, i.e., $u^\dag$ is the unique solution to \eqref{eq:general-inverse}. 
We can thus use $G_u$ in the modified Landweber iteration \eqref{eq:modified-landweber} to obtain a convergent \emph{Bouligand--Landweber method} for the iterative regularization of the non-smooth ill-posed problem $F(u)=y$. 
\begin{corollary}\label{cor:bouligandlandweber}
    Assume that \eqref{eq:cond-meas} holds for $u^\dag\in L^2(\Omega)$. Then there exists $\rho>0$ and $0 < \lambda\leq \Lambda$ such that for all starting points $u^0\in \overline B(u^\dag,\rho)$ and step sizes $w_n\in[\lambda,\Lambda]$, the Bouligand--Landweber iteration \eqref{eq:modified-landweber} stopped according to the discrepancy principle \eqref{eq:discrepancy} is a well-posed and strongly convergent regularization method.
\end{corollary}
\begin{proof}
    We merely have to argue that \crefrange{ass:cont}{ass:adj_bounded} of \cref{sec:Alg} are satisfied. Taking $U = Y = L^2(\Omega)$, \cref{ass:cont,ass:bouligand} follow from \cref{prop:well-posed,lem:weak-con}, respectively, where we can take any $\rho_0>0$ in the latter.
    Under the condition \eqref{eq:cond-meas}, \cref{cor:TCC} guarantees that \cref{ass:gtcc} holds for some choice of $\rho>0$. Finally, \cref{ass:adj_bounded} holds for $Z=H^1_0(\Omega)$ due to the self-adjointness of $G_u$ and \cref{lem:weak-con} again. The claim now follows from \cref{thm:free,thm:REG}.
\end{proof}
We note that if $|\{ y_n^\delta = 0 \}| =0$ with $y_{n}^\delta :=F(u_n^\delta)$ for some $n\in \N$, we obtain from \cref{prop:gateaux} that $F$ is G\^{a}teaux differentiable in $u_n^\delta$ and that $G_{u_n^\delta} = F'(u_n^\delta)$. Hence in this case, the corresponding Bouligand--Landweber step \eqref{eq:modified-landweber} coincides with the classical Landweber step \eqref{eq:landweber_smooth}.

\begin{remark}
    The results of this section -- and hence of this work -- can be extended to the case of piecewise continuously differentiable nonlinearities, i.e., to a forward operator given as the solution mapping to
    \begin{equation} 
        \left\{
            \begin{aligned}
                A y + f(y) &= u \quad \text{in } \Omega,\\
                y &=0\quad \text{on } \partial \Omega,
            \end{aligned}
        \right.
    \end{equation} 
    where $A$ is a second-order strongly uniformly elliptic operator and $f$ is a superposition operator defined by a piecewise continuously differentiable and non-decreasing function; see \cref{app:piecewise-diff}.
\end{remark}


\section{Numerical experiments} \label{sec:Num}

In this section, we present results of numerical experiments illustrating the performance of the Bouligand--Landweber iteration for the model problem \eqref{eq:maxpde-intro}. Although our focus is not on the numerical approximation, we first give a short description of our discretization scheme and the solution of the non-smooth PDE \eqref{eq:maxpde-intro} using a semi-smooth Newton (SSN) method for the sake of completeness. The last subsection then contains numerical examples.

\subsection{Discretization and semi-smooth Newton method}

For the discretization of the non-smooth semilinear elliptic problem \eqref{eq:maxpde} and the generalized linearization equation \eqref{eq:bouligand_pde}, we use standard continuous piecewise linear finite elements (FE), see, e.g., \cite{Knabner-Angermann,Glowinski1984} . From now on, we restrict ourselves to the case $\Omega \subset \R^2$. Denote by $\mathcal{T}_h$ the triangulation of $\Omega$ with the discretization parameter $h$ indicating the maximum length of the edges of all the triangles of $\mathcal{T}_h$. For each triangulation $\mathcal{T}_h$, let $V_h$ be the finite-dimensional subspace of $H^1_0(\Omega)$ consisting of functions whose restrictions to a triangle $T \in \mathcal{T}$ are polynomials of first degree. By $\{\varphi_j\}_{j=1}^n$, we denote the basis of $V_h$ corresponding to the set of nodes $\mathcal{N}_h:=\{p_1,\dots,p_n\}$, i.e., $V_h$ is spanned by functions $\varphi_1,\dots,\varphi_n$ and $\varphi_j(p_i) = \delta_{ji}$ where $(\delta_{ji})_{j,i=1}^n$ is the Kronecker delta. 
Note that for $v_h\in V_h$, we do not necessarily have $v_h^+\in V_h$. We thus use a mass lumping approach to discretize the non-smooth semilinear elliptic equation \eqref{eq:maxpde-intro} in weak form as
\begin{equation}
    \label{dis-ori-P} 
    \int_{\Omega} \nabla y_h \cdot \nabla v_h \,dx + \frac{1}{3}\sum_{T \in \mathcal{T}_h} |T| \sum_{p_i \in \mathcal{N}_h\cap\overline T} \max(y_h(p_i),0)v_h(p_i) = \int_{\Omega} u_hv_h\,dx \quad \text{for all } v_h \in V_h,
\end{equation} 
where $y_h$ and $u_h \in V_h$ denote the FE approximation of $y$ and $u$, respectively, and $\overline{T}$ stands for the closure of $T$ (i.e., the inner sum is over all vertices of the triangle $T$). 
By a slight abuse of notation, from now we on write $y\in\R^n$ and $u\in\R^n$ instead of $(y_h(p_i))_{i=1}^n$ and $(u_h(p_i))_{i=1}^n$, respectively. The discrete equation \eqref{dis-ori-P} is then equivalent to the nonlinear algebraic system
\begin{equation}
    A y + D \max(y,0) = Mu \label{coor-ori-P}
\end{equation} 
with the stiffness matrix $A := \left( (\nabla\varphi_j, \nabla \varphi_i)_{L^2(\Omega)} \right)_{i,j=1}^n$, the mass matrix $M := \left( (\varphi_j, \varphi_i)_{L^2(\Omega)} \right)_{i,j=1}^n$, the lumped mass matrix $D := \frac{1}{3}\diag\left(\omega_1,\dots, \omega_n\right)$, $\omega_i :=|\{\varphi_i\neq 0\}|$, and $\max(\cdot, 0): \R^n \to \R^n$ the componentwise max-function.

Similarly, the equation \eqref{eq:bouligand_pde} characterizing the Bouligand subderivative is discretized as
\begin{equation}\label{dis-linearization1} 
    \int_{\Omega} \nabla \eta_h \cdot \nabla v_h \,dx + \int_{\Omega} \1_{\{y>0\}} \eta_h v_h\,dx = \int_{\Omega} w_hv_h\,dx \quad \text{for all } v_h \in V_h. 
\end{equation}
Here $\eta_h$ and $w_h$ stand for the FE-approximation of $\eta$ and $w$, respectively.
Using the continuity of integrands and the two-dimensional trapezoidal method, the second term on the left hand side of \eqref{dis-linearization1} can be approximated by 
\begin{equation*}
    \frac{1}{3}\sum_{T \in \mathcal{T}_h} |T| \sum_{p_i \in \mathcal{N}_h\cap\overline T \cap \{y>0 \} } \eta_h(p_i) v_h(p_i) = \frac{1}{3}\sum_{T \in \mathcal{T}_h} |T| \sum_{p_i \in \mathcal{N}_h\cap\overline T} \1_{\{y(p_i) >0\}} \eta_h(p_i) v_h(p_i)
\end{equation*} 
for $h$ small enough. 
From this and $y(p_i) = y_h(p_i)$, the discrete equation \eqref{dis-linearization1} can be rewritten as 
\begin{equation}
    \label{dis-linearization} 
    \begin{multlined}[t][0.9\textwidth]
        \int_{\Omega} \nabla \eta_h \cdot \nabla v_h \,dx + \frac{1}{3}\sum_{T \in \mathcal{T}_h} |T| \sum_{p_i \in \mathcal{N}_h\cap\overline T} \1_{\{y_{h}(p_i)>0\}}(p_i) \eta_h(p_i) v_h(p_i) \\
        = \int_{\Omega} w_hv_h\,dx \quad \text{for all } v_h \in V_h.
    \end{multlined}
\end{equation}
Again, by a slight abuse of notation, we denote the coefficient vectors $(\eta_h(p_i))_{i=1}^n$ and $(w_h(p_i))_{i=1}^n$ by $\eta\in\R^n$ and $w\in\R^n$, respectively.
The discrete equation \eqref{dis-linearization} thus becomes the linear algebraic system
\begin{equation}
    (A + K_y)\eta = Mw, \label{coor-linearization}
\end{equation}
where
the matrix $K_y$ is defined by
\begin{equation*}
    K_y = \frac{1}{3} \diag\left(\omega_i \1_{\{y_i >0\}} \right) \in \R^{n \times n}.
\end{equation*}
By standard arguments, the variational equations \eqref{dis-ori-P} and \eqref{dis-linearization} as well as the corresponding algebraic systems \eqref{coor-ori-P} and \eqref{coor-linearization} admit unique solutions, whose relation is given in the following lemma.
\begin{lemma}
    Let $F_h:\R^n\mapsto \R^n$ be the mapping that assigns $u \in \R^n$ to the unique solution $y \in \R^n$ to \eqref{coor-ori-P}. Similarly, denote for arbitrary $u\in \R^n$ by $G_{u,h}:\R^n\to \R^n$ the mapping that assigns $w\in \R^n$ to the unique solution $\eta\in\R^n$ to \eqref{coor-linearization}.
    Then $G_{u,h}\in \partial_B F_h(u)$.
\end{lemma}
\begin{proof}
    First, the Lipschitz continuity of $F_h$ implies that for any $\tilde u \in \R^n$ with $(F_h(\tilde u))_i \neq 0$ for all $1 \leq i \leq n$, we have that $F_h(\tilde u + v)  = F_h(\tilde u) + G_{u,h}v$ provided that $|v|_2$ is small enough.
    Consequently, $F_h$ is differentiable at $\tilde u$ and $F_h'(\tilde u) = G_{\tilde u, h}$. 
    For each $k \geq 1$, we now choose $y^k \in \R^n$ componentwise as 
    \begin{equation*}
        y_i^k := 
        \begin{cases}
            y_i &\text{if }y_i\neq 0,\\
            -\frac1k & \text{if }y_i = 0,
        \end{cases}
        \quad\text{for } 1\leq i \leq n,
    \end{equation*}
    set $u^k := M^{-1}\left(Ay^k+D\max(y^k,0)\right)$. Since $y^k_i \neq 0$ and $\1_{ \{y^k_i >0\}} = \1_{\{ y_i >0\}}$ for all $1 \leq i \leq n$, the mapping $F_h$ is Gâteaux differentiable at $u^k$, and $F_h'(u^k) = G_{u, h}$. Since $A$ and the componentwise $\max$ are continuous, we obtain that $u^k \to u$ in $\R^n$ and hence $G_{u,h} \in \partial_{B}F_h(u)$.
\end{proof}

\bigskip 

We now show that the non-smooth nonlinear system \eqref{coor-ori-P} can be solved by a semi-smooth Newton method. Defining the mapping $H: \R^n \to \R^n$ by 
\begin{equation*}
    H(y) = A y + D \max(y,0) - Mu,
\end{equation*}
the discrete system \eqref{coor-ori-P} is equivalent to $H(y)=0$.
For each $y^{k} \in \R^n$, we now set
\begin{equation*}
    M_k := A + DE_k,\qquad E_k := \mathrm{diag}\left( \1_{ \{y^{k}_i \geq 0\}} \right).
\end{equation*}
Since the componentwise $\max$ function is locally Lipschitz and piecewise continuously differentiable in each component, we deduce from \cite[Props.~2.26, 2.10, 3.5, 3.8]{Ulbrich2011} that $M_k$ is a Newton derivative of $H$ at $y^{k}$.
Denoting the \emph{active set} at $y^{k}$ by
\begin{equation}\label{eq:active-set}
    AC^{k} := \left\{ i: 1 \leq i \leq n, y^{k}_i \geq 0 \right\},
\end{equation}
we have for any $y \in \R^n$ that 
\begin{equation*}
    y^TM_ky = y^TAy + y^TDE_ky = y^TAy + \sum_{i \in AC^{k}} d_{ii}|y_i|^2 \geq y^TAy \geq c_0 |y|_2^2
\end{equation*} 
for some constant $c_0>0$ and hence that 
$\| M_k^{-1}\|_2 \leq c_0^{-1}$. Here, $|y|_2$ and $\|M\|_2$ denote the Euclidean norm of $y\in\R^n$ and the induced (spectral) norm of $M\in\R^{n\times n}$, respectively. By \cite[Prop.~2.12]{Ulbrich2011}, the semi-smooth Newton iteration
\begin{equation}
    M_k \delta y = -H(y^k),\qquad y^{k+1}=y^k+\delta y,
\end{equation}
then converges locally superlinearly to a solution to \eqref{coor-ori-P}.
Furthermore, since the equation is piecewise linear, the semi-smooth Newton method has a finite termination property: If the active sets \eqref{eq:active-set} coincide between successive iterations $k$ and $k+1$, then $H(y^{k+1})=0$; cf.~\cite[Rem.~7.1.1]{Kunisch:2008a}. Correspondingly, we terminate the iteration as soon as the active sets remain unchanged.

\subsection{Numerical examples} 

We consider $\Omega=(0,1)^2\subset \R^2$ and use a uniform triangular Friedrichs--Keller triangulation with $n_h\times n_h$ vertices for $n_h=512$ unless noted otherwise.
The semi-smooth Newton systems are solved by a direct sparse solver,
and the semi-smooth Newton iteration for solving the non-smooth nonlinear system \eqref{coor-ori-P} is started from $y^0=0$ and terminated when the active sets corresponding to two consecutive steps coincide.
The Python implementation used to generate the following results (as well as a Julia implementation) can be downloaded from \url{https://www.github.com/clason/bouligandlandweber}.

The exact solution to be reconstructed is defined as
\begin{multline*}
    u^\dag(x_1,x_2) = \max(y^\dag(x_1,x_2),0) \\
    + 
    \left[
        4 \pi^2 y^\dag(x_1,x_2) - 2\left((2x_1-1)^2 + 2(x_1-1+\beta)(x_1-\beta) \right)\sin (2\pi x_2)
    \right]
    \1_{(\beta,1-\beta]}(x_1)
\end{multline*}
where
\begin{equation*}
    y^\dag(x_1,x_2) = 
    \left[(x_1-\beta)^2(x_1-1+ \beta)^2 \sin (2\pi x_2)\right]
    \1_{(\beta,1-\beta]}(x_1)
\end{equation*} 
with constant $\beta = 0.005$ is the corresponding exact state; see \cref{fig:target_exactdata}.
\begin{figure}[t]
    \centering
    \begin{minipage}[t]{0.495\textwidth}
        \centering
        \includegraphics[width=\linewidth]{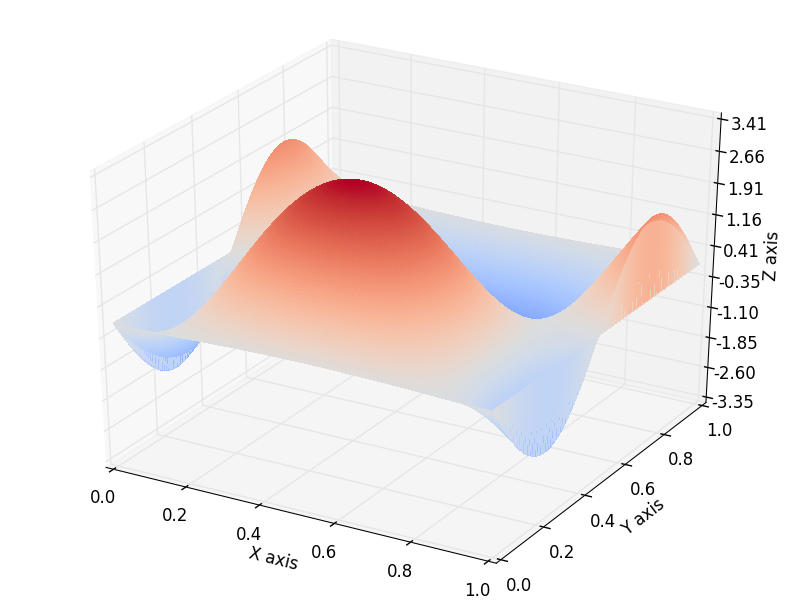}
        \caption{exact data $u^\dag$}
        \label{fig:target_exactdata}
    \end{minipage}
    \hfill
    \begin{minipage}[t]{0.495\textwidth}
        \includegraphics[width=\linewidth]{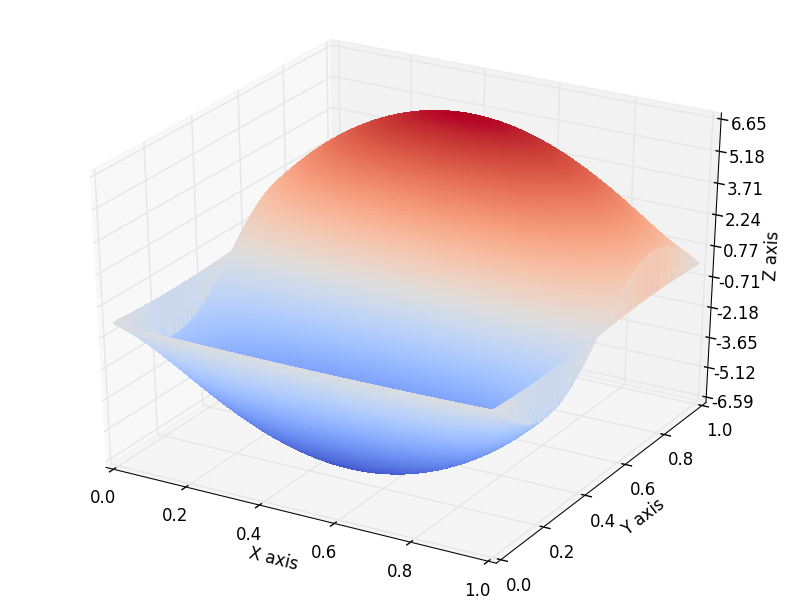}
        \caption{starting point $u_0=\bar u$}
        \label{fig:guess-source}
    \end{minipage}
\end{figure}
It is easy to see that $y^\dag \in H^2(\Omega) \cap H^1_0(\Omega)$ satisfies \eqref{eq:maxpde} for the right-hand side $u^\dag$ and that $y^\dag$ vanishes on a set of measure $2\beta$. Therefore, the forward operator $F$ is not Gâteaux differentiable at $u^\dag$; see \cref{prop:gateaux}.

We then add random Gaussian noise componentwise to the discrete projection $y^\dag_h$ to obtain noisy data $y^\delta_h$ with ($L^2$) noise level
\begin{equation*}
    \delta := \| y^\dag_h - y^\delta_h \|_{L^2(\Omega)}.
\end{equation*}
Here and below, all norms for discrete functions $v_h$ are calculated exactly via $\|v_h\|_{L^2(\Omega)}^2 = v_h^TM v_h$ (identifying again the function $v_h$ with its vector of expansion coefficients). To keep the notation simple, we omit the subscript $h$ from now on.

In the following, we illustrate the convergence for both noise-free and noisy data and two different choices of starting points: the trivial point $u_0\equiv 0$ and the discrete projection of 
\begin{equation} \label{guess-source}
    \bar u:= u^\dag - 2 \rho \sin(\pi x_1) \sin(2\pi x_2),
\end{equation}
see \cref{fig:guess-source}. 
We point out that for the second starting point, $u^\dag$ satisfies the \emph{generalized source condition}
\begin{equation}
    \label{source-condition}
    u^\dag - u_0 \in \mathcal{R}\left(G_{u^\dag}^*\right),
\end{equation} 
where $\mathcal{R}(T)$ denotes the range of operator $T$. Note also that this choice of $u_0$ is far from the exact solution $u^\dag$. 
In all cases, the parameters in the Bouligand--Landweber iteration \eqref{eq:modified-landweber} are set to
\begin{equation*}\label{eq:parameters}
    \mu = 0.1, \quad \tau = 1.4, \quad \rho = 5,\quad w_n= \lambda= \Lambda = \frac{2-2\mu}{\bar L^2},\quad \bar L = 5 \cdot 10^{-2},
\end{equation*}
and the Bouligand--Landweber iteration is terminated at $N_{\max}=5000$.

We first address the convergence for noise-free data $y^\dag$ from \cref{thm:free} by plotting in \cref{fig:free} the relative error
\begin{equation} \label{rel-error-free}
    E_n:=\frac{\| u^\dag -u_n \|_{L^2(\Omega)}}{\| u^\dag\|_{L^2(\Omega)}}
\end{equation}
of the iterates $u_n$ as a function of the iteration index $n$.
As \cref{fig:free-zero} shows, the iteration slows down for the trivial starting point $u_0\equiv 0$ after $100$ steps of rather fast convergence. However, the relative error continues to decrease significantly even after that.
In contrast, \cref{fig:free-source} demonstrates that the rate of convergence for the starting point $u_0=\bar u$ from \eqref{guess-source} is substantially higher. Although here the initial relative error is three times greater than for the trivial starting point, the relative error drops quickly from $3.33460$ to less than $10^{-3}$ after $25$ steps and then continues to reduce.  
\begin{figure}[t]
    \centering
    \begin{subfigure}[b]{0.49\textwidth}
        \centering
        \input{rel_error_free_zero.tikz}
        \caption{starting point $u_0\equiv 0$}
        \label{fig:free-zero}
    \end{subfigure}
    \hfill
    \begin{subfigure}[b]{0.49\textwidth}
        \centering
        \input{rel_error_free_source.tikz}
        \caption{starting point $u_0=\bar u$}
        \label{fig:free-source}
    \end{subfigure}
    \caption{relative error $E_n$ from \eqref{rel-error-free} of iterates in the noise-free setting}
    \label{fig:free}
\end{figure}
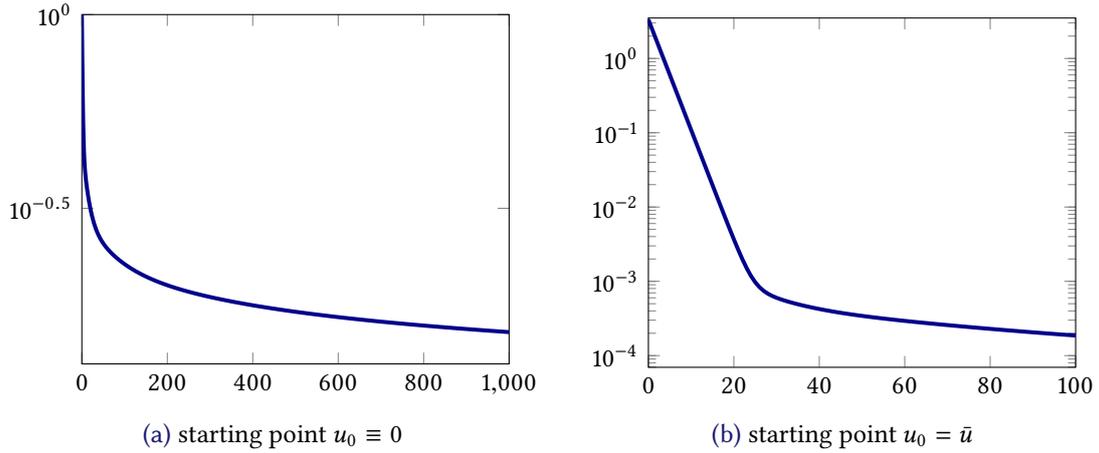

We next turn to the regularization property from \cref{thm:REG}. \cref{tab:noise} shows for a decreasing sequence of noise levels and both starting points (for the same realization of the random data) the stopping index $N_\delta=N(\delta,y^\delta)$, the total number of semi-smooth Newton steps, and the relative error 
\begin{equation} \label{rel-error-noise}
    E_N^\delta:= \frac{\| u^\dag -u^\delta_{N_\delta} \|_{L^2(\Omega)}}{\| u^\dag\|_{L^2(\Omega)}}.
\end{equation}
\begin{table}
    \centering
    \begin{tabular}{%
            S[table-format=1.2e-2,scientific-notation=true,round-mode=places,round-precision=2]
            S[table-format=4]
            S[table-format=1.2e-1,scientific-notation=true,round-mode=places,round-precision=2]
            S[table-format=5]
            S[table-format=2]
            S[table-format=1.2e-1,scientific-notation=true,round-mode=places,round-precision=2]
            S[table-format=2.2,round-mode=places,round-precision=2]
            S[table-format=3]
        }
        \toprule
        & \multicolumn{3}{c}{$u_0\equiv 0$} & \multicolumn{4}{c}{$u_0=\bar u$} \\
        \cmidrule(lr){2-4} \cmidrule(lr){5-8}
        {$\delta$} & {$N_\delta$} & {$E_N^\delta$}     & {\#\,SSN} & {$N_\delta$} & {$E_N^\delta$}        & {$R_N^\delta$}   & {\#\,SSN}   \\
        \midrule
        1.059e-02  & 3            & 0.5350603488049664 & 7         & 7            & 0.3038229834595208    & 4.426907180089214  & 23        \\
        5.289e-03  & 6            & 0.4180029139981761 & 16        & 9            & 0.15376641213903283   & 3.1703566958991205 & 31        \\
        1.061e-03  & 42           & 0.2656438267572647 & 125       & 14           & 0.027605772447678992  & 1.2708364774063334 & 50        \\
        5.299e-04  & 106          & 0.2246960078391614 & 342       & 16           & 0.014122792114159833  & 0.9199372612366256 & 55        \\
        1.059e-04  & 1120         & 0.1485472865894022 & 4409      & 21           & 0.0026765254692178563 & 0.38999229008982084& 75        \\
        5.288e-05  & 3224         & 0.1227892523783321 & 12562     & 23           & 0.0014909642898076302 & 0.3074317947414472 & 84        \\
        1.060e-05  & {---}        & {---}              & {---}     & 28           & 0.0006803581750327172 & 0.31339489771213125& 102       \\
        5.292e-06  & {---}        & {---}              & {---}     & 35           & 0.0004926208918608504 & 0.3210834366382792 & 133       \\
        \bottomrule
    \end{tabular}
    \caption{regularization property: noise level $\delta$; stopping index $N_\delta=N(\delta,y^\delta)$; relative error $E_N^\delta$ from \eqref{rel-error-noise}; empirical convergence rate $R_N^\delta$ from \eqref{conv-rate}; total number of semi-smooth Newton steps (--- : not converged)}
    \label{tab:noise}
\end{table}
First we note that since the trivial starting point $u_0\equiv 0$ is actually closer to $u^\dag$ than to $\bar u$, the discrepancy principle is satisfied earlier for the former when $\delta > 5\cdot 10^{-3}$ although the relative error is smaller for the latter. Considering the convergence behavior for $u_0\equiv 0$, the relative error decreases slowly from $5.35\cdot10^{-1}$ to $1.23\cdot10^{-1}$; however, the stopping index $N_\delta$ increases rapidly from $3$ to $3224$ before failing to converge within the prescribed maximum number of $5000$ iterations. This is reasonable as the classical Landweber iteration is known to be similarly slow but reliable. In contrast, for the starting point $u_0=\bar u$, the relative error decreases rapidly from $3.04\cdot 10^{-1}$ to $4.93\cdot10^{-4}$ while the stopping index $N_\delta$ increases only slightly from $7$ to $35$. 
As expected, we thus see much faster convergence for the Bouligand--Landweber iteration if the exact solution satisfies a generalized source condition.
For $u_0=\bar u$, \cref{tab:noise} also shows the empirical convergence rate
\begin{equation} \label{conv-rate}
    R_N^\delta:= \frac{\| u^\dag - u^\delta_{N_\delta}\|_{L^2(\Omega)}}{\sqrt{\delta}},
\end{equation}
which stabilizes around $0.3$ for $\delta \leq 5 \cdot 10^{-5}$. This corresponds to the convergence rate $\mathcal{O}(\sqrt{\delta})$ expected from the classical source condition $u^\dag-u_0\in \mathcal{R}(F'(u^\dag)^*)$.
For both starting points, the average number of semi-smooth Newton iterations per Bouligand--Landweber iteration is consistently between $2$ and $4$, increasing slightly as $\delta$ decreases (which can further be reduced by warm-starting the Newton iteration, i.e., starting with $y^0=y_{n-1}$ instead of $y^0=0$).

Finally, we illustrate the effects of discretization by showing for $n_h=256$ in \cref{fig:noise-zero,fig:noise-source} the noisy data $y^\delta$ for the noise levels $\delta \in \{1.049 \cdot 10^{-2}, 1.060 \cdot 10^{-3}, 1.055 \cdot 10^{-4}\}$ together with the corresponding reconstructions $u^\delta_{N_\delta}$ and the starting points $u_0\equiv 0$ and $u_0=\bar u$, respectively. As can be seen, the stopping indices ($N_\delta = 4,42,1119$ for $u_0\equiv 0$ and $N_\delta = 7,14,21$ for $u_0\equiv \bar u$) are very similar to the results for $n_h=512$. The same holds for the relative errors (not shown in the figures), which are 
$E_N^\delta\approx4.74 \cdot 10^{-1},2.62 \cdot 10^{-1},1.42 \cdot 10^{-1}$ for $u_0\equiv 0$ and  $E_N^\delta \approx 3.00 \cdot 10^{-1}, 2.76 \cdot 10^{-2},2.68 \cdot 10^{-3}$ for $u_0 = \bar u$.
\begin{figure}[tp]
    \centering
    \begin{subfigure}[t]{0.495\textwidth}
        \centering
        \includegraphics[width=\linewidth]{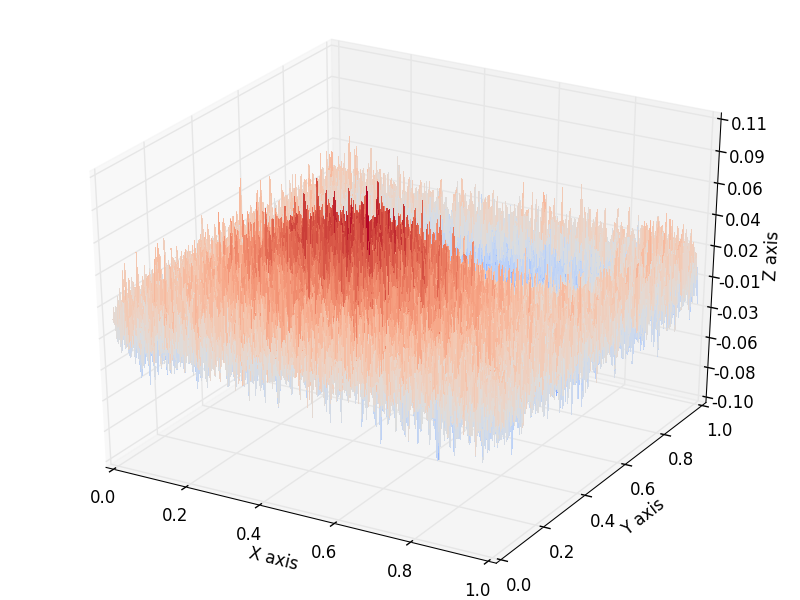}
        \caption{$y^\delta$, $\delta = 1.049 \cdot 10^{-2}$}
    \end{subfigure}
    \hfill
    \begin{subfigure}[t]{0.495\textwidth}
        \centering 
        \includegraphics[width=\linewidth]{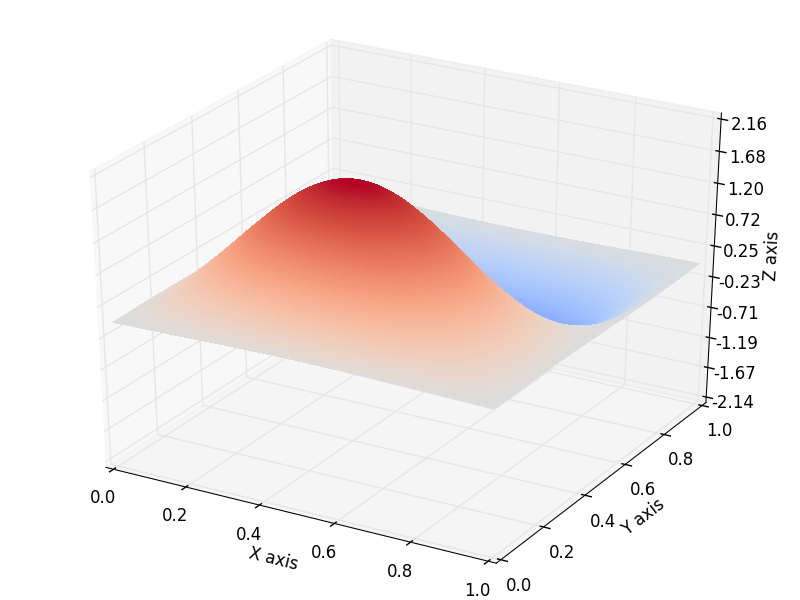}
        \caption{$u^\delta_{N_\delta}$, $N_\delta = 4$} 
    \end{subfigure}
    \\
    \begin{subfigure}[t]{0.495\textwidth}
        \centering
        \includegraphics[width=\linewidth]{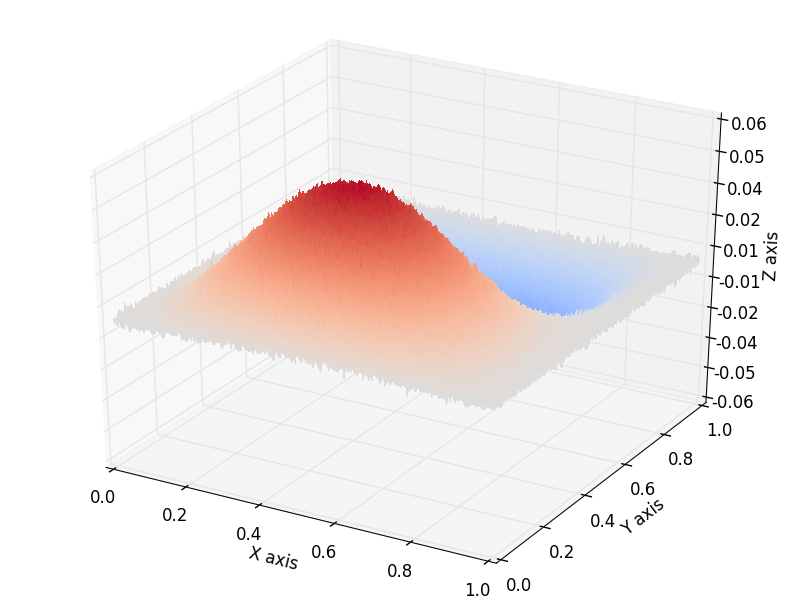}
        \caption{$y^\delta$, $\delta = 1.060 \cdot 10^{-3}$}
    \end{subfigure}
    \hfill
    \begin{subfigure}[t]{0.495\textwidth}
        \centering 
        \includegraphics[width=\linewidth]{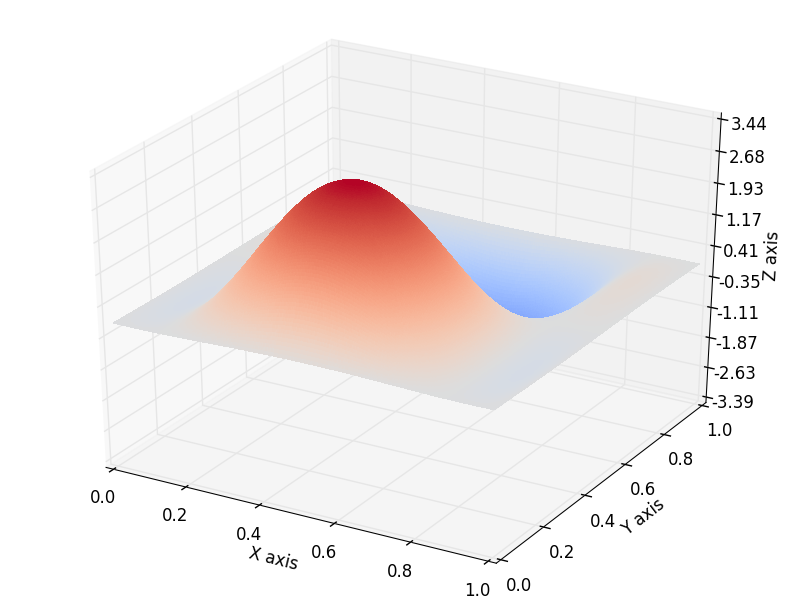}
        \caption{$u^\delta_{N_\delta}$, $N_\delta= 42$}
    \end{subfigure}
    \\
    \begin{subfigure}[t]{0.495\textwidth}
        \centering
        \includegraphics[width=\linewidth]{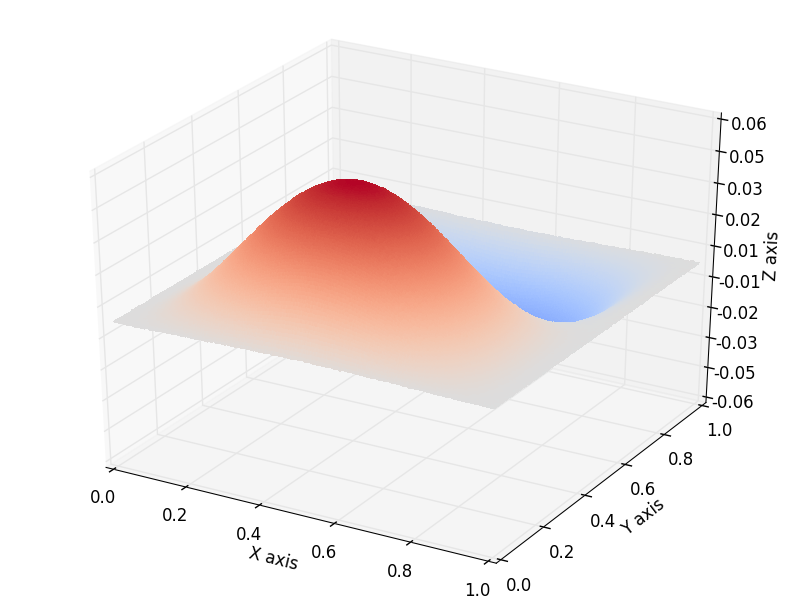}
        \caption{ $y^\delta$, $\delta = 1.055 \cdot 10^{-4}$}
    \end{subfigure}
    \hfill
    \begin{subfigure}[t]{0.495\textwidth}
        \centering 
        \includegraphics[width=\linewidth]{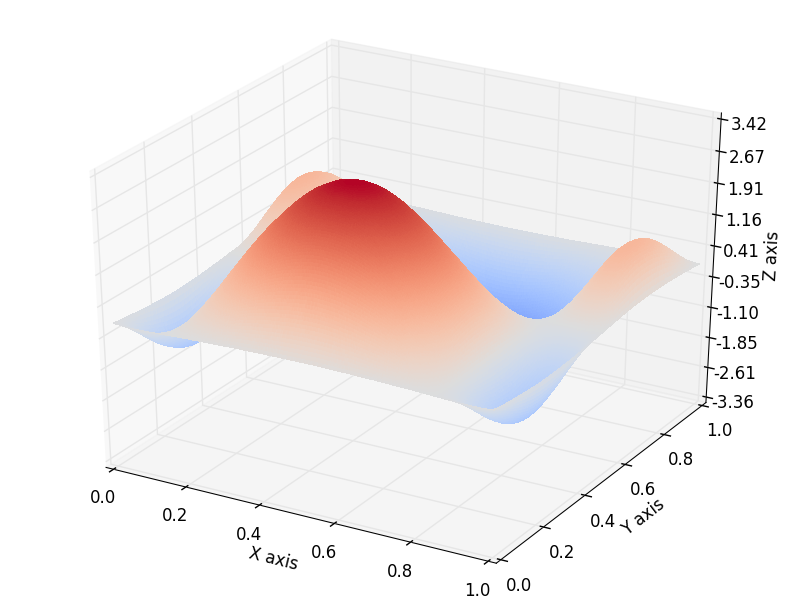}
        \caption{$u^\delta_{N_\delta}$, $N_\delta = 1119 $}
    \end{subfigure}
    \caption{noisy data $y^\delta$ and reconstructions $u^\delta_{N_\delta}$ for starting point $u_0\equiv 0$}
    \label{fig:noise-zero}
\end{figure}
\begin{figure}[tp]
    \centering
    \begin{subfigure}[t]{0.495\textwidth}
        \centering
        \includegraphics[width=\linewidth]{noise_01.png}
        \caption{$y^\delta$, $\delta = 1.049 \cdot 10^{-2}$}
    \end{subfigure}
    \hfill
    \begin{subfigure}[t]{0.495\textwidth}
        \centering 
        \includegraphics[width=\linewidth]{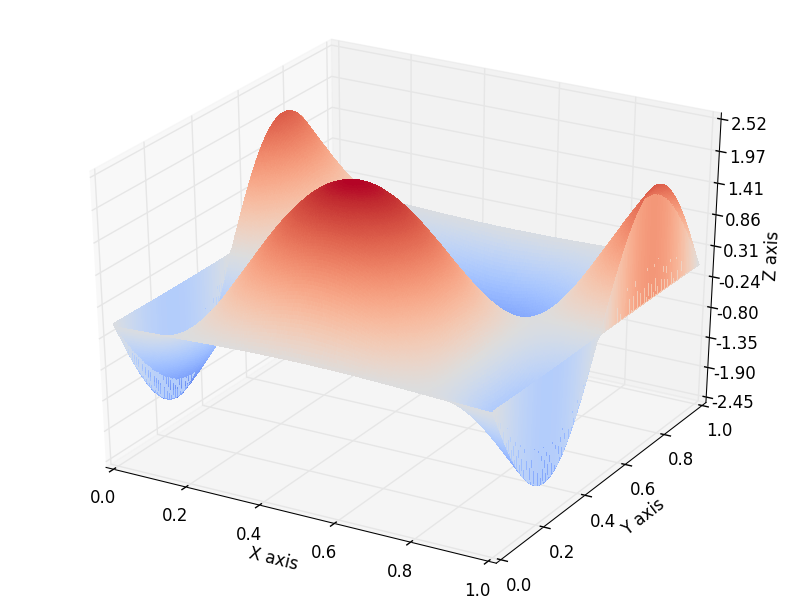}
        \caption{$u^\delta_{N_\delta}$, $N_\delta = 7$} 
    \end{subfigure}
    \\
    \begin{subfigure}[t]{0.495\textwidth}
        \centering
        \includegraphics[width=\linewidth]{noise_001.png}
        \caption{$y^\delta$, $\delta =1.060 \cdot 10^{-3}$}
    \end{subfigure}
    \hfill
    \begin{subfigure}[t]{0.495\textwidth}
        \centering 
        \includegraphics[width=\linewidth]{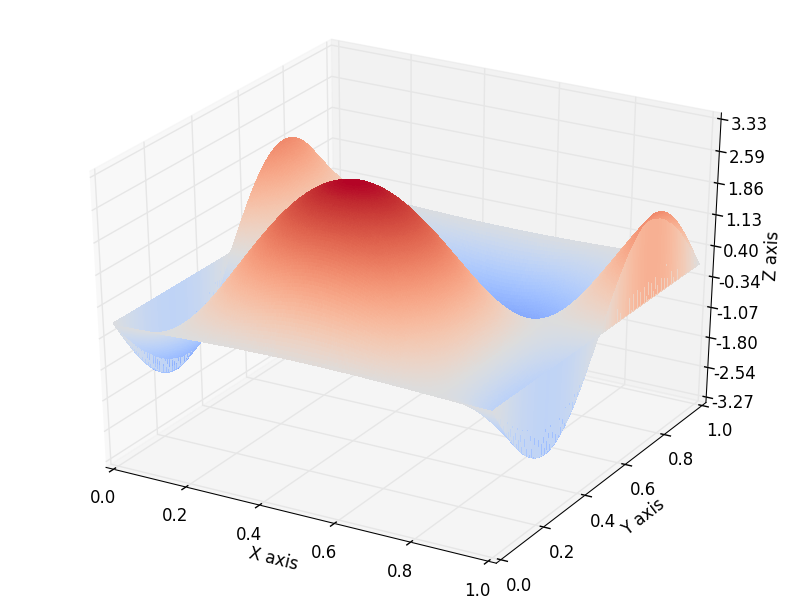}
        \caption{$u^\delta_{N_\delta}$, $N_\delta = 14$} 
    \end{subfigure}
    \\
    \begin{subfigure}[t]{0.495\textwidth}
        \centering
        \includegraphics[width=\linewidth]{noise_0001.png}
        \caption{$y^\delta$, $\delta = 1.055 \cdot 10^{-4}$}
    \end{subfigure}
    \hfill
    \begin{subfigure}[t]{0.495\textwidth}
        \centering 
        \includegraphics[width=\linewidth]{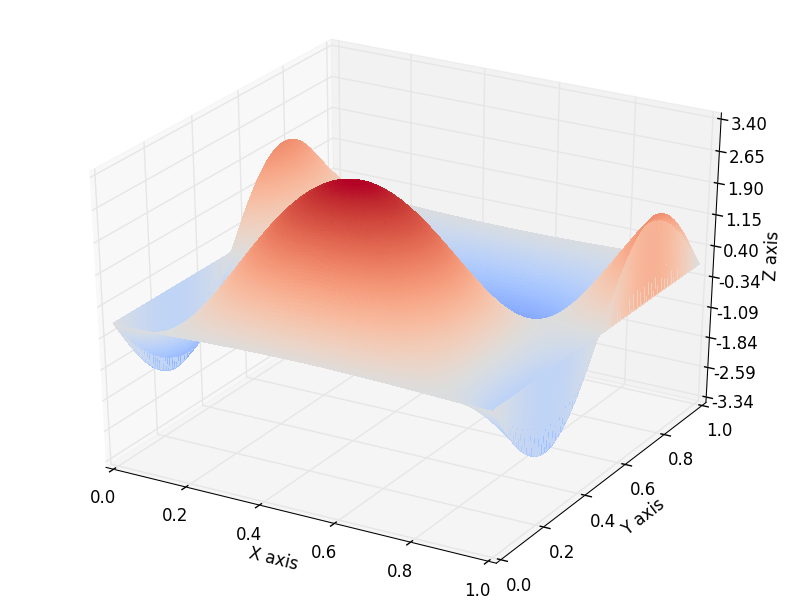}
        \caption{$u^\delta_{N_\delta}$, $N_\delta = 21$} 
    \end{subfigure}
    \caption{noisy data $y^\delta$ and reconstructions $u^\delta_{N_\delta}$ for starting point $u_0 = \bar u$}
    \label{fig:noise-source}
\end{figure}


\section{Conclusion} \label{sec:Con}

We have considered the iterative regularization of an inverse source problem for a non-smooth elliptic PDE. By using a Bouligand subderivative in place of the non-existent Fréchet derivative of the forward mapping, a modified Landweber method (which we call Bouligand--Landweber iteration in this case) can be applied. To account for the missing continuity of the Bouligand subderivative mapping, a new convergence analysis of the modified Landweber method is provided that is based on the concept of asymptotic stability and merely requires a generalized tangential cone condition together with some boundedness assumptions. This condition is verified for our non-smooth model problem provided that the non-differentiability of the forward mapping is sufficiently ``weak'' at the exact solution, and thus the Bouligand--Landweber iteration provides a convergent regularization method. Numerical examples verify the convergence of the iteration for exact as well as for noisy data. While the convergence is slow for an arbitrary starting point (sufficiently close to the solution), it is significantly faster for a starting point for which the exact solution satisfies a generalized source condition.

This work can be extended in a number of directions. First, it would be interesting to derive convergence rates under the generalized source condition \eqref{source-condition}. Another practically relevant issue would be to extend the analysis of the method to cover other classes of non-smooth PDEs such as time-dependent equations or equations with non-smooth nonlinearities entering higher-order terms as for the two-phase Stefan problem; further work will also consider solution operators for variational inequalities for which Bouligand differentiability has recently been shown \cite{Meyer:2018,Rauls:2018a,Rauls:2018b}. For practical applications, the convergence analysis of the modified Landweber method should also take into account the discretization of the non-smooth PDE, where adaptivity can be used to further reduce the computational effort as in \cite{Ramlau:2008} for linear inverse problems in Besov spaces.   
Finally, similar non-smooth extensions of iterative regularization methods of Newton-type should lead to significantly faster convergence.


\section*{Acknowledgment} 
The authors would like to thank the anonymous reviewers for their detailed and constructive comments that helped to significantly improve the presentation.
This work was supported by the DFG under the grants CL 487/2-1 and RO 
2462/6-1, both within the priority programme SPP 1962 ``Non-smooth and 
Complementarity-based Distributed Parameter Systems:
Simulation and Hierarchical Optimization''.

\appendix

\section{Elliptic equations with piecewise differentiable nonlinearities}\label{app:piecewise-diff}

In this appendix, we show that  \crefrange{ass:cont}{ass:adj_bounded} are satisfied for a general class of non-smooth semilinear elliptic equations with $PC^1$-nonlinearities.

\bigskip

We first recall the following definition from, e.g., \cite[Chap.~4]{Scholtes} and \cite[Def.~2.19]{Ulbrich2011}.
Let $V\subset \R$ be an open set. A function $f: V \to \R$ is called a \emph{piecewise differentiable function} or \emph{$PC^1$-function} if $f$ is continuous, and for each point $x_0 \in V$ there exist a neighborhood $W \subset V$ and a finite set of $C^1$-functions $f_i: W \to \R$, $i=1,2,\dots,N$, such that
\begin{equation*}
    f(x) \in \{f_1(x), f_2(x),\dots,f_N(x)\} \quad \text {for all} \quad x \in W.
\end{equation*}
The set $\{f_1,f_2,\dots,f_N\}$ is said to be the \emph{selection functions} of $f$ on $W$. We denote by $S_f\subset V$ the set of all points in $V$ at which $f$ is not differentiable, i.e.,
\begin{equation}
    S_f := \{x \in V: f' \ \text {does not exist at} \ x \}.
\end{equation}   
We assume in the following that the set $S_f$ consists of a finite number of points $t_1,t_2,\dots,t_k$. By virtue of the decomposition theorem for piecewise smooth functions \cite[Prop.~2D.7]{Dontchev-Rockafellar2014}, $f$ can be represented as 
\begin{equation*} 
    f(t) = \sum_{i=1}^{k+1} \1_{(t_{i-1},t_i]}(t)f_i(t)  \quad \text{for all } t \in \R,
\end{equation*}
where $f_i$, $1 \leq i \leq k+1$, are $C^1$-functions on $\R$ and
\begin{equation*}
    -\infty =: t_0 < t_1 < \cdots < t_k < t_{k+1} := \infty
\end{equation*} 
with the convention $(t_k, t_{k+1}] := (t_k, \infty)$.
Moreover, we assume that each $f_i$ is non-decreasing on $(t_{i-1},t_{i})$, $1 \leq i \leq k+1$, and that
\begin{equation} \label{affine1}
    f_{i}(t_i) = f_{i+1}(t_i) \quad \text {for all} \quad 1 \leq i \leq k.
\end{equation}
We require the following technical lemma regarding the nonlinearity.
\begin{lemma}\label{lem:remainder}
    For each $M>0$, let $r_{iM}: [-M, M] \times \R \to [0, \infty)$, $i = 1,2,\dots,k+1$, be defined as 
    \begin{equation}\label{eq:remainder}
        r_{iM}(t,s) := 
        \begin{cases} 
            \left | \frac{f_i(t+s)- f_i(t)}{s} - f'_i(t) \right| & s\neq 0,\\
            0 & s=0.
        \end{cases}
    \end{equation} 
    Then, $r_{iM}$ is continuous and satisfies
    \begin{equation} \label{eq:infinitesimal}
        r_{iM}(t,s) \to 0 \quad \text {as } s \to 0 \quad \text {uniformly in} \ t \in [-M,M].     
    \end{equation}
\end{lemma}
\begin{proof}
    Clearly, $r_{iM}$ is continuous at every point $(t,s)$ with $s \neq 0$. Moreover, we have for any $t \in [-M, M]$ and $s \neq 0$ that
    \begin{equation*}
        r_{iM}(t,s) =  \left | \int_{0}^1 \left( f'_i(t+s \tau) - f'_i(t) \right) d\tau \right | \leq \int_{0}^1 \left| f'_i(t+s \tau) - f'_i(t) \right| d\tau.
    \end{equation*}
    From this and the uniform continuity of $f'_i$ on bounded sets, Lebesgue's dominated convergence theorem yields \eqref{eq:infinitesimal}. Consequently, $r_{iM}$ is continuous at $(t,0)$. 
\end{proof}

\bigskip

We now consider the non-smooth semilinear elliptic equation 
\begin{equation} \label{eq:semi-elliptic}
    \left\{
        \begin{aligned}
            A y + f(y) &= u \quad \text{in } \Omega,\\
            \quad y    &=0  \quad \text{on } \partial \Omega,
        \end{aligned}
    \right.
\end{equation} 
with $u \in L^2(\Omega)$, $\Omega\subset \R^d$ for $d\leq 3$, $A$ an elliptic second-order partial differential operator with bounded and measurable coefficients satisfying
\begin{equation*}
    c_0\norm{y}_{H^1_0(\Omega)}^2\leq \langle  Ay,y \rangle \leq c_1\norm{y}_{H^1_0(\Omega)}^2\quad \text{for all } y\in H^1_0(\Omega)
\end{equation*}
for some $c_1\geq c_0>0$,
and a given $PC^1$-function $f$ satisfying the above assumptions. Here $\langle \cdot, \cdot \rangle$ stands for the pairing between $H^1_0(\Omega)$ and $H^{-1}(\Omega)$.
From \cite[Thm.~4.7]{Troltzsch}, we know that for each $u \in L^2(\Omega)$, the equation \eqref{eq:semi-elliptic} admits a unique weak solution $y_u \in  H^1_0(\Omega) \cap C(\overline{\Omega})$. 
Furthermore, a constant $C_\infty$ exists such that
\begin{equation} \label{eq:a-priori}
    \norm{y_u}_{H^1_0(\Omega)} + \norm{y_u}_{C(\overline{\Omega})} \leq C_\infty\norm{u - f(0)}_{L^2(\Omega)}  \quad \text{for all } u \in L^2(\Omega).
\end{equation}

From now on, we denote by $F: L^2(\Omega) \to H^1_0(\Omega) \cap C(\overline{\Omega}) \hookrightarrow L^2(\Omega)$ the solution operator of \eqref{eq:semi-elliptic}.  Since the $f_i$ are all $C^1$-functions, they are thus Lipschitz continuous on bounded sets. From this and \cite[Prop.~4.1.2]{Scholtes}, $f$ is also Lipschitz continuous on bounded sets. By a standard argument, we arrive at the following result, which generalizes \cref{prop:well-posed}.
\begin{proposition} \label{prop:appe_Lipschitz}
    The solution operator  $F: L^2(\Omega) \to H^1_0(\Omega) \cap C(\overline{\Omega})$ is Lipschitz continuous on bounded sets in $L^2(\Omega)$, i.e., for any bounded set $W \subset L^2(\Omega)$ there exists a constant $L_W>0$ such that
    \begin{equation} \label{eq:F-Lip2}
        \norm{F(u)-F(v)}_{H^1_0(\Omega)} + \norm{F(u)-F(v)}_{C(\overline\Omega)} \leq L_W \norm{u-v}_{L^2(\Omega)} \quad \text{for all } u, v \in W.
    \end{equation}      
\end{proposition}
\begin{proof}
    Take any $u , v \in W$ and set $y:= F(u)$ and $z:= F(v)$. We then have
    \begin{equation} \label{eq:subtract}
        \left\{
            \begin{aligned}
                A (y-z)  &= u -v - \left(f(y) - f(z) \right)\quad \text{in } \Omega,\\
                \quad y-z    &=0  \quad \text{on } \partial \Omega,
            \end{aligned}
        \right.
    \end{equation} 
    which together with \cite[Thm.~4.7]{Troltzsch} leads to 
    \begin{equation} \label{eq:esti-subtract}
        \norm{y-z}_{H^1_0(\Omega)} + \norm{y-z}_{C(\overline\Omega)}
        \begin{aligned}[t]
            &\leq C \norm{u - v - \left(f(y) - f(z)\right) }_{L^2(\Omega)} \\
            & \leq C \left( \norm{u - v  }_{L^2(\Omega)} +\norm{f(y) - f(z) }_{L^2(\Omega)} \right)
        \end{aligned}           
    \end{equation}      
    for some constant $C>0$. Moreover, \eqref{eq:a-priori} implies that $y$ and $z$ belong to a bounded set in $C(\overline{\Omega})$. From this and the Lipschitz continuity on bounded sets of $f$, we obtain 
    \begin{equation} \label{eq:esti-nonlinearity}
        \norm{f(y) - f(z) }_{L^2(\Omega)} \leq C_W  \norm{y - z  }_{L^2(\Omega)}.
    \end{equation}
    In addition, testing the first equation in \eqref{eq:subtract} with $(y-z)$ and using the non-decreasing monotonicity of $f$ yields that
    \begin{equation*}
        \langle A(y-z) , y-z  \rangle \leq \left(u-v , y-z  \right)_{L^2(\Omega)} \leq \norm{u-v}_{L^2(\Omega)}\norm{y-z}_{L^2(\Omega)}.
    \end{equation*}     
    The uniform ellipticity of $A$ and the Poincaré inequality thus imply that 
    \begin{equation} \label{eq:esti-L2-Lip}
        \norm{y-z}_{L^2(\Omega)} \leq C_P \norm{u-v}_{L^2(\Omega)}
    \end{equation}
    for some constant $C_P>0$. By inserting \eqref{eq:esti-L2-Lip} into \eqref{eq:esti-nonlinearity} and then using \eqref{eq:esti-subtract}, we obtain the desired estimate.
\end{proof}
By virtue of \cref{prop:appe_Lipschitz} and the compact embedding $H^1_0(\Omega) \hookrightarrow L^2(\Omega)$, the uniqueness of solutions to \eqref{eq:semi-elliptic} guarantees that $F:L^2(\Omega) \to L^2(\Omega)$ is completely continuous and hence satisfies \cref{ass:cont}.

\bigskip

For each $u \in L^2(\Omega)$, we further denote by $G_u: L^2(\Omega) \to H^1_0(\Omega) \cap C(\overline \Omega) \hookrightarrow L^2(\Omega)$ the solution operator of the linear equation 
\begin{equation} \label{eq:linearization-appe}
    \left\{
        \begin{aligned}
            A \eta + a_u \eta &= w \quad \text {in } \Omega,\\
            \eta &= 0 \quad \text {on }  \partial \Omega,
        \end{aligned}
    \right.
\end{equation}
for given $w \in L^2(\Omega)$ and 
\begin{equation*} 
    a_u(x) :=  \sum_{i=1}^{k+1} \1_{(t_{i-1},t_i]}(y_u(x))f_i'(y_u(x))\quad\text{for all }x\in \Omega
\end{equation*} 
with $y_u := F(u)$. It is easy to see that 
\begin{equation*}
    a_u(x) \in \partial_B f(y_u(x))\quad\text{for all } x\in\Omega,
\end{equation*} 
where $\partial_B f(t)$ stands for the Bouligand subdifferential of $f$ at $t$.
\begin{remark}
    When $f(t) : = t^+$, we have $k=1$, $f_1(t) = 0$, $f_2(t) = t$, and $S_f = \{t_1\}$ for $t_1=0$. In this case, $a_u=\1_{\{y_u >0\}}$, and so (for $A=-\Delta$), $G_u$ coincides with the operator defined in \cref{prop:Gu}.
\end{remark}
Let $W$ be an arbitrary bounded subset in $L^2(\Omega)$. 
From the a priori estimate \eqref{eq:a-priori}, we see that the set $\{y_u = F(u): u \in W\}$ is bounded in $C(\overline\Omega)$.
Therefore, there exists a constant $C_W>0$ such that
\begin{equation*}
    0 \leq a_u(x) \leq C_W  \quad \text{for all }x\in \Omega
\end{equation*}
for all $u \in W$. From \cite[Thm.~4.7]{Troltzsch}, we obtain for each $\frac{d}{2} < p \leq 2$ a constant $C_{W,p} >0$ such that
\begin{equation}
    \label{eq:apriori2}
    \norm{G_uw}_{H^1_0(\Omega)} + \norm{G_uw}_{C(\overline\Omega)} \leq C_{W,p} \norm{w}_{L^p(\Omega)} \quad \text {for all } u \in W, w \in L^p(\Omega).
\end{equation} 
This yields the boundedness of $\{ \norm{G_u}_{\Linop(L^2(\Omega),H^1_0(\Omega))} \}_{u \in W}$ and so of $\{ \norm{G_u}_{\Linop(L^2(\Omega),L^2(\Omega))} \}_{u \in W}$. On the other hand, for any $h \in L^2(\Omega)$, $\zeta := G_u^*h$ satisfies
\begin{equation*} 
	\left\{
	\begin{aligned}
		A^* \zeta + a_u \zeta &= h \quad \text {in } \Omega,\\
		\zeta &= 0 \quad \text {on }  \partial \Omega,
	\end{aligned}
	\right.
\end{equation*}
where $A^*$ stands for the adjoint operator of $A$. Similar to \eqref{eq:apriori2}, there holds
\begin{equation*}
    \norm{G_u^*h}_{H^1_0(\Omega)} + \norm{G_u^*h}_{C(\overline\Omega)} \leq \hat L \norm{h}_{L^2(\Omega)} \quad \text {for all } u \in W, h \in L^2(\Omega)
\end{equation*} 
for some constant $\hat L >0$. Thus, $G_u$ fulfills \cref{ass:bouligand,ass:adj_bounded} 
with $U = Y := L^2(\Omega)$ and $Z := H^1_0(\Omega)$ for any $\rho_0>0$. 

\bigskip

It remains to verify the generalized tangential cone condition for \cref{ass:gtcc}. We start with a further technical lemma regarding the nonlinearity.
\begin{lemma}\label{lem:lipschitz}
    Let $\rho^*>0$ and
   \begin{equation*}
        |t|\leq 
    M:=\max\left\{ \sup\nolimits_{u \in \overline  B_{L^2(\Omega)}(u^\dag,\rho^*)}  \|y_u\|_{C(\overline\Omega)},|t_1|,\dots,|t_k|\right\}.
    \end{equation*}
    Then, for any $1 \leq i \leq k$, we have
    \begin{align}
        |f_i(t)-f_i(t_i)| &  \leq \beta_{i} |t-t_i|, \label{fi-Lip1}\\
        |f_{i+1}(t)-f_{i+1}(t_i)| &  \leq \beta_{i}|t -t_i|,  \label{fi-Lip2}
    \end{align} 
    with
    \begin{equation*}
        \beta_{i} := \max \left\{ \sup\nolimits_{|s|\leq 2M}r_{iM}(t_i,s)+ |f'_i(t_i)|, \sup\nolimits_{|s|\leq 2M}r_{(i+1)M}(t_i,s) + |f'_{i+1}(t_i)|\right\} < \infty.
    \end{equation*}
\end{lemma}
\begin{proof}
    We first note that $\beta_i<\infty$ for all $1 \leq i \leq k$ since the functions $r_{iM}$ and $r_{(i+1)M}$ are continuous due to \cref{lem:remainder}.
    For any $t \in [-M,M]$, we see from the definition of $r_{iM}$ that 
    \begin{equation*}
        \left|f_i(t) - f_i(t_i) - f_i'(t_i)(t-t_i) \right| = r_{iM}(t_i,t-t_i)|t-t_i|,
    \end{equation*}
    which implies that
    \begin{equation*}
        |f_i(t)-f_i(t_i)|  \leq \left(|f'_i(t_i)| + r_{iM}(t_i,t-t_i) \right)|t-t_i| \leq \beta_i |t-t_i|
    \end{equation*}
    and hence \eqref{fi-Lip1}. The inequality \eqref{fi-Lip2} can be shown similarly.
\end{proof}
The following lemma is a generalization of the key \cref{lem:gtcc-key}.
\begin{lemma} \label{lem:GTCC}
    Let $\rho^*>0$ be given as in \cref{lem:lipschitz} and $\epsilon >0$ be such that
    \begin{equation}
        \label{eq:epsilon_choice1}
        \epsilon < \frac{t_{i}- t_{i-1}}{2} \quad \text {for all} \  2 \leq i \leq k
    \end{equation}
    and $\frac{d}{2} < p <2$. Then there exists a $\bar \rho \in (0, \rho^*]$ such that for all
    $u, \hat u \in \overline  B_{L^2(\Omega)}(u^\dagger, \bar \rho)$, one has $F(u), F(\hat u) \in B_{C(\overline{\Omega})}(y^\dag ,\epsilon)$, and
    \begin{equation*}
        \|F(\hat u) - F(u) - G_u(\hat u - u) \|_{L^2(\Omega)} \leq L_p |\Omega|^{1/2} \| \zeta(u,\hat u) \|_{L^{p'}(\Omega)} \|F(\hat u) - F(u) \|_{L^2(\Omega)}
    \end{equation*} 
    for some constant $L_p>0$ with $p' := \frac{2p}{2-p}$ and
    \begin{multline*}
        \zeta(u,\hat u)  
        := \sum_{i=1}^k \left[\1_{\{y_u \in (t_{i}-\epsilon,t_{i}], y_{\hat u} \in (t_{i},t_{i}+\epsilon)\}} + \1_{\{  y_{\hat u} \in (t_{i}-\epsilon,t_{i}],  y_{u} \in (t_{i},t_{i}+\epsilon)\}} \right ] \beta_i\\
        + \sum_{i=1}^{k+1} \1_{(t_{i-1}, t_{i}]}(y_u)r_{iM}(y_u, y_{\hat u} - y_{u}),
    \end{multline*}
    where the constants $\beta_i$ and $M$ are those from \cref{lem:lipschitz}.
\end{lemma}
\begin{proof}
    Set $\hat y := F(\hat u)$, $y := F(u)$, $\xi := G_u(\hat u - u)$, and $\omega := \hat y - y - \xi$. We then have that
    \begin{align*}
        A\hat y + f(\hat y) &= \hat u, \\
        Ay + f(y) &=u, \\
        A \xi + a_u \xi &= \hat u -u.
    \end{align*}
    This implies that 
    \begin{equation*}
        A(\hat y - y - \xi) + a_u (\hat y - y - \xi) = f(y) - f(\hat y) + a_u (\hat y - y)
    \end{equation*}
    or, equivalently,
    \begin{equation}
        \label{eq:omega}
        A\omega +  a_u \omega = b
    \end{equation}
    with 
    \begin{equation*}
        b := f(y) - f(\hat y) + a_u (\hat y - y). 
    \end{equation*}
    A computation then yields that
    \begin{equation}\label{eq:b-term}
        \begin{aligned}[t]
            b   &= \sum_{i=1}^{k+1}\1_{(t_{i-1},t_i]}(y) f_{i}(y) - \sum_{i=1}^{k+1}\1_{(t_{i-1},t_i]}(\hat y)f_{i}(\hat y)  + \sum_{i=1}^{k+1}\1_{(t_{i-1},t_i]}(y)f'_{i}(y) (\hat y - y)\\    
                & = b_1 + b_2
        \end{aligned}
    \end{equation}
    with 
    \begin{align*}
        b_1 &:= - \sum_{i=1}^{k+1} \1_{(t_{i-1},t_i]}(y)  \left( f_{i}(\hat y) - f_{i}( y) - f'_{i}(y)(\hat y - y) \right)  
        \shortintertext{and} 
        b_2 &:= \sum_{i=1}^{k+1} \left(  \1_{(t_{i-1},t_i]}(y) -  \1_{(t_{i-1},t_i]}(\hat y)  \right) f_{i}(\hat y). 
    \end{align*}
    From the definition of $r_{iM}$, it holds that
    \begin{equation} 
        \label{eq:b1-term}
        |b_{1}| \leq  \sum_{i=1}^{k+1} \1_{(t_{i-1},t_i]}(y) r_{iM}(y, \hat y -y)  |\hat y - y|.
    \end{equation}
    To estimate $b_2$, we first observe that
    \begin{equation}
    	\label{eq:compute-1}
    	d_i := \1_{(t_{i-1},t_i]}(y) -\1_{(t_{i-1},t_i]}(\hat y) = \left \{ 
    	\begin{aligned}
            1 & \quad \text {if } y \in (t_{i-1}, t_i] \text{ and } \hat y \notin  (t_{i-1}, t_i], \\
            - 1& \quad \text {if } \hat y \in (t_{i-1}, t_i] \text{ and } y \notin  (t_{i-1}, t_i], \\
    		0 & \quad \text {otherwise}. 
    	\end{aligned}
    	\right.
    \end{equation}
    Secondly, using the local Lipschitz continuity \eqref{eq:F-Lip2}, we can find a constant $\bar \rho >0$ such that
    \begin{equation}\label{eq:close} 
        \| \hat y - y \|_{C(\overline{\Omega})} = \| F({\hat u}) - F(u) \|_{C(\overline{\Omega})} < \epsilon\quad\text{for all }u, \hat u \in \overline  B_{L^2(\Omega)}(u^\dag, \bar \rho).
    \end{equation} 
    Moreover, 
    \begin{multline*}
    	\left \{ y \in (t_{i-1}, t_i], \hat y \notin  (t_{i-1}, t_i], |y -\hat y| < \epsilon  \right\} \\
        \begin{aligned}[t]
            &= \left\{ y \in (t_{i-1}, t_{i-1}+\epsilon), \hat y \in  (y - \epsilon, t_{i-1}] \right\} \cup \left\{ y \in (t_{i}-\epsilon, t_{i}], \hat y \in  (t_i, y +  \epsilon) \right\} \\
            &= \left\{ \hat y \in (t_{i-1} -\epsilon, t_{i-1}], y \in  (t_{i-1}, \hat y + \epsilon) \right\} \cup \left\{\hat y \in (t_{i}, t_{i} + \epsilon),  y \in  (\hat y - \epsilon, t_i] \right\} 
        \end{aligned}
    \end{multline*}
    and
    \begin{multline*}
    	\left\{ \hat y \in (t_{i-1}, t_i], y \notin  (t_{i-1}, t_i], |y -\hat y| < \epsilon  \right\} \\
    	= \left\{ \hat y \in (t_{i-1}, t_{i-1}+\epsilon), y \in  (\hat y - \epsilon, t_{i-1}] \right\} \cup \left\{\hat y \in (t_{i}-\epsilon, t_{i}],  y \in  (t_i, \hat y +  \epsilon) \right\}
    \end{multline*}  
    with the convention that 
    \begin{equation*}
    (-\infty -\epsilon, -\infty) = (-\infty, -\infty + \epsilon) = (+\infty - \epsilon, +\infty) = (+\infty, +\infty +\epsilon)= \emptyset. 
    \end{equation*}  
    Hence, for all $y,\hat y$ satisfying \eqref{eq:close}, we can decompose \eqref{eq:compute-1} into 
    \begin{equation} \label{eq:d_i-term}
        \begin{aligned}[t]
            d_i & = \1_{\left\{ \hat y \in (t_{i-1} -\epsilon, t_{i-1}], y \in  (t_{i-1}, \hat y + \epsilon) \right\}} + \1_{ \left\{\hat y \in (t_{i}, t_{i} + \epsilon),  y \in  (\hat y - \epsilon, t_i] \right\}}  \\  
            \MoveEqLeft[-1] - \1_{\left\{ \hat y \in (t_{i-1},t_{i-1}+\epsilon), y \in (\hat y-\epsilon,t_{i-1}] \right\} } - \1_{\left\{ \hat y \in (t_{i}-\epsilon,t_{i}],  y \in (t_{i},\hat y+\epsilon)\right\}}.  
        \end{aligned}
    \end{equation}    
    Multiplying both sides of \eqref{eq:d_i-term} by $f_i(\hat y)$ and then summing up, we obtain that
    \begin{align*}
        \sum_{i=1}^{k+1} d_i f_i(\hat y) &
        \begin{aligned}[t]
            & = \sum_{i=1}^{k} \left[ \1_{\{\hat y \in (t_{i},t_{i} + \epsilon), y \in (\hat y-\epsilon, t_i] \}} -  \1_{\{ \hat y \in (t_{i}-\epsilon,t_{i}],  y \in (t_{i},\hat y+\epsilon)\}} \right] \left( f_i(\hat y) - f_{i+1}(\hat y  )\right) 
        \end{aligned}
    \end{align*}
    and hence that
    \begin{equation*} 
        \begin{aligned}[t]
            |b_2|
            & \leq  \sum_{i=1}^{k} \left[ \1_{\{\hat y \in (t_{i},t_{i} + \epsilon), y \in (\hat y-\epsilon, t_i] \}} +  \1_{\{ \hat y \in (t_{i}-\epsilon,t_{i}],  y \in (t_{i},\hat y+\epsilon)\}} \right] \left |f_i(\hat y) - f_{i+1}(\hat y  )\right | \\
            & \leq  \sum_{i=1}^{k} \left[ \1_{\{\hat y \in (t_{i},t_{i} + \epsilon), y \in (t_i-\epsilon, t_i] \}} +  \1_{\{ \hat y \in (t_{i}-\epsilon,t_{i}],  y \in (t_{i},t_i+\epsilon)\}} \right] \left |f_i(\hat y) - f_{i+1}(\hat y  )\right | .
        \end{aligned}
    \end{equation*}
    Furthermore, on the set $\{y \in (t_{i}-\epsilon,t_{i}], \hat y \in (t_{i},t_{i}+\epsilon)\}$ we deduce from the non-decreasing monotonicity of $f_i$ and $f_{i+1}$ that
    \begin{equation*}
        f_i(\hat y) \geq f_i(t_i) = f_{i+1}(t_i) \leq f_{i+1}(\hat y),
    \end{equation*} 
    which gives 
    \begin{equation*}
        f_{i+1}(t_i) - f_{i+1}(\hat y) \leq f_i(\hat y) - f_{i+1}(\hat y) \leq f_i(\hat y) - f_i(t_i).
    \end{equation*}
    Consequently,
    \begin{equation*}
        |f_{i}(\hat y) - f_{i+1}(\hat y)| \leq \max \{ |f_{i+1}(t_i) - f_{i+1}(\hat y)|, |f_i(\hat y) - f_i(t_i)|  \}
    \end{equation*} on the set $\{y \in (t_{i}-\epsilon,t_{i}], \hat y \in (t_{i},t_{i}+\epsilon)\}$.
    Combining this with \eqref{fi-Lip1} and \eqref{fi-Lip2} from \cref{lem:lipschitz} yields
    \begin{equation*} 
        \begin{aligned}[t]
            \1_{\{y \in (t_{i}-\epsilon,t_{i}], \hat y \in (t_{i},t_{i}+\epsilon)\} }|f_{i}(\hat y) - f_{i+1}(\hat y)| 
            & \leq \1_{\{y \in (t_{i}-\epsilon,t_{i}], \hat y \in (t_{i},t_{i}+\epsilon)\} } \beta_i |\hat y - t_i| \\
            &  \leq  \1_{\{y \in (t_{i}-\epsilon,t_{i}], \hat y \in (t_{i},t_{i}+\epsilon)\} } \beta_i|\hat y - y|.             
        \end{aligned}
    \end{equation*}
    Similarly, we obtain that
    \begin{equation*} 
        \begin{aligned}[t]
            \1_{\{\hat y \in (t_{i}-\epsilon,t_{i}], y \in (t_{i},t_{i}+\epsilon)\} }|f_{i}(\hat y) - f_{i+1}(\hat y)| 
            &  \leq \beta_i \1_{\{\hat y \in (t_{i}-\epsilon,t_{i}], y \in (t_{i},t_{i}+\epsilon)\} } |\hat y - y|.             
        \end{aligned}
    \end{equation*}
    These inequalities show that 
    \begin{equation*} 
        |b_2| \leq \sum_{i=1}^{k} \beta_i \left( \1_{\{y \in (t_{i}-\epsilon,t_{i}], \hat y \in (t_{i},t_{i}+\epsilon)\} } +\1_{\{\hat y \in (t_{i}-\epsilon,t_{i}], y \in (t_{i},t_{i}+\epsilon)\} } \right) |\hat y - y|.
    \end{equation*}
    Combining this with \eqref{eq:b-term} and \eqref{eq:b1-term} yields  that
    \begin{equation*}
        |b|   \leq \zeta(u,\hat u) |y - \hat y|. 
    \end{equation*}
    We now apply the estimate \eqref{eq:apriori2} to \eqref{eq:omega} to estimate
    \begin{equation*}
        \| \omega \|_{C(\overline{\Omega})} = \norm{G_u b}_{C(\overline{\Omega})} \leq L_p \norm{b}_{L^p(\Omega)}
    \end{equation*}
    for some constant $L_p>0$. 
    From this and the Hölder inequality, we obtain the desired result.  
\end{proof}

We can now verify \cref{ass:gtcc}.
\begin{corollary} \label{cor:gtcc}
    Let $\mu>0$ and assume that $|\{y^\dag = t_i\}|$ is sufficiently small for all $1 \leq i \leq k$. Then there exists a $\rho>0$ such that \eqref{eq:GTCC} holds
    for all $u, \hat u \in \overline  B_{L^2(\Omega)}(u^\dag, \rho)$.
\end{corollary}
\begin{proof}
    Since $|\{y^\dag = t_i\}|$ is sufficiently small for all $1 \leq i \leq k$, there exists a constant $\epsilon > 0$ satisfying \eqref{eq:epsilon_choice1} and 
    \begin{equation}
        \label{eq:epsilon_choice2}
        2L_p |\Omega|^{1/2} \sum_{i=1}^k \left|\left \{ |y^\dag - t_i | < \epsilon\right \} \right|^{1/p'} \beta_i \leq \frac{\mu}{2}
    \end{equation}      
    with $L_p$ and $p'$ as in \cref{lem:GTCC}.
    Let $\bar\rho$ be defined as in \cref{lem:GTCC}. Since $y_u, y_{\hat u} \in B_{C(\overline{\Omega})}(y^\dag ,\epsilon)$ for all $u ,\hat u \in \overline B_{L^2(\Omega)}(u^\dag ,\bar \rho)$ with $y_u := F(u), \hat y := F(\hat u)$, we have that
    \begin{align}
        \label{eq:contain1}
        &\left\{y_u \in (t_{i}-\epsilon,t_{i}], y_{\hat u} \in (t_{i},t_{i}+\epsilon)\right\} \subset  \left\{ |y^\dag - t_i | < \epsilon \right\}, \\
        & \left\{  y_{\hat u} \in (t_{i}-\epsilon,t_{i}],  y_{u} \in (t_{i},t_{i}+\epsilon)\right\} \subset  \left\{ |y^\dag - t_i | < \epsilon \right\},
        \label{eq:contain2}
    \end{align}
    for all $1 \leq i \leq k$. On the other hand, using the continuity of $F$ from $L^2(\Omega)$ to $C(\overline{\Omega})$ and the uniform limit \eqref{eq:infinitesimal}, Lebesgue's dominated convergence theorem implies that the superposition operators $r_{iM}:L^2(\Omega) \to L^{p'}(\Omega)$ defined by \eqref{eq:remainder} satisfy
    \begin{equation*}
        r_{iM}(y_u, y_{\hat u} - y_u) \to 0 \ \text {in} \ L^{p'}(\Omega) \quad \text {as } u, \hat u \to u^\dag \ \text {in} \ L^{2}(\Omega)
    \end{equation*} 
    for all $1 \leq i \leq k+1$. We can thus find a $\rho \in (0, \bar \rho]$ such that 
    \begin{equation}
        \label{eq:choice_rho}
        L_p |\Omega|^{1/2}\sum_{i =1}^{k+1} \norm{r_{iM}(y_u, y_{\hat u} - y_u)}_{L^{p'}(\Omega)} \leq \frac{\mu}{2}
    \end{equation}
    for all $u, \hat u \in \overline B_{L^2(\Omega)}(u^\dag ,\rho)$.
    Using \eqref{eq:epsilon_choice2}, \eqref{eq:contain1}, \eqref{eq:contain2}, and \eqref{eq:choice_rho}, the definition of $\zeta(u, \hat u)$ now ensures that $L_p |\Omega|^{1/2}\norm{\zeta(u, \hat u)}_{L^{p'}(\Omega)} \leq \mu$
    for all $u, \hat u \in \overline B_{L^2(\Omega)}(u^\dag ,\rho)$. 
    The generalized tangential cone condition \eqref{eq:GTCC} then follows from \cref{lem:GTCC}.
\end{proof}

\printbibliography

\end{document}

%% file: rel_error_free_zero.tikz
\begin{tikzpicture}[baseline,style={font=\small}]
    \begin{semilogyaxis}[
        width=\textwidth,
        xmin=0,
        xmax=1000,
        ymin=0,
        ymax=1.0,
        ]
        \addplot [color=DarkBlue,line width=1.5pt]
        table[row sep=crcr]{%
            0     1.0    \\
            1     0.7734033447504293    \\
            2     0.6276054854326919    \\
            3     0.535142519173757    \\
            4     0.47753055928678706    \\
            5     0.44138975878509384    \\
            6     0.41784094703565183    \\
            7     0.4014678241274383    \\
            8     0.38916783441624275    \\
            9     0.3792320109530731    \\
            10     0.3707373601214417    \\
            11     0.3631874697953304    \\
            12     0.35631261250766333    \\
            13     0.34996224818869737    \\
            14     0.344048281813025    \\
            15     0.3385154119750958    \\
            16     0.3333257268710733    \\
            17     0.32845072399799097    \\
            18     0.32386718304314044    \\
            19     0.3195550290181793    \\
            20     0.31549622340981026    \\
            21     0.3116741831726377    \\
            22     0.30807347150848985    \\
            23     0.30467962694137923    \\
            24     0.30147906326322466    \\
            25     0.2984590048443487    \\
            26     0.29560743989209987    \\
            27     0.2929130822862602    \\
            28     0.2903653377220499    \\
            29     0.28795427176599997    \\
            30     0.2856705789536292    \\
            31     0.28350555239583863    \\
            32     0.2814510538574763    \\
            33     0.2794994842610377    \\
            34     0.27764375474708247    \\
            35     0.27587725835476407    \\
            36     0.27419384245463496    \\
            37     0.2725877820041271    \\
            38     0.2710537537216475    \\
            39     0.2695868112294499    \\
            40     0.26818236122275657    \\
            41     0.26683614069043454    \\
            42     0.2655441952122807    \\
            43     0.2643028583359482    \\
            44     0.26310873203383345    \\
            45     0.2619586682252993    \\
            46     0.26084975134666377    \\
            47     0.259779281941815    \\
            48     0.2587447612439702    \\
            49     0.25774387671297755    \\
            50     0.2567744884916887    \\
            51     0.2558346167411614    \\
            52     0.2549224298144604    \\
            53     0.2540362332271752    \\
            54     0.25317445938328637    \\
            55     0.2523356580149018    \\
            56     0.2515184872951739    \\
            57     0.2507217055846384    \\
            58     0.24994416377240022    \\
            59     0.2491847981748142    \\
            60     0.24844262395597008    \\
            61     0.24771672903568723    \\
            62     0.24700626845237697    \\
            63     0.246310459149876    \\
            64     0.24562857515882566    \\
            65     0.24495994314501215    \\
            66     0.24430393829843583    \\
            67     0.24365998053881216    \\
            68     0.24302753101425864    \\
            69     0.24240608887187737    \\
            70     0.24179518827991273    \\
            71     0.24119439568282583    \\
            72     0.24060330727169468    \\
            73     0.2400215466537185    \\
            74     0.2394487627055498    \\
            75     0.2388846275964396    \\
            76     0.238328834968163    \\
            77     0.23778109825944757    \\
            78     0.23724114916380662    \\
            79     0.23670873621037009    \\
            80     0.23618362345788319    \\
            81     0.23566558929329257    \\
            82     0.23515442532637237    \\
            83     0.23464993537299728    \\
            84     0.2341519345199156    \\
            85     0.23366024826461557    \\
            86     0.23317471172425888    \\
            87     0.2326951689081564    \\
            88     0.23222147204877103    \\
            89     0.23175348098645865    \\
            90     0.23129106260370483    \\
            91     0.23083409030483673    \\
            92     0.23038244353758336    \\
            93     0.22993600735304445    \\
            94     0.22949467200106805    \\
            95     0.229058332558027    \\
            96     0.22862688858454253    \\
            97     0.2282002438105864    \\
            98     0.22777830584581255    \\
            99     0.22736098591305265    \\
            100     0.2269481986030922    \\
            101     0.22653986164900713    \\
            102     0.22613589571843928    \\
            103     0.22573622422240228    \\
            104     0.2253407731392032    \\
            105     0.22494947085232583    \\
            106     0.22456224800105934    \\
            107     0.22417903734289668    \\
            108     0.22379977362670542    \\
            109     0.2234243934757971    \\
            110     0.22305283528010242    \\
            111     0.2226850390966929    \\
            112     0.2223209465579619    \\
            113     0.2219605007868874    \\
            114     0.22160364631870647    \\
            115     0.22125032902858302    \\
            116     0.220900496064709    \\
            117     0.2205540957864092    \\
            118     0.2202110777068561    \\
            119     0.21987139244002285    \\
            120     0.21953499165150164    \\
            121     0.21920182801288451    \\
            122     0.21887185515944915    \\
            123     0.21854502765081854    \\
            124     0.21822130093441452    \\
            125     0.21790063131143242    \\
            126     0.217582975905152    \\
            127     0.21726829263141442    \\
            128     0.2169565401710543    \\
            129     0.21664767794413403    \\
            130     0.2163416660858768    \\
            131     0.216038465424108    \\
            132     0.21573803745812845    \\
            133     0.21544034433883893    \\
            134     0.2151453488501396    \\
            135     0.21485301439134513    \\
            136     0.21456330496066847    \\
            137     0.21427618513960897    \\
            138     0.21399162007820827    \\
            139     0.2137095754811193    \\
            140     0.2134300175943512    \\
            141     0.21315291319276145    \\
            142     0.21287822956810415    \\
            143     0.21260593451768303    \\
            144     0.21233599633355638    \\
            145     0.21206838379216722    \\
            146     0.2118030661445152    \\
            147     0.211540013106682    \\
            148     0.21127919485075533    \\
            149     0.21102058199616053    \\
            150     0.21076414560125734    \\
            151     0.21050985715531836    \\
            152     0.21025768857073013    \\
            153     0.21000761217552613    \\
            154     0.20975960070613223    \\
            155     0.20951362730035195    \\
            156     0.2092696654905708    \\
            157     0.20902768919718742    \\
            158     0.20878767272218854    \\
            159     0.20854959074296203    \\
            160     0.208313418306228    \\
            161     0.20807913082216495    \\
            162     0.20784670405864655    \\
            163     0.20761611413566672    \\
            164     0.20738733751985142    \\
            165     0.20716035101912353    \\
            166     0.2069351317774906    \\
            167     0.20671165726989876    \\
            168     0.20648990529727151    \\
            169     0.20626985398158024    \\
            170     0.2060514817610449    \\
            171     0.20583476738542328    \\
            172     0.20561968991137744    \\
            173     0.20540622869796005    \\
            174     0.2051943634021162    \\
            175     0.2049840739743513    \\
            176     0.20477534065438294    \\
            177     0.2045681439669412    \\
            178     0.20436246471758537    \\
            179     0.2041582839886258    \\
            180     0.20395558313509476    \\
            181     0.2037543437807652    \\
            182     0.203554547814275    \\
            183     0.20335617738525183    \\
            184     0.20315921490055938    \\
            185     0.2029636430205499    \\
            186     0.20276944465540148    \\
            187     0.20257660296148872    \\
            188     0.20238510133783655    \\
            189     0.20219492342258494    \\
            190     0.20200605308953062    \\
            191     0.20181847444472945    \\
            192     0.20163217182311177    \\
            193     0.2014471297851655    \\
            194     0.20126333311368155    \\
            195     0.20108076681050732    \\
            196     0.20089941609337736    \\
            197     0.2007192663927819    \\
            198     0.200540303348859    \\
            199     0.20036251280835252    \\
            200     0.20018588082161595    \\
            201     0.20001039363962195    \\
            202     0.19983603771105724    \\
            203     0.19966279967942518    \\
            204     0.19949066638020388    \\
            205     0.19931962483804738    \\
            206     0.19914966226401348    \\
            207     0.19898076605281295    \\
            208     0.1988129237801593    \\
            209     0.19864612320007893    \\
            210     0.19848035224230817    \\
            211     0.1983155990097108    \\
            212     0.19815185177571712    \\
            213     0.197989098981834    \\
            214     0.19782732923515364    \\
            215     0.19766653130591721    \\
            216     0.1975066941251012    \\
            217     0.1973478067820438    \\
            218     0.19718985852211354    \\
            219     0.19703283874438363    \\
            220     0.19687673699935257    \\
            221     0.19672154298672695    \\
            222     0.19656724655317268    \\
            223     0.19641383769015147    \\
            224     0.19626130653177193    \\
            225     0.19610964335265346    \\
            226     0.19595883856583593    \\
            227     0.1958088827207381    \\
            228     0.1956597665010878    \\
            229     0.19551148072295618    \\
            230     0.19536401633274994    \\
            231     0.19521736440528553    \\
            232     0.1950715161418614    \\
            233     0.19492646286834334    \\
            234     0.19478219603335103    \\
            235     0.19463870720636495    \\
            236     0.19449598807594415    \\
            237     0.19435403044792457    \\
            238     0.19421282624367517    \\
            239     0.19407236749834628    \\
            240     0.19393264635917126    \\
            241     0.19379365508376378    \\
            242     0.19365538603848426    \\
            243     0.19351783169677636    \\
            244     0.19338098463756265    \\
            245     0.19324483754366406    \\
            246     0.19310938320022253    \\
            247     0.19297461449317083    \\
            248     0.19284052440768135    \\
            249     0.19270710602670077    \\
            250     0.19257435252947197    \\
            251     0.19244225719004546    \\
            252     0.1923108133758862    \\
            253     0.1921800145464445    \\
            254     0.19204985425174986    \\
            255     0.1919203261310683    \\
            256     0.19179142391153692    \\
            257     0.19166314140682908    \\
            258     0.19153547251585826    \\
            259     0.19140841122147437    \\
            260     0.19128195158919167    \\
            261     0.19115608776595716    \\
            262     0.1910308139788855    \\
            263     0.19090612453406425    \\
            264     0.19078201381534934    \\
            265     0.19065847628318663    \\
            266     0.19053550647345557    \\
            267     0.1904130989963167    \\
            268     0.19029124853508814    \\
            269     0.19016994984514143    \\
            270     0.1900491977527917    \\
            271     0.18992898715424544    \\
            272     0.18980931301452184    \\
            273     0.1896901703664122    \\
            274     0.18957155430945338    \\
            275     0.18945346000892263    \\
            276     0.1893358826948228    \\
            277     0.18921881766091028    \\
            278     0.1891022602637354    \\
            279     0.18898620592167903    \\
            280     0.1888706501140176    \\
            281     0.1887555883799928    \\
            282     0.1886410163179142    \\
            283     0.18852692958426304    \\
            284     0.18841332389279605    \\
            285     0.18830019501369302    \\
            286     0.18818753877269337    \\
            287     0.18807535105025602    \\
            288     0.1879636277807318    \\
            289     0.18785236495154306    \\
            290     0.18774155860238387    \\
            291     0.18763120482443657    \\
            292     0.187521299759571    \\
            293     0.18741183959959132    \\
            294     0.18730282058549416    \\
            295     0.18719423900669982    \\
            296     0.18708609120032738    \\
            297     0.18697837355048758    \\
            298     0.18687108248756593    \\
            299     0.18676421448750635    \\
            300     0.18665776607115256    \\
            301     0.18655173380355308    \\
            302     0.18644611429330085    \\
            303     0.1863409041918743    \\
            304     0.1862361001929901    \\
            305     0.1861316990319692    \\
            306     0.18602769748511683    \\
            307     0.18592409236909577    \\
            308     0.1858208805403212    \\
            309     0.18571805889437307    \\
            310     0.1856156243653977    \\
            311     0.18551357392553106    \\
            312     0.18541190458433393    \\
            313     0.18531061338823002    \\
            314     0.18520969741994406    \\
            315     0.1851091537979723    \\
            316     0.18500897967603658    \\
            317     0.18490917224256614    \\
            318     0.18480972872017654    \\
            319     0.18471064636514736    \\
            320     0.18461192246694705    \\
            321     0.18451355434770725    \\
            322     0.18441553936176464    \\
            323     0.18431787489515147    \\
            324     0.18422055836515877    \\
            325     0.184123587219845    \\
            326     0.18402695893758578    \\
            327     0.18393067102663105    \\
            328     0.18383472102465628    \\
            329     0.18373910649833206    \\
            330     0.18364382504289067    \\
            331     0.18354887428170524    \\
            332     0.18345425186586672    \\
            333     0.1833599554737938    \\
            334     0.18326598281081313    \\
            335     0.1831723316087699    \\
            336     0.18307899962563526    \\
            337     0.18298598464512575    \\
            338     0.18289328447631953    \\
            339     0.1828008969532992    \\
            340     0.18270881993476343    \\
            341     0.1826170513036867    \\
            342     0.1825255889669511    \\
            343     0.18243443085500602    \\
            344     0.1823435749215161    \\
            345     0.1822530191430398    \\
            346     0.182162761518672    \\
            347     0.1820728000697384    \\
            348     0.18198313283945924    \\
            349     0.18189375789263923    \\
            350     0.181804673315355    \\
            351     0.18171587721463764    \\
            352     0.18162736771818422    \\
            353     0.18153914297405807    \\
            354     0.18145120115037888    \\
            355     0.18136354043506017    \\
            356     0.18127615903549596    \\
            357     0.18118905517831743    \\
            358     0.18110222710907872    \\
            359     0.1810156730920142    \\
            360     0.18092939140975542    \\
            361     0.18084338036308714    \\
            362     0.1807576382706665    \\
            363     0.18067216346878945    \\
            364     0.18058695431112173    \\
            365     0.18050200916846634    \\
            366     0.1804173264285226    \\
            367     0.1803329044956309    \\
            368     0.18024874179055558    \\
            369     0.18016483675025102    \\
            370     0.1800811878276199    \\
            371     0.17999779349130482    \\
            372     0.1799146522254701    \\
            373     0.17983176252956437    \\
            374     0.17974912291812717    \\
            375     0.17966673192056992    \\
            376     0.17958458808097125    \\
            377     0.1795026899578745    \\
            378     0.17942103612408467    \\
            379     0.17933962516647062    \\
            380     0.17925845568577337    \\
            381     0.17917752629641262    \\
            382     0.1790968356263037    \\
            383     0.17901638231666106    \\
            384     0.17893616502183096    \\
            385     0.17885618240909984    \\
            386     0.17877643315852645    \\
            387     0.17869691596275206    \\
            388     0.17861762952685656    \\
            389     0.17853857256816627    \\
            390     0.1784597438160889    \\
            391     0.1783811420119699    \\
            392     0.1783027659089096    \\
            393     0.1782246142716192    \\
            394     0.17814668587625732    \\
            395     0.17806897951028258    \\
            396     0.17799149397229477    \\
            397     0.17791422807189433    \\
            398     0.17783718062952772    \\
            399     0.1777603504763565    \\
            400     0.1776837364540991    \\         
            401     0.17760733741490117    \\
            402     0.1775311522211955    \\
            403     0.17745517974557157    \\
            404     0.17737941887063532    \\
            405     0.17730386848888155    \\
            406     0.1772285275025548    \\
            407     0.177153394823551    \\
            408     0.1770784693732498    \\
            409     0.17700375008243074    \\
            410     0.17692923589113255    \\
            411     0.1768549257485318    \\
            412     0.1767808186128259    \\
            413     0.17670691345113584    \\
            414     0.17663320923936512    \\
            415     0.17655970496210238    \\
            416     0.1764863996125147    \\
            417     0.17641329219222282    \\
            418     0.17634038171121091    \\
            419     0.17626766718771014    \\
            420     0.1761951476480981    \\
            421     0.1761228221268025    \\
            422     0.17605068966618748    \\
            423     0.1759787493164608    \\
            424     0.17590700013558577    \\
            425     0.17583544118917    \\
            426     0.17576407155036924    \\
            427     0.1756928902998209    \\
            428     0.17562189652551027    \\
            429     0.1755510893227233    \\
            430     0.1754804677939168    \\
            431     0.17541003104866676    \\
            432     0.17533977820356136    \\
            433     0.17526970838211142    \\
            434     0.17519982071468915    \\
            435     0.17513011433842304    \\
            436     0.1750605883971221    \\
            437     0.17499124204120847    \\
            438     0.17492207442762392    \\
            439     0.17485308471975572    \\
            440     0.1747842720873653    \\
            441     0.17471563570650392    \\
            442     0.17464717475944924    \\
            443     0.17457888843462538    \\
            444     0.17451077592652298    \\
            445     0.1744428364356499    \\
            446     0.17437506916844414    \\
            447     0.17430747333720467    \\
            448     0.17424004816003671    \\
            449     0.17417279286076978    \\
            450     0.17410570666889447    \\
            451     0.17403878881952156    \\
            452     0.17397203855327403    \\
            453     0.17390545511625613    \\
            454     0.17383903775999068    \\
            455     0.17377278574135238    \\
            456     0.1737066983224892    \\
            457     0.17364077477079068    \\
            458     0.1735750143588286    \\
            459     0.17350941636426423    \\
            460     0.17344398006984008    \\
            461     0.1733787047632854    \\
            462     0.17331358973729089    \\
            463     0.17324863428942625    \\
            464     0.1731838377221013    \\
            465     0.173119199342519    \\
            466     0.17305471846261655    \\
            467     0.17299039439900718    \\
            468     0.17292622647293696    \\
            469     0.172862214010245    \\
            470     0.17279835634129548    \\
            471     0.1727346528009339    \\
            472     0.17267110272845776    \\
            473     0.17260770546753787    \\
            474     0.17254446036620666    \\
            475     0.17248136677678277    \\
            476     0.17241842405584756    \\
            477     0.17235563156418016    \\
            478     0.17229298866673975    \\
            479     0.1722304947326006    \\
            480     0.17216814913491837    \\
            481     0.17210595125088599    \\
            482     0.17204390046170204    \\
            483     0.17198199615250814    \\
            484     0.1719202377123686    \\
            485     0.1718586245342333    \\
            486     0.17179715601486925    \\
            487     0.17173583155487085    \\
            488     0.17167465055856884    \\
            489     0.17161361243403775    \\
            490     0.17155271659302082    \\
            491     0.17149196245092546    \\
            492     0.17143134942677138    \\
            493     0.17137087694315747    \\
            494     0.17131054442622537    \\
            495     0.1712503513056318    \\
            496     0.1711902970145009    \\
            497     0.17113038098941802    \\
            498     0.17107060267036098    \\
            499     0.171010961500693    \\
            500     0.17095145692711836    \\
            501     0.17089208839966508    \\
            502     0.17083285537163526    \\
            503     0.1707737572995827    \\
            504     0.1707147936432865    \\
            505     0.17065596386572002    \\
            506     0.1705972674330076    \\
            507     0.17053870381441438    \\
            508     0.17048027248231593    \\
            509     0.17042197291214692    \\
            510     0.170363804582407    \\
            511     0.17030576697460023    \\
            512     0.17024785957323327    \\
            513     0.17019008186577897    \\
            514     0.17013243334263822    \\
            515     0.17007491349714188    \\
            516     0.17001752182548766    \\
            517     0.16996025782675256    \\
            518     0.16990312100284857    \\
            519     0.16984611085847776    \\
            520     0.16978922690115497    \\
            521     0.1697324686411495    \\
            522     0.16967583559146185    \\
            523     0.16961932726780884    \\
            524     0.16956294318861492    \\
            525     0.1695066828749486    \\
            526     0.169450545850541    \\
            527     0.16939453164174545    \\
            528     0.16933863977751057    \\
            529     0.1692828697893692    \\
            530     0.1692272212114118    \\
            531     0.16917169358026302    \\
            532     0.1691162864350605    \\
            533     0.16906099931744742    \\
            534     0.16900583177152725    \\
            535     0.1689507833438671    \\
            536     0.16889585358346318    \\
            537     0.1688410420417285    \\
            538     0.16878634827246838    \\
            539     0.1687317718318692    \\
            540     0.16867731227846636    \\
            541     0.16862296917312894    \\
            542     0.1685687420790613    \\
            543     0.16851463056175114    \\
            544     0.1684606341889795    \\
            545     0.16840675253078277    \\
            546     0.168352985159451    \\
            547     0.16829933164950336    \\
            548     0.16824579157766428    \\
            549     0.16819236452286024    \\
            550     0.16813905006618793    \\
            551     0.16808584779091182    \\
            552     0.1680327572824425    \\
            553     0.167979778128308    \\
            554     0.167926909918161    \\
            555     0.16787415224374297    \\
            556     0.1678215046988757    \\
            557     0.16776896687945247    \\
            558     0.16771653838341022    \\
            559     0.16766421881072466    \\
            560     0.1676120077633826    \\
            561     0.1675599048453876    \\
            562     0.16750790966272383    \\
            563     0.167456021823353    \\
            564     0.167404240937198    \\
            565     0.16735256661612938    \\
            566     0.1673009984739555    \\
            567     0.1672495361263908    \\
            568     0.1671981791910652    \\
            569     0.16714692728749872    \\
            570     0.16709578003709363    \\
            571     0.16704473706309766    \\
            572     0.1669937979906362    \\
            573     0.16694296244665693    \\
            574     0.1668922300599411    \\
            575     0.1668416004610791    \\
            576     0.1667910732824634    \\
            577     0.16674064815827358    \\
            578     0.16669032472446793    \\
            579     0.1666401026187687    \\
            580     0.1665899814806414    \\
            581     0.16653996095130785    \\
            582     0.16649004067370402    \\
            583     0.1664402202924869    \\
            584     0.16639049945401543    \\
            585     0.166340877806349    \\
            586     0.16629135499922257    \\
            587     0.16624193068404644    \\
            588     0.1661926045138873    \\
            589     0.16614337614346394    \\
            590     0.16609424522913033    \\
            591     0.16604521142886758    \\
            592     0.16599627440228296    \\
            593     0.16594743381057497    \\
            594     0.1658986893165456    \\
            595     0.16585004058458774    \\
            596     0.1658014872806584    \\
            597     0.16575302907229325    \\
            598     0.16570466562856534    \\
            599     0.16565639662011228    \\
            600     0.16560822171909104    \\   
            601     0.16556014059919963    \\
            602     0.16551215293563634    \\
            603     0.16546425840512108    \\
            604     0.16541645668586172    \\
            605     0.16536874745755512    \\
            606     0.1653211304013781    \\
            607     0.165273605199979    \\
            608     0.16522617153746672    \\
            609     0.16517882909939682    \\
            610     0.1651315775727732    \\
            611     0.16508441664602638    \\
            612     0.16503734600901757    \\
            613     0.16499036535302597    \\
            614     0.1649434743707319    \\
            615     0.16489667275622422    \\
            616     0.16484996020496934    \\
            617     0.16480333641383016    \\
            618     0.16475680108103347    \\
            619     0.1647103539061776    \\
            620     0.1646639945902203    \\
            621     0.16461772283546391    \\
            622     0.16457153834555374    \\
            623     0.1645254408254731    \\
            624     0.16447942998152693    \\
            625     0.16443350552134456    \\
            626     0.16438766715385436    \\
            627     0.16434191458930453    \\
            628     0.16429624753922364    \\
            629     0.16425066571643804    \\
            630     0.16420516883504868    \\
            631     0.16415975661043186    \\
            632     0.16411442875923068    \\
            633     0.1640691849993491    \\
            634     0.16402402504993663    \\
            635     0.16397894863139076    \\
            636     0.16393395546534695    \\
            637     0.16388904527466783    \\
            638     0.16384421778344482    \\
            639     0.16379947271698064    \\
            640     0.16375480980179138    \\
            641     0.1637102287655887    \\
            642     0.16366572933729215    \\
            643     0.16362131124699894    \\
            644     0.163576974226002    \\
            645     0.16353271800675953    \\
            646     0.16348854232290133    \\
            647     0.16344444690922583    \\
            648     0.16340043150168576    \\
            649     0.16335649583738018    \\
            650     0.16331263965456103    \\
            651     0.16326886269261143    \\
            652     0.16322516469204945    \\
            653     0.16318154539451768    \\
            654     0.16313800454277946    \\
            655     0.16309454188071257    \\
            656     0.16305115715330087    \\
            657     0.16300785010662064    \\
            658     0.16296462048786972    \\
            659     0.1629214680453066    \\
            660     0.16287839252829064    \\
            661     0.16283539368725686    \\
            662     0.1627924712737086    \\
            663     0.16274962504022264    \\
            664     0.1627068547404267    \\
            665     0.16266416012901463    \\
            666     0.16262154096172846    \\
            667     0.16257899699534387    \\
            668     0.1625365279876893    \\
            669     0.16249413369762195    \\
            670     0.1624518138850235    \\
            671     0.16240956831080106    \\
            672     0.16236739673688372    \\
            673     0.16232529892620307    \\
            674     0.1622832746427033    \\
            675     0.16224132365133148    \\
            676     0.16219944571802816    \\
            677     0.16215764060972968    \\
            678     0.16211590809434756    \\
            679     0.16207424794079203    \\
            680     0.16203265991892918    \\
            681     0.16199114379961543    \\
            682     0.1619496993546583    \\
            683     0.16190832635683577    \\
            684     0.16186702457988122    \\
            685     0.16182579379846906    \\
            686     0.16178463378823887    \\
            687     0.1617435443257546    \\
            688     0.16170252518853243    \\
            689     0.16166157615500132    \\
            690     0.16162069700453505    \\
            691     0.1615798875174311    \\
            692     0.16153914747488918    \\
            693     0.16149847665903588    \\
            694     0.16145787485290725    \\
            695     0.1614173418404366    \\
            696     0.16137687740646114    \\
            697     0.16133648133671258    \\
            698     0.16129615341781736    \\
            699     0.16125589343728355    \\
            700     0.16121570118350495    \\
            701     0.16117557644575442    \\
            702     0.1611355190141765    \\
            703     0.1610955286797766    \\
            704     0.1610556052344438    \\
            705     0.1610157484709128    \\
            706     0.16097595818277882    \\
            707     0.16093623416449476    \\
            708     0.16089657621135184    \\
            709     0.16085698411948965    \\
            710     0.16081745768589625    \\
            711     0.16077799670837942    \\
            712     0.16073860098558918    \\
            713     0.1606992703170064    \\
            714     0.16066000450292142    \\
            715     0.16062080334445492    \\
            716     0.16058166664354448    \\
            717     0.1605425942029364    \\
            718     0.1605035858261798    \\
            719     0.16046464131763497    \\
            720     0.16042576048246104    \\
            721     0.16038694312660637    \\
            722     0.1603481890568203    \\
            723     0.16030949808063727    \\
            724     0.1602708700063754    \\
            725     0.16023230464313748    \\
            726     0.16019380180079526    \\
            727     0.16015536129000474    \\
            728     0.16011698292218335    \\
            729     0.16007866650951708    \\
            730     0.160040411864956    \\
            731     0.16000221880220888    \\
            732     0.15996408713573845    \\
            733     0.15992601668075962    \\
            734     0.15988800725323435    \\
            735     0.15985005866986798    \\
            736     0.15981217074811155    \\
            737     0.15977434330615164    \\
            738     0.1597365761629032    \\
            739     0.1596988691380187    \\
            740     0.15966122205187772    \\
            741     0.15962363472557625    \\
            742     0.15958610698094228    \\
            743     0.15954863864050742    \\
            744     0.15951122952752822    \\
            745     0.15947387946595926    \\
            746     0.1594365882804787    \\
            747     0.15939935579645167    \\
            748     0.15936218183994846    \\
            749     0.1593250662377472    \\
            750     0.15928800881729988    \\
            751     0.15925100940676745    \\
            752     0.15921406783498823    \\
            753     0.1591771839314851    \\
            754     0.15914035752646302    \\
            755     0.1591035884508063    \\
            756     0.15906687653607285    \\
            757     0.15903022161449093    \\
            758     0.15899362351895638    \\
            759     0.15895708208303047    \\
            760     0.15892059714094503    \\
            761     0.158884168527569    \\
            762     0.1588477960784572    \\
            763     0.15881147962979134    \\
            764     0.15877521901841887    \\
            765     0.15873901408182584    \\
            766     0.15870286465815    \\
            767     0.15866677058615739    \\
            768     0.15863073170526878    \\
            769     0.15859474785552607    \\
            770     0.15855881887760778    \\
            771     0.1585229446128252    \\
            772     0.1584871249031133    \\
            773     0.1584513595910282    \\
            774     0.15841564851975112    \\
            775     0.1583799915330801    \\
            776     0.15834438847542384    \\
            777     0.15830883919180905    \\
            778     0.1582733435278673    \\
            779     0.1582379013298401    \\
            780     0.15820251244457353    \\
            781     0.15816717671950623    \\
            782     0.158131894002689    \\
            783     0.15809666414275617    \\
            784     0.15806148698893976    \\
            785     0.15802636239106585    \\
            786     0.15799129019953756    \\
            787     0.15795627026535605    \\
            788     0.15792130244009092    \\
            789     0.157886386575906    \\
            790     0.15785152252552911    \\
            791     0.15781671014226784    \\
            792     0.157781949280003    \\
            793     0.15774723979318472    \\
            794     0.1577125815368239    \\
            795     0.15767797436650172    \\
            796     0.15764341813835706    \\
            797     0.15760891270909122    \\
            798     0.15757445793595867    \\
            799     0.15754005367676524    \\
            800     0.15750569978987627    \\
            801     0.15747139613420416    \\
            802     0.1574371425691953    \\
            803     0.15740293895486088    \\
            804     0.157368785151734    \\
            805     0.15733468102089762    \\
            806     0.1573006264239723    \\
            807     0.157266621223107    \\
            808     0.15723266528098623    \\
            809     0.1571987584608226    \\
            810     0.15716490062635308    \\
            811     0.15713109164184677    \\
            812     0.15709733137208928    \\
            813     0.15706361968238589    \\
            814     0.15702995643856635    \\
            815     0.15699634150696462    \\
            816     0.1569627747544415    \\
            817     0.15692925604835428    \\
            818     0.15689578525657905    \\
            819     0.15686236224749317    \\
            820     0.15682898688998173    \\
            821     0.1567956590534277    \\
            822     0.15676237860771625    \\
            823     0.15672914542323022    \\
            824     0.15669595937084316    \\
            825     0.15666282032192475    \\
            826     0.15662972814833734    \\
            827     0.15659668272242908    \\
            828     0.15656368391703585    \\
            829     0.15653073160547148    \\
            830     0.15649782566154585    \\
            831     0.15646496595953044    \\
            832     0.15643215237419356    \\
            833     0.15639938478076007    \\
            834     0.1563666630549439    \\
            835     0.15633398707292345    \\
            836     0.15630135671134632    \\
            837     0.15626877184733007    \\
            838     0.1562362323584537    \\
            839     0.1562037381227632    \\
            840     0.1561712890187608    \\
            841     0.15613888492541445    \\
            842     0.15610652572214437    \\
            843     0.1560742112888262    \\
            844     0.15604194150579057    \\
            845     0.15600971625381496    \\
            846     0.15597753541413212    \\
            847     0.15594539886841235    \\
            848     0.1559133064987854    \\
            849     0.15588125818781073    \\
            850     0.15584925381849643    \\
            851     0.15581729327428281    \\
            852     0.15578537643905632    \\
            853     0.15575350319713271    \\
            854     0.15572167343326096    \\
            855     0.15568988703262618    \\
            856     0.15565814388084    \\
            857     0.1556264438639412    \\
            858     0.1555947868683943    \\
            859     0.15556317278108986    \\
            860     0.15553160148934203    \\
            861     0.15550007288087797    \\
            862     0.15546858684385112    \\
            863     0.15543714326682528    \\
            864     0.15540574203878715    \\
            865     0.15537438304912704    \\
            866     0.15534306618765234    \\
            867     0.155311791344579    \\
            868     0.1552805584105268    \\
            869     0.15524936727652844    \\
            870     0.15521821783401515    \\
            871     0.1551871099748168    \\
            872     0.1551560435911748    \\
            873     0.15512501857571961    \\
            874     0.15509403482148448    \\
            875     0.15506309222189554    \\
            876     0.1550321906707703    \\
            877     0.15500133006232053    \\
            878     0.154970510291148    \\
            879     0.15493973125224741    \\
            880     0.15490899284098592    \\
            881     0.15487829495313524    \\
            882     0.15484763748483357    \\
            883     0.15481702033261283    \\
            884     0.15478644339337533    \\
            885     0.15475590656440927    \\
            886     0.15472540974337712    \\
            887     0.15469495282831172    \\
            888     0.15466453571763036    \\
            889     0.15463415831010952    \\
            890     0.15460382050490468    \\
            891     0.15457352220153817    \\
            892     0.1545432632998919    \\
            893     0.15451304370023    \\
            894     0.15448286330315922    \\
            895     0.15445272200966653    \\
            896     0.15442261972109048    \\
            897     0.15439255633912782    \\
            898     0.1543625317658427    \\
            899     0.15433254590364281    \\
            900     0.15430259865529344    \\
            901     0.1542726899239253    \\
            902     0.15424281961299918    \\
            903     0.15421298762634886    \\
            904     0.1541831938681327    \\
            905     0.15415343824287733    \\
            906     0.15412372065543994    \\
            907     0.15409404101103022    \\
            908     0.1540643992151964    \\
            909     0.15403479517382979    \\
            910     0.15400522879315875    \\
            911     0.15397569997974872    \\
            912     0.15394620864050773    \\
            913     0.15391675468267585    \\
            914     0.15388733801381965    \\
            915     0.1538579585418502    \\
            916     0.15382861617500082    \\
            917     0.15379931082183876    \\
            918     0.1537700423912545    \\
            919     0.1537408107924719    \\
            920     0.15371161593503185    \\
            921     0.15368245772880734    \\
            922     0.15365333608398524    \\
            923     0.15362425091107762    \\
            924     0.15359520212091968    \\
            925     0.15356618962465898    \\
            926     0.15353721333376344    \\
            927     0.15350827316001775    \\
            928     0.1534793690155122    \\
            929     0.15345050081265899    \\
            930     0.15342166846418223    \\
            931     0.15339287188310685    \\
            932     0.15336411098277777    \\
            933     0.15333538567683694    \\
            934     0.1533066958792398    \\
            935     0.15327804150424493    \\
            936     0.15324942246641657    \\
            937     0.1532208386806151    \\
            938     0.15319229006200621    \\
            939     0.153163776526053    \\
            940     0.1531352979885233    \\
            941     0.15310685436547106    \\
            942     0.15307844557325867    \\
            943     0.15305007152853223    \\
            944     0.15302173214823472    \\
            945     0.15299342734960475    \\
            946     0.15296515705016947    \\
            947     0.1529369211677438    \\
            948     0.1529087196204313    \\
            949     0.1528805523266281    \\
            950     0.1528524192050063    \\
            951     0.15282432017453618    \\
            952     0.15279625515445644    \\
            953     0.15276822406429905    \\
            954     0.15274022682387198    \\
            955     0.1527122633532646    \\
            956     0.15268433357284716    \\
            957     0.15265643740326343    \\
            958     0.15262857476543418    \\
            959     0.15260074558056064    \\
            960     0.15257294977010766    \\
            961     0.1525451872558221    \\
            962     0.15251745795972552    \\
            963     0.15248976180409315    \\
            964     0.15246209871148794    \\
            965     0.15243446860473114    \\
            966     0.15240687140691267    \\
            967     0.1523793070413943    \\
            968     0.1523517754317922    \\
            969     0.15232427650199135    \\
            970     0.15229681017614582    \\
            971     0.1522693763786612    \\
            972     0.1522419750342102    \\
            973     0.15221460606771522    \\
            974     0.1521872694043716    \\
            975     0.1521599649696203    \\
            976     0.15213269268916466    \\
            977     0.15210545248895682    \\
            978     0.15207824429520758    \\
            979     0.15205106803438104    \\
            980     0.15202392363318742    \\
            981     0.15199681101859822    \\
            982     0.1519697301178179    \\
            983     0.1519426808583204    \\
            984     0.15191566316780933    \\
            985     0.15188867697424785    \\
            986     0.15186172220582655    \\
            987     0.15183479879100692    \\
            988     0.15180790665847008    \\
            989     0.15178104573715223    \\
            990     0.15175421595622843    \\
            991     0.15172741724510586    \\
            992     0.15170064953344803    \\
            993     0.15167391275114261    \\
            994     0.1516472068283185    \\
            995     0.15162053169534256    \\
            996     0.1515938872828159    \\
            997     0.15156727352157628    \\
            998     0.15154069034269027    \\
            999     0.15151413767745722    \\
            1000     0.15148761545741557    \\
        };
    \end{semilogyaxis}
\end{tikzpicture}%

%% file: rel_error_free_source.tikz
\begin{tikzpicture}[baseline,style={font=\small}]
    \begin{semilogyaxis}[
        width=\textwidth,
        xmin=0,
        xmax=100,
        ymin=0,
        ymax=3.5,
        ]
        \addplot [color=DarkBlue,line width=1.5pt]
        table[row sep=crcr]{%
            0     3.3346036537920054    \\
            1     2.3683363451320973    \\
            2     1.6825383938838403    \\
            3     1.1948447161748323    \\
            4     0.8490682492572671    \\
            5     0.6029242888392988    \\
            6     0.4282698198685035    \\
            7     0.30419634160969267    \\
            8     0.21607301119239927    \\
            9     0.1534778268751937    \\
            10     0.10901737158006786    \\
            11     0.07743805715620637    \\
            12     0.055008872529725694    \\
            13     0.0390795538830058    \\
            14     0.027767781177568843    \\
            15     0.01973662660243989    \\
            16     0.014036695857953738    \\
            17     0.009993809474749024    \\
            18     0.007129432791347158    \\
            19     0.0051039236860827685    \\
            20     0.00367642340629369    \\
            21     0.0026760135328614694    \\
            22     0.0019813018670897975    \\
            23     0.0015053416386217597    \\
            24     0.0011849665039625453    \\
            25     0.0009730859125140265    \\
            26     0.0008343319747211917    \\
            27     0.0007426501088152762    \\
            28     0.0006800649867883443    \\
            29     0.0006349643661507712    \\
            30     0.0006003800745318971    \\
            31     0.0005722568337202765    \\
            32     0.0005483301820446472    \\
            33     0.0005273039542218782    \\
            34     0.0005084496833571379    \\
            35     0.0004913182726989556    \\
            36     0.00047563349321680533    \\
            37     0.00046119588167167915    \\
            38     0.00044786273424878725    \\
            39     0.0004355135887232344    \\
            40     0.00042405097623967746    \\
            41     0.0004133860774804207    \\
            42     0.0004034426270825716    \\
            43     0.00039414988119007734    \\
            44     0.0003854457659438473    \\
            45     0.00037727292614533046    \\
            46     0.00036958069517619487    \\
            47     0.00036232266329380225    \\
            48     0.000355457761318989    \\
            49     0.00034894867770518637    \\
            50     0.0003427623836714007    \\
            51     0.00033686906505517273    \\
            52     0.0003312423253048376    \\
            53     0.00032585844622736    \\
            54     0.00032069641839332626    \\
            55     0.00031573741949260056    \\
            56     0.0003109647606986271    \\
            57     0.00030636351280337337    \\
            58     0.00030192041840590314    \\
            59     0.0002976236205376815    \\
            60     0.0002934625672603977    \\
            61     0.0002894278126310959    \\
            62     0.0002855109267910912    \\
            63     0.0002817043487515243    \\
            64     0.0002780013075041107    \\
            65     0.0002743957123883616    \\
            66     0.000270882086726736    \\
            67     0.00026745548572364717    \\
            68     0.0002641114421223014    \\
            69     0.00026084590444070175    \\
            70     0.0002576551932707897    \\
            71     0.00025453595463689765    \\
            72     0.0002514851252900237    \\
            73     0.0002484998973686028    \\
            74     0.0002455776910739593    \\
            75     0.00024271612781484988    \\
            76     0.00023991300882070357    \\
            77     0.00023716629467530702    \\
            78     0.00023447408863960014    \\
            79     0.00023183462101015833    \\
            80     0.00022924623614186602    \\
            81     0.00022670738045518812    \\
            82     0.00022421659234502954    \\
            83     0.0002217724929558462    \\
            84     0.00021937377832833218    \\
            85     0.00021701921227455174    \\
            86     0.00021470762025472905    \\
            87     0.00021243788385102348    \\
            88     0.00021020893597669193    \\
            89     0.00020801975656651843    \\
            90     0.0002058693688128968    \\
            91     0.00020375683578404913    \\
            92     0.000201681257449506    \\
            93     0.00019964176800703053    \\
            94     0.0001976375335161499    \\
            95     0.00019566774976762102    \\
            96     0.00019373164038399455    \\
            97     0.00019182845510620983    \\
            98     0.00018995746825536172    \\
            99     0.00018811797733971853    \\
            100     0.00018630930179546558    \\    	
        };
    \end{semilogyaxis}
\end{tikzpicture}%